\documentclass[11pt]{amsart}

\usepackage{amsmath}
\usepackage{amsxtra}
\usepackage{mathrsfs}
\usepackage{bbm}
\usepackage{amscd}
\usepackage{amsthm}
\usepackage{amsfonts}
\usepackage{amssymb}
\usepackage{eucal}
\usepackage{xcolor}
\usepackage{url}
\usepackage[all]{xy}
\usepackage{tikz-cd}
\usepackage{enumitem}
\usepackage{xfrac}
\usepackage{hyperref}
\usetikzlibrary{shapes.geometric}

\usepackage{amsaddr}

\newtheorem{theorem}{Theorem}[section]
\newtheorem{cor}[theorem]{Corollary}
\newtheorem{lem}[theorem]{Lemma}
\newtheorem{prop}[theorem]{Proposition}

\newtheorem{thm}[theorem]{Theorem}
\newtheorem{conj}[theorem]{Conjecture}
\newtheorem{example}[theorem]{Example}
\newtheorem*{question}{Question}
\newenvironment{lemma*}
 {\pushQED{\qed}\lem}
 {\popQED\endlem}
\newenvironment{thm*}
 {\pushQED{\qed}\thm}
 {\popQED\endthm}

\theoremstyle{definition}
\newtheorem{defn}[theorem]{Definition}

\theoremstyle{remark}
\newtheorem{rem}[theorem]{Remark}

\numberwithin{equation}{section}
\makeatletter
\def\l@subsection{\@tocline{2}{0pt}{1.5pc}{4pc}{}}
\makeatother
\makeatletter
\def\thm@space@setup{%
  \thm@preskip=5.25pt plus 1pt minus 1pt
  \thm@postskip=3.25pt plus 1pt minus 1pt 
}
\makeatother

\setlist[itemize]{topsep = 0pt}

\DeclareMathOperator{\Mod}{Mod}
\DeclareMathOperator{\res}{res}
\DeclareMathOperator{\topp}{top}
\DeclareMathOperator{\rad}{rad}
\DeclareMathOperator{\Ima}{Im}

\begin{document}

\newcommand{\nc}{\newcommand}
\nc{\g}{\mathfrak{g}}
\nc{\OOs}{\mathcal{O}^{sh}_{\g}}
\nc{\uqmu}[2]{U_q^{#1}({#2})}
\nc{\C}{\mathbb{C}}
\nc{\arr}{\rightarrow}
\nc{\OO}{\mathcal{O}}
\nc{\Z}{\mathbb{Z}}
\nc{\SLt}{\mathfrak{sl}_3}
\nc{\SL}{\mathfrak{sl}_2}
\nc{\bi}{\bibitem}
\nc{\qbin}[3]{\genfrac{[}{]}{0pt}{}{#1}{#2}_{#3}}

\title[Inflations for representations of shifted quantum affine algebras]{Inflations for representations of \\ shifted quantum affine algebras}

\author{\vspace*{-5mm}Théo Pinet\textsuperscript{\,\lowercase{a,b}}}
\address{\vspace*{-2mm}\scriptsize{\textsuperscript{\normalfont a}\,Universit\'e Paris Cité and Sorbonne Universit\'e, CNRS, IMJ-PRG, F-75006, Paris, France.\\ \textsuperscript{\normalfont b}\,Universit{\'e} de Montr{\'e}al, DMS, Montr{\'e}al, Qu{\'e}bec, Canada, H3C 3J7.}}
\email{tpinet@imj-prg.fr}
\begin{abstract} 
\vspace*{0mm}
Fix a finite-dimensional simple Lie algebra $\mathfrak{g}$ and let $\mathfrak{g}_J\subseteq\mathfrak{g}$ be a Lie subalgebra coming from a Dynkin diagram inclusion. Then, the corresponding restriction functor is not essentially surjective on finite-dimensional simple $\mathfrak{g}_J$--modules. \par
In this article, we study Finkelberg--Tsymbaliuk's shifted quantum affine algebras $U_q^{\mu}(\mathfrak{g})$ and the associated categories $\mathcal{O}^{\mu}$ (defined by Hernandez). In particular, we introduce natural subalgebras $U_q^{\nu}(\mathfrak{g}_J)\,{\subseteq}\,U_q^{\mu}(\mathfrak{g})$ and obtain a functor $\mathcal{R}_J$ from $\mathcal{O}^{sh}\,{=}\bigoplus_{\mu}\mathcal{O}^{\mu}$ to $\bigoplus_{\nu}(U_q^{\nu}(\mathfrak{g}_J)\text{-Mod})$ using the canonical restriction functors. We then establish that $\mathcal{R}_J$ is essentially surjective on finite-dimensional simple objects by constructing notable preimages that we call \textit{inflations}. \par 
We conjecture that all simple objects in $\mathcal{O}^{sh}_J$ (which is the analog of $\mathcal{O}^{sh}$ for the subalgebras $U_q^{\nu}(\mathfrak{g}_J)$) admit some inflation and prove this for $\mathfrak{g}$ of type A--B or $\mathfrak{g}_J$ a direct sum of copies of $\mathfrak{sl}_2$ and $\mathfrak{sl}_3$. We finally apply our results to deduce certain \textit{$R$-matrices} and examples~of~\textit{cluster structures over Grothendieck rings}.\medskip\par\noindent
\textbf{Keywords:}
Shifted quantum affine algebras, category $\mathcal{O}$, $R$-matrices, cluster algebras.
\end{abstract}
\maketitle
\thispagestyle{empty}
\vspace*{-3.25mm}
{\scriptsize
\setlength{\parskip}{0in}
\tableofcontents}\vspace*{-7mm}
\section{Introduction}\label{sec:Intro}
It is well known that the only finite-dimensional complex simple Lie algebra having a simple representation of dimension 2 is $\SL$ (up to isomorphism). In a similar way, an untwisted affine Lie algebra $\hat{\g}$ (resp.~quantum group $\uqmu{}{\g}$, untwisted quantum loop algebra $\uqmu{}{\hat{\g}}$) associated to a finite-dimensional simple Lie algebra $\g$ can have $2$-dimensional irreducible representations only if the underlying Lie algebra $\g$ is isomorphic to $\mathfrak{sl}_2$ (see Proposition \ref{prop:AffNoInfl}). The restriction functors constructed using the inclusions of $\SL$ in $\g$ (or of $\widehat{\mathfrak{sl}}_2$ in $\hat{\g}$, etc.) are thus not essentially surjective on finite-dimensional simple modules when $\g\not\simeq \SL$. \par
Shifted quantum affine algebras $\uqmu{\mu}{\g}$ are infinite-dimensional algebras parametrized by a Lie algebra $\g$ like above and a coweight $\mu$ of $\g$. These algebras are variations of the untwisted quantum loop algebras $\uqmu{}{\hat{\g}}$ (in \textit{Drinfeld's presentation}) which are in turn algebras of critical importance in the study of quantum integrable systems \cite{fh1,fjmm,hTsyst,mrv1,mrv2}, of Nakajima quiver varieties \cite{fuj,mo,n1,n2,n3} and of cluster algebras \cite{bi,hl1,hl2,kkop}. The algebras $\uqmu{\mu}{\g}$ were introduced by \cite{ft1} in the context of $K$-theoretic Coulomb branches of $3d$ $\mathcal{N}=4$ SUSY quiver gauge theories and have since then played a primordial role in many areas of research \cite{fi,n4} (see also \cite{bfn,fkprw} for the related case of shifted Yangians).
 \par
Denote by $\uqmu{\mu}{\g}$--Mod the category of all (potentially infinite-dimensional) representations of the shifted quantum affine algebra $\uqmu{\mu}{\g}$. Then, $\uqmu{\mu}{\g}$--Mod admits a notable subcategory $\OO^{\mu}$ which was introduced in \cite{hshift} as an adaptation of the standard BGG category $\mathcal{O}$ of usual Lie theory. This category $\OO^{\mu}$ contains all finite-dimensional $\uqmu{\mu}{\g}$-modules as well as remarkable infinite-dimensional ones. For instance, the category $\OO^{\mu}$ with $\mu$ the opposite of a fundamental coweight contains infinite-dimensional objects, called \textit{negative\,prefundamental\,representations}, which were introduced in \cite{hj} (for the Borel subalgebra $U_q(\hat{\mathfrak{b}})$ of $\uqmu{}{\hat{\g}}$
) and were proven to~be deeply related to the famous Baxter $Q$-operators of integrable systems in \cite{fh1}.\par
Consider the category $\OO^{sh} = \bigoplus_{\mu} \OO^{\mu}$ where the sum runs over the coweight lattice $\Lambda^{\vee}$ of $\g$. By~\cite{hAff,hfus,hshift}, this category is endowed with a distinguished operation $\star$, the \textit{fusion~product}, which gives to the Grothendieck group $K_0(\OO^{sh})$ the structure of a ring. Moreover, \cite{hshift} proves that \textit{Frenkel--Reshitikhin's $q$-character theory} \cite{fr}, which is one of the main tools used in the representation theory of quantum loop algebras, can be adjusted to the study of $\OO^{sh}$ in such a way that
\textit{the $q$-character of a fusion product of simple objects is precisely the product of the $q$-characters of the underlying simple objects} (in some ring of formal series, see Section \ref{sec:Fusion}). 
This allowed \cite{hshift} to show that the Grothendieck ring $K_0(\OO^{sh})$ is in fact commutative (by the same proof as for the case of Grothendieck rings linked to untwisted quantum loop algebras).
\par
Now, fix a subalgebra $\g_J\subseteq \g$ coming from an inclusion of Dynkin diagrams $J\subseteq I$ (with $I$ the diagram of $\g$) and choose a coweight $\mu\in \Lambda^{\vee}$. Use the generating set 
$$\{x_{i,r}^{\pm},\phi_{i,r}^{\pm}\,|\,i\in I,r\in \Z\}\subseteq \uqmu{\mu}{\g}$$
 of Section \ref{sec:ShiftedqAff} and consider the subalgebra generated by $\{x_{j,r}^{\pm},\phi_{j,r}^{\pm}\,|\,j\in J,r\in \Z\}$ (inside $\uqmu{\mu}{\g}$). Then, this subalgebra only depends on the projection $\nu=\res_J(\mu)$ of $\mu$ on the coweight lattice $\Lambda_J^{\vee}$ of $\g_J$ and will be denoted $\uqmu{\nu}{\g_J}$.
  We will also note $\iota_J^{\mu}$ the inclusion $\uqmu{\nu}{\g_J}\subseteq \uqmu{\mu}{\g}$ and
$$\res_J^{\mu}=(\iota_J^{\mu})^*: \uqmu{\mu}{\g}\text{--Mod}\arr \uqmu{\nu}{\g_J}\text{--Mod}$$ the associated restriction functor. The family $\{\res_J^{\mu}\}_{\mu\in \Lambda^{\vee}}$ induces a canonical functor
$$ \textstyle \mathscr{R}_J : \OO^{sh} \arr \bigoplus_{\nu\in \Lambda^{\vee}_J} \uqmu{\nu}{\g_J}\text{--Mod}$$
which is the main object of study of this paper. A first question to ask about this functor is:
\begin{question} Is $\mathscr{R}_J$ essentially surjective on finite-dimensional simple modules?
\end{question}
In other terms, given some coweight $\nu$ of $\g_J$ and a finite-dimensional simple $\uqmu{\nu}{\g_J}$-module $W$, the above question asks if it is always possible to find a pair $(\mu,V)$ with 
\begin{itemize}
\setlength{\itemsep}{1.5pt}
\item[(1)] $\mu$ a coweight of $\g$ with $\res_J(\mu)=\nu$ and 
\item[(2)] $V$ a $\uqmu{\mu}{\g}$-module for which $\res_J^{\mu}(V)\simeq W$ (as $\uqmu{\nu}{\g_J}$-modules).
\end{itemize}
A heuristic justification for the plausibility of this question comes from the (partly conjectural) relationship between shifted quantum affine algebras and Coulomb branches (see Section \ref{sec:Unicity}). However, our initial motivation for tackling this problem is more naive and comes simply from the fact that, unlike the Lie algebra $\g$ (or $\hat{\g}$, $\uqmu{}{\g}$ or $\uqmu{}{\hat{\g}}$), the shifted quantum affine algebra $\uqmu{\mu}{\g}$ can have $2$-dimensional irreducible representations even if $\g\not\simeq \SL$ by\,\cite[Example 6.6]{hshift}. Furthermore, the $2$-dimensional irreducible objects $V$ of $\OO^{sh}$ described in \cite{hshift} verify 
$x_{i,r}^{\pm}V = 0$ for every $r\in\Z$ and all but one $i\in I$. This brings to the following definition (where we denote by $\OO^{\nu}_J$ the category defined analogously to $\OO^{\nu}$, but for the subalgebra $\uqmu{\nu}{\g_J}$).
\begin{defn}[Definition \ref{def:Infl}]\label{def:InflIntro}
Fix $\nu\in \Lambda^{\vee}_J$ with $W$ an object of $\OO^{\nu}_J$. Take moreover $\mu\in \Lambda^{\vee}$ satisfying $\res_J(\mu)=\nu$ with $V$ in $\OO^{\mu}$. Then, $V$ is a \textit{$J$-inflation of $W$ to $\g$ with coweight $\mu$} if 
\begin{itemize}
\setlength{\itemsep}{1.5pt}
\item[(i)] $\res_J^{\mu}(V)\simeq W$ in $\OO^{\nu}_J$ and
\item[(ii)] $x_{i,r}^{\pm}V = 0$ for all $i\not\in J$ and $r\in \Z$.
\end{itemize}
\end{defn}
Our first main result in this article is:
\begin{thm}[Theorem \ref{cor:ABFin}, part (i)]\label{thm:main1} Take $\nu\in \Lambda^{\vee}_J$ with $W$ a finite-dimensional irreducible $\uqmu{\nu}{\g_J}$-module. Then, $W$ admits a $J$-inflation to $\g$. 
\end{thm}
Note that any $J$-inflation of $W$ is in particular a preimage of $W$ for the functor $\mathscr{R}_J$. Hence, Theorem \ref{thm:main1} gives a positive answer to the previous question. We can however extend this first question by asking whether $\mathscr{R}_J$ is essentially surjective on all (potentially infinite-dimensional) simple objects of the category $\OO^{sh}_J = \bigoplus_{\nu\in \Lambda^{\vee}_J}\OO^{\nu}_J$. \par 
This follow-up question is partially answered by our second main result, which is:
\begin{thm}[Corollary \ref{cor:ExistenceSL} and Theorem \ref{cor:ABFin}, part (ii)]\label{thm:main2} Take $W$ a simple module in $\OO^{sh}_J$. Then, $W$ admits a $J$-inflation to $\g$ if either 
\begin{itemize}
\setlength{\itemsep}{1.5pt}
\item[(i)] $\g$ is of type A--B or if
\item[(ii)] $\g_J$ is isomorphic (as a Lie algebra) to a direct sum of copies of $\SL$ and $\SLt$.
\end{itemize}
\end{thm}
The only obstruction preventing the generalization of Theorem \ref{thm:main2} to every Lie algebra~$\g$~is the
lack of proof, for $\g$ not of type A--B, for a technical result on the multiplicities appearing~in the $q$-characters~of negative prefundamental representations (see Lemma \ref{lem:multPsiboundedPrefondAB} and the remarks made after Theorem \ref{cor:ABFin}). \par It is thus natural to make the following conjecture (see Conjecture \ref{conj:multPsiboundedPrefond} and Corollary \ref{thm:ExistenceIfConj}):
\begin{conj}\label{conj:main}
All simple objects of $\OO^{sh}_J$ admit $J$-inflations to $\g$.
\end{conj}
One of the main ingredients of the proofs of Theorem \ref{thm:main1} and Theorem \ref{thm:main2} is the compatibility between Definition \ref{def:InflIntro} and the fusion product $\star$ of $\OO^{sh}$ (and $\OO^{sh}_J$). More precisely, given simple objects $W_1,\dots,W_r$ of $\OO^{sh}_J$ and respective inflations $V_1,\dots,V_r$ to $\g$, then Corollary \ref{cor:FusInfl} shows that $V_1\star \dots \star V_r$ is a $J$-inflation to $\g$ of $W_1\star \dots\star W_r$. In addition, given an inflation $V$
for a \textit{highest $\ell$-weight object} $W$ of $\OO^{sh}_J$ (cf.~Definition \ref{def:HighestlWeight} or let $W=W_1\star \dots \star W_r$ as above), then Proposition \ref{prop:InfHighest} shows that the radical $\rad V$ and the head $\topp V$ are respectively $J$-inflations of $\rad W$ and $\topp W$. This (and various results of \cite{hshift}) in turn implies that it suffices to prove
Theorem \ref{thm:main2} for $W$ an arbitrary negative prefundamental representation of $\OO^{sh}_J$. This is done explicitly in Example \ref{ex:InflPrefondSL}, Example \ref{ex:InflPrefondSLt} and Example \ref{rem:InflPrefondSS} assuming that $\g_J$ is isomorphic, as a Lie algebra, to a direct sum of (potentially many) copies of $\SL$ and $\SLt$. For the more subtle case of $\g$ of type A--B, the proof relies on a detailed study of $q$-characters and follows roughly the reasoning used in the proof of \cite[Proposition 6.9]{fjmm}. Our proof in type A--B~moreover shows Theorem \ref{thm:main1} for $\g$ of arbitrary type (and Conjecture \ref{conj:main} up to~some~technical~unproven lemma in type C--D--E--F--G). \par
Let us now make more precise the connection between our results and the work of \cite{fjmm} (see also the end of Section \ref{sec:Inflqchar} for a detailed explanation). This convoluted connection comes from one of our results, Theorem \ref{thm:EquivalentDefInf}, which gives an equivalent definition of inflations ---~for highest $\ell$-weight or irreducible modules in $\OO^{sh}_J$ --- that relies on $q$-characters and resembles~the definition of some representations described in \cite[Proposition 6.9]{fjmm}~(for~the~Borel~subalgebra~$\uqmu{}{\hat{\mathfrak{b}}}$ of the quantum loop algebra $\uqmu{}{\hat{\mathfrak{g}}}$).
We stress that the analogy~between~inflations and the representations of \cite{fjmm} is far from being perfect --- even after considering the link between the representation theories of Borel quantum loop algebras and of shifted quantum affine algebras (see Section \ref{sec:Inflqchar} where we discuss results of \cite{hshift}) --- but nevertheless enables~to draw inspiration from \cite[Proposition 6.9]{fjmm} for the proofs of Theorems \ref{thm:main1}--\ref{thm:main2} (after some non-trivial adaptations; see Section \ref{sec:InflPrefund}).  \par 
The first and foremost benefit of using $J$-inflations is their ability to reduce problems about objects and morphisms~of $\OO^{sh}$ to (typically easier) problems about objects and morphisms of $\OO^{sh}_J$. Inflations thus generate a framework within the category $\OO^{sh}$ in which a lot of interesting objects or quantities (i.e.~$q$-characters, fusion products, etc.) can be described as for $\OO^{sh}_J$.~This framework can then be used to deduce new information about $\OO^{sh}$ or to shed a new light over already studied \textit{parts} of this category\,(by viewing these as \textit{inflations} of corresponding\,\textit{parts}\,for the category $\OO^{sh}_J$). Examples of application of this last idea arise in the study of ``\textit{compatible cluster structures}" on subrings of $K_0(\OO^{sh})$ and of \textit{particular $R$-matrices} in $\OO^{sh}$ (see Section~\ref{sec:applications} for details about these applications and terminology). Another, maybe more straightforward, potential application of inflations (which is connected to the above study of cluster structures) is the deduction of relations for the Grothendieck ring $K_0(\OO^{sh})$. We study this application in more details in Section \ref{sec:Inflqchar} (see also Proposition \ref{prop:InfHighest} and Theorem \ref{thm:InfHighestDual}) and show notably\,(in Examples \ref{ex:QQ}--\ref{ex:QQ2}) that the $Q\widetilde{Q}$ and $QQ^{*}$-systems of \cite{fh2,hl1} may be interpreted as\,being ``\textit{inflations}".\,We also prove an ``\textit{inflated version}" of the $T$-system of type A${}_1$ in Example~\ref{ex:newT} (which seems to be a new relation for the ring $K_0(\OO^{sh})$).
\par
The article is divided as follows. In Section \ref{sec:ShiftedqAff}, we recall the definition of the shifted quantum affine algebra $\uqmu{\mu}{\g}$ and describe important tools for the study of the associated category $\OO^{\mu}$ (such as the fusion product and the $q$-character map of $\OO^{sh}$).~We also initiate the study~of~the subalgebras $\uqmu{\nu}{\g_J}$ (with $\nu$ a coweight of $\g_J\subseteq \g$) and of the associated restriction functors. Then, in Section \ref{sec:Infl}, we define inflations and study their properties relatively to fusion products, radicals (of modules) and $q$-characters. We also prove our main results and give many explicit examples with, in particular, our \textit{inflated version} of the $T$-system in type A${}_1$. Finally, we~end the paper in Section \ref{sec:Concl} with results about \textit{strongly/weakly minimal inflations} and a description of possible applications of inflations to the study of ``\textit{compatible cluster~structures}"~on~subrings of $K_0(\OO^{sh})$ and to the construction of \textit{particular $R$-matrices} in $\OO^{sh}$.
\addtocontents{toc}{\setcounter{tocdepth}{-10}}
\subsection*{Acknowledgements} We are deeply grateful to David Hernandez, whose guidance and generosity were instrumental in bringing this work to life. Our most sincere thanks also go to Joel Kamnitzer for his input on the possible connection between inflations and Coulomb branches, and to Vyjayanthi Chari for her kind invitation to Riverside, where most of the early research leading to this paper was made. We finally extend our thanks to both Alexis Leroux-Lapierre and Yvan Saint-Aubin for many useful discussions. This work was supported by a scholarship from the Fonds de recherche du Québec--Nature et Technologie and has received funding from the European Union’s Horizon 2020 research and innovation programme \includegraphics[scale=0.07]{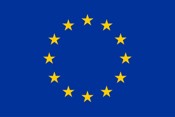} under the Marie Skłodowska-Curie grant agreement No 945332. This support is gratefully acknowledged.
\addtocontents{toc}{\setcounter{tocdepth}{2}}
\newpage
\section{Representation theory of shifted quantum affine algebras}\label{sec:ShiftedqAff}
Fix a finite-dimensional simple Lie algebra $\g$ and set $I=\{1,\dots,n\}$ with $n$ the rank~of~$\g$.~Let also $C=(C_{i,j})_{i,j\in I}$ be the indecomposable Cartan matrix of finite type associated to $\g$. Then, there exists a single diagonal matrix $D$ with relatively prime diagonal entries $d_1,\dots,d_n\in \mathbb{N}_{>0}$ for which the matrix $DC$ is symmetric. Denote by $\{\alpha_i\}_{i\in I}$, $\{\alpha_i^{\vee}\}_{i\in I}$, $\{\omega_i\}_{i\in I}$ and $\{\omega_i^{\vee}\}_{i\in I}$ the simple roots, simples coroots, fundamental weights and fundamental coweights of $\g$ (respectively). Write moreover $Q = \bigoplus_{i\in I} \Z \alpha_i$, $Q^+ = \bigoplus_{i\in I} \mathbb{N} \alpha_i$, $\Lambda=\bigoplus_{i\in I}\Z\omega_i$ and $\Lambda^{\vee} = \bigoplus_{i\in I}\Z\omega_i^{\vee}$. Then, $\Lambda_{\mathbb{Q}} = \mathbb{Q}\otimes_{\Z}\Lambda$ can be partially ordered under the rule $\omega\leq \omega'$ if and only if $\omega'-\omega\in Q^+$. We denote by $\Lambda^{\vee}_+=\bigoplus_{i\in I}\mathbb{N}\omega_i^{\vee}$ the set of \textit{dominant coweights} of $\g$. \par
Fix now a parameter $q\in \mathbb{C}^{\times}$ which is not a root of unity and let $q_i=q^{d_i}$ for all $i\in I$. For $x\in \C^{\times}$ not a root of unity and $m,p\in \mathbb{N}$ with $m\geq p$, we will use the notation
$$ [m]_x = \frac{x^m-x^{-m}}{x-x^{-1}},\, [m]_x! = \prod_{r=1}^m [r]_x\text{ and } \qbin{m}{p}{x} = \frac{[m]_x!}{[p]_x![m-p]_x!}$$
for the usual $q$-numbers, $q$-factorials and $q$-binomial coefficients (respectively).\par
This section introduces important objects and recalls notable results relatively to the representation theory of shifted quantum affine algebras. We refer the interested reader to \cite{ft1,hshift} for details and to \cite{fh1,fjmm,ghl,hbourb,her,mo} for applications of the study~of~shifted~and unshifted quantum affine algebras. In this article, $\mathbb{N}=\mathbb{Z}_{\geq 0}$ and all vector spaces, algebras~and tensor products are defined over $\C$ (unless otherwise specified).
\vspace*{-3mm}
\subsection{Shifted quantum affine algebras}\label{sec:Def} We start with the definition of the shifted quantum affine algebra $\uqmu{\mu}{\g}$ following \cite{ft1} and their study of \textit{quantized K-theoretic Coulomb branches of $3d$ $\mathcal{N}=4$ SUSY quiver gauge theories}.
\begin{defn}[{\cite[Section 5]{ft1}}, see also {\cite[Section 3.1]{hshift}}]\label{def:Uqmu} For $\mu\in \Lambda^{\vee}$, the shifted quantum affine algebra $\uqmu{\mu}{\g}$ is the unital algebra generated by $\{x_{i,r}^{\pm},\phi_{i,r}^{\pm}\,|\,i\in I\text{ and }r\in \Z\}$ with both $\phi_{i,0}^+$ and $\phi_{i,\alpha_i(\mu)}^-$ invertible for $i\in I$ and with the relations (for $i,j\in I$, $r,s\in \Z$ and $m\in\Z\backslash\{0\}$)
\begin{gather}
[\phi_{i,r}^{\pm},\phi_{j,s}^{\pm}]=[\phi_{i,r}^{\pm},\phi_{j,s}^{\mp}]=0, \label{eq:CommPhi} \\
\phi_{i,0}^+x_{j,r}^{\pm} = q_i^{\pm C_{i,j}}x_{j,r}^{\pm}\phi_{i,0}^+ \text{ and } \phi_{i,\alpha_i(\mu)}^-x_{j,r}^{\pm} = q_i^{\mp C_{i,j}}x_{j,r}^{\pm}\phi_{i,\alpha_i(\mu)}^-,\label{eq:PhiTX}\\
[h_{i,m},x_{j,r}^{\pm}] = {\textstyle \pm\frac{1}{m}}[mC_{i,j}]_{q_i}x_{j,m+r}^{\pm},\label{eq:Relhx}\\
(q_i-q_i^{-1})[x_{i,r}^+,x_{j,s}^-]=\delta_{i,j}(\phi_{i,r+s}^+-\phi_{j,r+s}^-),\label{eq:Relxpxmphi}\\
x_{i,r+1}^{\pm}x_{j,s}^{\pm}-q_i^{\pm C_{i,j}}x_{j,s}^{\pm}x_{i,r+1}^{\pm}=q_i^{\pm C_{i,j}}x_{i,r}^{\pm}x_{j,s+1}^{\pm}-x_{j,s+1}^{\pm}x_{i,r}^{\pm}\label{eq:xpmRelSupp}
\end{gather}
as well as, for $p=1-C_{i,j}$, $i\neq j$ and $r',r_1,\dots,r_p\in \Z$, 
\begin{equation}\label{eq:qSerre}
\sum_{\pi\in \Sigma_p}\sum_{0\leq \ell\leq p}(-1)^{\ell}\qbin{p}{\ell}{q_i}x_{i,r_{\pi(1)}}^{\pm}\dots x_{i,r_{\pi(\ell)}}^{\pm}x_{j,r'}^{\pm}x_{i,r_{\pi(\ell+1)}}^{\pm}\dots x_{i,r_{\pi(p)}}^{\pm}=0
\end{equation}
(with $\Sigma_p$ the symmetric group on $p$ letters). The elements $h_{i,m}\in\uqmu{\mu}{\g}$ (for $i\in I$, $m\in \Z\backslash\{0\}$) appearing in \eqref{eq:Relhx} are defined implicitly through the equalities of formal series (in $z$)
\begin{equation}\label{eq:RelhPhi}
\sum_{r\in \Z}\phi_{i,\pm r}^{\pm}z^{\pm r} = \phi_{i}^{\pm}(z) = z^{\mu(i,\pm)}\phi_{i,\mu(i,\pm)}^{\pm}\exp\Bigg(\pm(q_i-q_i^{-1})\sum_{m>0}h_{i,\pm m}z^{\pm m}\Bigg)
\end{equation}
where $\mu(i,+) = 0$ and $\mu(i,-)=\alpha_i(\mu)$.
\end{defn}
\begin{rem} An equivalent presentation of $\uqmu{\mu}{\g}$ is given in \cite[Section 3.1]{hshift} using the formal Laurent power series (in $z$) or \textit{currents} $$\textstyle x_i^{\pm}(z) = \sum_{r\in \Z}x_{i,r}^{\pm}z^r\text{ and }\phi_i^{\pm}(z)=\sum_{r\in \Z}\phi_{i,\pm r}^{\pm}z^{\pm r}$$ with $i\in I$. Note that, by \eqref{eq:RelhPhi}, $\phi_{i,r}^{+} = 0$ if $r<0$ and $\phi_{i,r}^- = 0$ if $r>\alpha_i(\mu)$ for all $i\in I$ so that $$\textstyle \phi_i^{+}(z) = \sum_{r\geq 0}\phi_{i,r}^+z^r\text{ and }z^{-\alpha_i(\mu)}\phi^{-}_i(z) = \sum_{r\geq 0}\phi_{i,\alpha_i(\mu)-r}^{-}z^{-r}$$ 
are actually proper power series in $z$ and $z^{-1}$ (respectively).
\end{rem}
The algebras $\uqmu{\mu}{\g}$ have a $\Z$-grading given via $\deg(x_{i,r}^{\pm}) = \deg(\phi_{i,r}^{\pm}) = r$ and $\deg(h_{i,m}) = m$ for $i\in I$, $r\in \Z$ and $m\in\Z\backslash\{0\}$. Using this grading, one defines, for each $a\in \C^{\times}$, an algebra automorphism $\tau_a$ of $\uqmu{\mu}{\g}$ by $\tau_a(x) = a^k x$
whenever $x\in \uqmu{\mu}{\g}$ is homogeneous of degree $k$. This allows to deform any representation $V$ of $\uqmu{\mu}{\g}$ into a $1$-parameter family $\{V(a)\}_{a\in \C^{\times}}$ of $\uqmu{\mu}{\g}$-modules (where $V(a)$ is the pullback of $V$ by the automorphism $\tau_a$).\par
Another important fact about $\uqmu{\mu}{\g}$ is that multiplication induces a vector space isomorphism (the so-called \textit{triangular decomposition}) \cite[Proposition 5.1(a)]{ft1} (see also \cite{hAff})
\begin{equation}\label{eq:TriangularDec}
\uqmu{\mu}{\g}\simeq \uqmu{\mu,-}{\g}\otimes \uqmu{\mu,0}{\g}\otimes \uqmu{\mu,+}{\g}
\end{equation}
with $\uqmu{\mu,\pm}{\g}$ and $\uqmu{\mu,0}{\g}$ respectively the subalgebras of $\uqmu{\mu}{\g}$ generated by $$\{x_{i,r}^{\pm}\,|\,i\in I,r\in \Z\}\text{ and }\{\phi_{i,r}^+,\phi_{i,r}^-,(\phi_{i,0}^+)^{-1},(\phi_{i,\alpha_i(\mu)}^-)^{-1}\,|\,i\in I,r\in \Z\}.$$ 
Note that the subalgebras $\uqmu{\mu,\pm}{\g}$ can be presented using their generating set above with~the defining relations \eqref{eq:xpmRelSupp}--\eqref{eq:qSerre}. In the same way, the subalgebra $\uqmu{\mu,0}{\g}$, called \textit{Cartan--Drinfeld subalgebra} of $\uqmu{\mu}{\g}$, is isomorphic to the commutative algebra with generators as above \cite{hAff}.
\begin{rem}\label{rem:borQaff} The shifted quantum affine algebra $\uqmu{0}{\g}$ with shift $\mu = 0$ is a central extension of the usual \textit{quantum loop algebra} $\uqmu{}{\hat{\g}}$  \cite[Remark 5.4(b)]{ft1}. It is well known (see, e.g.,\,\cite{CP2}) that the latter algebra $\uqmu{}{\hat{\g}}$ is generated by a finite set $\{e_i,f_i,k_i^{\pm 1}\}_{i\in I\cup\{0\}}$ of \textit{Drinfeld--Jimbo generators} and that this set allows one to endow $\uqmu{}{\hat{\g}}$ with the structure of a Hopf algebra. However, no coproduct has yet been defined for general shifted quantum affine algebras $\uqmu{\mu}{\g}$ (except if $\g=\mathfrak{sl}_{n+1}$ where \cite[Section 10]{ft1} gives such a definition). An alternative operation, called \textit{fusion product}, nevertheless enables the recovery of part of the advantages of having a Hopf structure \cite[Section 5.3]{hshift}. More informations about this operation are given in Section \ref{sec:Fusion}. In the meantime, let us emphasize, for further purposes, that the shifted algebra $\uqmu{\mu}{\g}$ for $\mu\in-\Lambda_+^{\vee}$ contains, by \cite[Proposition 3.4]{hshift}, a copy of the \textit{Borel quantum loop subalgebra} $\uqmu{}{\hat{\mathfrak{b}}}$ of $\uqmu{}{\hat{\g}}$. There is hence a natural connection between the representations of Borel~quantum loop algebras and antidominantly shifted quantum affine algebras. This connection moreover partially remains in general as $\uqmu{}{\hat{\mathfrak{b}}}$ still contains, for all $\mu\in\Lambda^{\vee}$, a copy of the subalgebra of $\uqmu{\mu}{\g}$ generated by 
$\{x_{i,r}^+,x_{i,s}^-,\phi_{i,r}^+,(\phi_{i,0}^+)^{-1}\,|\,i\in I,r\in\Z, s>\max(0,\alpha_i(\mu))\}$ \cite[Section 3.4]{hshift}. 
\end{rem}
For $J\subseteq I$, denote by $\g_J$ the \textit{semisimple} (but typically \textit{non-simple}) Lie algebra with Cartan matrix $(C_{i,j})_{i,j \in J}$.~Set $$\textstyle Q_J = \bigoplus_{j\in J}\Z\alpha_j\text{, }Q_J^{+} = \bigoplus_{j\in J} \mathbb{N}\alpha_j\text{, }\Lambda_J = \bigoplus_{j\in J}\Z\omega_j\text{ and }\Lambda_J^{\vee} = \bigoplus_{j\in J}\Z\omega_j^{\vee}.$$ 
Fix $\mu\in \Lambda^{\vee}$ and denote by $\nu=\res_J(\mu)$ the projection of $\mu$ onto $\Lambda_{J}^{\vee}$. Denote also by $\uqmu{\nu}{\g_J}$~the algebra having generators $\{x_{j,r}^{\pm},\phi_{j,r}^{\pm}\,|\,j\in J,r\in \Z\}$ and defining relations \eqref{eq:CommPhi}--\eqref{eq:RelhPhi} (where~we replace $I$ by $J$ everywhere and with $\phi_{j,0}^+$ and $\phi_{j,\alpha_j(\mu)}^{-}$ invertible for each $j\in J$). We understand $\uqmu{\nu}{\g_J}$ as being a shifted quantum affine algebra with underlying \textit{semisimple} Lie algebra $\g_J$. 
\par By the reasoning used in \cite[Sections 3--5]{hAff}, there is a triangular decomposition
\begin{equation}\label{eq:TriangularDecJ}
\uqmu{\nu}{\g_J} \simeq \uqmu{\nu,-}{\g_J}\otimes \uqmu{\nu,0}{\g_J}\otimes\uqmu{\nu,+}{\g_J}
\end{equation} 
where $\uqmu{\nu,\pm}{\g_J}$ and $\uqmu{\nu,0}{\g_J}$ are respectively the subalgebras of $\uqmu{\nu}{\g_J}$ generated by
$$ \{x_{j,r}^{\pm}\,|\,j\in J,r\in\Z\}\text{ and }\{\phi_{j,r}^+,\phi_{j,r}^-,(\phi_{j,0}^+)^{-1},(\phi_{j,\alpha_j(\mu)}^-)^{-1}\,|\,j\in J,r\in\Z\}.$$ \vspace*{-3.95mm}\\
Again, $\uqmu{\nu,\pm}{\g_J}$ can be presented using the generators above and defining relations \eqref{eq:xpmRelSupp}--\eqref{eq:qSerre} (with $J$ instead of $I$) whereas $\uqmu{\nu,0}{\g_J}$ is simply the commutative algebra with generating~set
$$\{\phi_{j,r}^+,\phi_{j,r}^-,(\phi_{j,0}^+)^{-1},(\phi_{j,\alpha_j(\mu)}^-)^{-1}\,|\,j\in J,r\in\Z\}.$$ 
There is a canonical algebra morphism $\iota_J^{\mu}:\uqmu{\nu}{\g_J}\arr \uqmu{\mu}{\g}$ given by identifying generators. 
\begin{prop} The morphism $\iota_J^{\mu}$ is injective. 
\end{prop}
\vspace*{-2mm}
\begin{proof}
Consider the maps $\pi_{\pm}:\uqmu{\mu,\pm}{\g}\arr\uqmu{\nu,\pm}{\g_J}$ given, for all $r\in\Z$, by $$\pi_{\pm}(x_{j,r}^{\pm}) =\left\{\begin{array}{ll} x_{j,r}^{\pm} & \text{if }j\in J,\\
0& \text{else}.
\end{array}\right.$$ 
These are well-defined surjective algebra morphisms by \eqref{eq:xpmRelSupp}--\eqref{eq:qSerre}. Also, $ \pi\circ \iota_J^{\mu}(x_{j,r}^{\pm}) = x_{j,r}^{\pm}$~for all $(j,r)\in J\times \Z$ and it follows that $\iota_J^{\mu}$ restricts to injective algebra morphisms from $\uqmu{\nu,+}{\g_J}$~to $\uqmu{\mu,+}{\g}$ and from $\uqmu{\nu,-}{\g_J}$~to $\uqmu{\mu,-}{\g}$. In addition, $\iota_J^{\mu}$ clearly gives an injective~morphism~of commutative algebras from $\uqmu{\nu,0}{\g_J}$ to $\uqmu{\mu,0}{\g}$ and must hence be injective (as a $\C$-linear~map from $\uqmu{\nu}{\g_J}$ to $\uqmu{\mu}{\g}$) since it is injective on each factor of the $\C$-linear isomorphism \eqref{eq:TriangularDecJ}.
\end{proof}\vspace*{-1mm}
We associate from now on $\uqmu{\nu}{\g_J}$ with the subalgebra $\Ima \iota_J^{\mu}\subseteq\uqmu{\mu}{\g}$ and denote~the~corresponding restriction functor by 
$$ \res_J^{\mu} = (\iota_J^{\mu})^* : \uqmu{\mu}{\g}\text{--Mod}\arr\uqmu{\nu}{\g_J}\text{--Mod}$$
with $\uqmu{\mu}{\g}\text{--Mod}$ the category of all left modules over $\uqmu{\mu}{\g}$ (and similarly for $\uqmu{\nu}{\g_J}$--Mod). Clearly, $\uqmu{\nu}{\g_J}$ inherits the $\Z$-grading of $\uqmu{\mu}{\g}$ and the automorphisms $\{\tau_a\}_{a\in \C^{\times}}$. Moreover, one can always see $\uqmu{\nu}{\g_J}$ as a tensor product of shifted quantum affine algebras with simple underlying Lie algebras using the following easily proven lemma.
\begin{lemma*}\label{lem:UqJnonconnected} Take $J\subseteq I$ and fix $J'\subseteq J$ such that $C_{j,j'} = 0$ for all $j\in J\backslash J'$ and $j'\in J'$. Let $\mu\in \Lambda^{\vee}$. Then multiplication gives an algebra automorphism 
$$ \uqmu{\res_J(\mu)}{\g_J}\simeq \uqmu{\res_{J'}(\mu)}{\g_{J'}}\otimes \uqmu{\res_{J\backslash J'}(\mu)}{\g_{J\backslash J'}}$$
with $\res_J(\mu)$, $\res_{J'}(\mu)$ and $\res_{J\backslash J'}(\mu)$
the respective projections of $\mu$ on $\Lambda_J^{\vee}$, $\Lambda_{J'}^{\vee}$ and $\Lambda_{J\backslash J'}^{\vee}$.
\end{lemma*}
This article originated from the following question: 
\begin{question} Fix $\nu\in \Lambda_J^{\vee}$ with a $\uqmu{\nu}{\g_J}$-module $W$. Can we extend the $\uqmu{\nu}{\g_J}$-action on $W$ to a well-defined $\uqmu{\mu}{\g}$-action for some $\mu\in \Lambda^{\vee}$ with $\res_J(\mu)=\nu$? In other terms, can we find $\mu\in \Lambda^{\vee}$ satisfying $\res_J(\mu)=\nu$ so that there exists a preimage of $W$ under the functor $\res_{J}^{\mu}$?
\end{question}\vspace*{-0.35mm}
The above question is clearly too general to be answered positively. Indeed, almost nothing is known about the (gigantic) category $\uqmu{\mu}{\g}$--Mod and even the corresponding question, but restricted to the more accessible subcategory $\OO^{\mu}\subseteq\uqmu{\mu}{\g}$--Mod of \cite[Section 4]{hshift}, seems too complicated to answer for general objects $W$. A more reasonable question to ask is: \vspace*{-0.15mm}
\begin{question} Take $\nu\in \Lambda_J^{\vee}$ with a simple module $W$ of the category $\mathcal{O}^{\nu}_{J}$ (see Section \ref{sec:O}). Can we find $\mu\in \Lambda^{\vee}$ with $\res_J(\mu)=\nu$ and $V$ in $\mathcal{O}^{\mu}=\OO^{\mu}_I$ such that $\res_{J}^{\mu}(V)\simeq W$ as $\uqmu{\nu}{\g_J}$-modules?
\end{question}\newpage
Answering partially to this last question is the main goal of Section \ref{sec:Infl}. Meanwhile, the next subsection recalls the definition of the category $\OO^{\mu}$ as well as some important related results.
\begin{rem}\label{rem:qAffRestnotdense} As stated in Section \ref{sec:Intro}, a simple finite-dimensional Lie algebra has 2-dimensional irreducible representations if and only if it is isomorphic to $\SL$. Similarly, an untwisted affine Lie algebra $\hat{\g}$ (quantum group $\uqmu{}{\g}$, untwisted quantum loop algebra $\uqmu{}{\hat{\g}}$) with $\g$ simple~has 2-dimensional simple representations if only if it is isomorphic to $\widehat{\mathfrak{sl}}_2$ (or $\uqmu{}{\SL}$, $\uqmu{}{\widehat{\mathfrak{sl}}_2}$, resp.) (see Proposition \ref{prop:AffNoInfl}). The second question above has hence a negative answer in the setting of finite-dimensional simple Lie algebras (or untwisted affine Lie algebras, etc.).
\end{rem}\vspace*{-3mm}
\subsection{Representations and the category $\OO^{sh}$}\label{sec:O}
Take $\mu\in \Lambda^{\vee}$ with $V$ a $\uqmu{\mu}{\g}$-module. Then, for $\gamma=(\gamma_i)_{i\in I}\in \mathfrak{t}^{\times}$ where $\mathfrak{t}^{\times} = (\C^{\times})^I$, define the \textit{weight space}
$$ V_{\gamma} = V_{\gamma}^+= \{v\in V\,|\,\phi_{i,0}^{+}v=\gamma_{i}v \text{ for all } i\in I\}$$
and the set of \textit{weights} $P(V)=\{\gamma\in\mathfrak{t}^{\times}\,|\,V_{\gamma}\neq 0\}$. A \textit{weight vector} $v$ is a non-zero element~of~a weight space $V_{\gamma}$ for some $\gamma\in \mathfrak{t}^{\times}$. We will also need
$$ V_{\gamma}^- = \{v\in V\,|\,\phi_{i,\alpha_i(\mu)}^{-}v=\gamma_{i}^{-1}v \text{ for all } i\in I\}\text{ and } P^-(V)=\{\gamma\in\mathfrak{t}^{\times}\,|\,V_{\gamma}^-\neq 0\}.$$\vspace*{-4.75mm}\\
There is an injective group morphism $[-]:\Lambda_{\mathbb{Q}}\arr\mathfrak{t}^{\times}$ where $[\omega_i] = (q_i^{\delta_{i,j}})_{j\in I}$ on the fundamental weights. We use this morphism to carry the order of $\Lambda_{\mathbb{Q}}$ to $\mathfrak{t}^{\times}$. Clearly, by \eqref{eq:CommPhi} and \eqref{eq:PhiTX},
\begin{equation}\label{eq:xWt}
\uqmu{\mu,0}{\g}\cdot V_{\gamma}^{\pm}\subseteq V_{\gamma}^{\pm}\text{ with }x_{i,r}^{+}V_{\gamma}^{\pm} \subseteq V_{\gamma[\alpha_i]}^{\pm} \text{ and } x_{i,r}^{-}V_{\gamma}^{\pm} \subseteq V_{\gamma[-\alpha_i]}^{\pm}
\end{equation}
for all $i\in I$, $r\in \Z$ and $\gamma\in \mathfrak{t}^{\times}$. We now have the following important definition.
\begin{defn}[{\cite[Definition 4.8]{hshift}}]\label{def:O} Let $\mu\in \Lambda^{\vee}$. Then, the category $\mathcal{O}^{\mu}$ is the full subcategory of the category of all $\uqmu{\mu}{\g}$-modules whose objects are the modules $V$ satisfying
\begin{itemize}
\setlength{\itemsep}{1.5pt}
\item[(i)] $V = \bigoplus_{\gamma\in\mathfrak{t}^{\times}}V_{\gamma} = \bigoplus_{\gamma\in\mathfrak{t}^{\times}}V_{\gamma}^-$,
\item[(ii)] $\dim V_{\gamma}^{\pm}<\infty$ for all $\gamma\in \mathfrak{t}^{\times}$ and
\item[(iii)] $P(V)\cup P^-(V)\subseteq \bigcup_{i=1}^s D(\lambda_i)$ for some $\lambda_1,\dots,\lambda_s\in \mathfrak{t}^{\times}$ where $D(\lambda) = \{\gamma\in\mathfrak{t}^{\times}\,|\,\gamma\leq \lambda\}$.
\end{itemize}
\end{defn}\vspace*{-0.5mm}
Let $\mathfrak{r}$ be the group (under pointwise multiplication of rational functions)
$$ \mathfrak{r} = \{\Psi=(\Psi_{i}(z))_{i\in I}\in(\C(z))^I\,|\,\Psi_i(z)\text{ is regular at }z=0 \text{ and }\Psi_i(0)\neq 0 \text{ for all }i\in I\}$$
and define $\mathfrak{r}_{\mu}=\{\Psi=(\Psi_{i}(z))_{i\in I}\in \mathfrak{r}\,|\,\deg\Psi_i(z) = \alpha_i(\mu)\text{ for all }i\in I\}$.
Then, for fixed $\Psi\in \mathfrak{r}_{\mu}$ and $V$ in $\OO^{\mu}$, we call \textit{$\ell$-weight space of $V$} the generalized simultaneous eigenspace
$$V_{\Psi} = \{v\in V\,|\,\text{ there is } p\in \mathbb{N}\text{ such that } (\phi_{i,r}^{+}-\Psi_{i,r}^+)^pv=0 \text{ for all } i\in I\text{ and } r\in \mathbb{N}\} $$
where we used the expansion $\Psi_i(z) = \sum_{r\geq 0}\Psi_{i,r}^+z^r$ for $i\in I$. By \cite[Section 4.4]{hshift}, every $V$ in $\OO^{\mu}$ decomposes as the sum of its $\ell$-weight spaces, that is
$ V = \bigoplus_{\Psi\in\mathfrak{r}_{\mu}} V_{\Psi}$. Also, for all $\Psi\in \mathfrak{r}$,\vspace*{-0.2mm}
$$ V_{\Psi} \subseteq V_{\varpi(\Psi)}$$\vspace*{-4.45mm}\\
where $\varpi:\mathfrak{r}\arr \mathfrak{t}^{\times}$ is the group homomorphism mapping $\Psi = (\Psi_i(z))_{i\in I}$ to $\varpi(\Psi) = (\Psi_i(0))_{i\in I}$. In particular, the $\ell$-weight spaces of objects of $\OO^{\mu}$ are finite-dimensional. We will call \textit{$\ell$-weights} the elements of $\{\Psi\in\mathfrak{r}\,|\,V_{\Psi}\neq 0\}$ and \textit{$\ell$-weight vectors} the non-zero elements of $\bigcup_{\Psi\in\mathfrak{r}}V_{\Psi}$. 
\begin{defn}[{\cite[Section 4.4]{hshift}}]\label{def:HighestlWeight} Fix $V$ in $\OO^{\mu}$ and $\Psi \in \mathfrak{r}$. Then, $V$ is of highest $\ell$-weight $\Psi$ if there is a non-zero vector $v\in V$ verifying $V=\uqmu{\mu}{\g}\cdot v$ with $x_i^{+}(z)v = 0$ and $\phi_i^{\pm}(z)v=\Psi_{i}(z)v$ for all $i\in I$. In this case, $\phi_{i,r}^{\pm}v = \Psi_{i,r}^{\pm}v$ for all $i\in I$ and $r\in \Z$ where 
$$\textstyle \Psi_i(z) = \sum_{r\geq 0}\Psi_{i,r}^+z^r=\sum_{r\geq -\alpha_i(\mu)} \Psi_{i,-r}^-z^{-r}$$ 
are respectively expansions of the rational function $\Psi_i(z)$ in $z$ and $z^{-1}$. We call the generating vector $v$ above a \textit{highest $\ell$-weight vector} of $V$. It is determined uniquely up to a scalar.
\end{defn}\vspace*{-0.5mm}
The highest $\ell$-weight $\Psi$ of a highest $\ell$-weight module $V$ of $\OO^{\mu}$ is also uniquely determined. 
\begin{theorem}[{\cite[Theorem 4.12]{hshift}}]\label{thm:ClassSimp} Every simple module in $\OO^{\mu}$ is of highest $\ell$-weight $\Psi$ for some $\Psi\in \mathfrak{r}_{\mu}$. Moreover, the application sending a given simple module in $\OO^{\mu}$ to the associated highest $\ell$-weight in $\mathfrak{r}_{\mu}$ is a bijection (up to isomorphism).
\end{theorem}\vspace*{-0.5mm}
There is hence, for every $\Psi\in \mathfrak{r}_{\mu}$, a unique (up to isomorphism) simple module $L(\Psi)$ of $\OO^{\mu}$ with highest $\ell$-weight $\Psi$. We include a description of three remarkable families of these~simple modules. We will need, for $i\in I$ and $a\in \C^{\times}$, the $\ell$-weights $Y_{i,a}\in\mathfrak{r}_0$ and $\Psi_{i,a}^{\pm 1}\in \mathfrak{r}_{\pm \omega_i^{\vee}}$~given~by 
$$ (Y_{i,a})_j(z) = \left\{\begin{array}{ll} q_i\frac{1-azq_i^{-1}}{1-azq_i} & \text{if }i=j,\\
1 & \text{else}\end{array}\right.\quad \text{and}\quad (\Psi_{i,a}^{\pm 1})_j(z) = \left\{\begin{array}{ll} (1-az)^{\pm 1} & \text{if }i=j,\\
1 & \text{else.}\end{array}\right.$$
\begin{example}[see, e.g., {\cite[Example 3.12]{fh1}}]\label{ex:KR}
Fix $i\in I$, $a\in \C^{\times}$ and $k\in \mathbb{N}_{>0}$. Then $\OO^0$ contains the fundamental representation $L(Y_{i,a})$ and the Kirillov-Reshitikhin module $$W_{k,a}^{(i)} = L(Y_{i,a}Y_{i,aq_i^{2}}\dots Y_{i,aq_i^{2(k-1)}})$$
(which are also representations of the quantum loop algebra $\uqmu{}{\hat{\g}}$). A realization of $W_{k,aq^{1-2k}}^{(1)}$ for $\g=\SL$ is given on the space with basis $\{v_m\}_{m=0}^k$ via
$$x_{1,r}^+v_m = a^rq^{2r(1-m)}v_{m-1}\text{ and }x_{1,r}^-v_m = a^rq^{-2mr}[m+1]_q[k-m]_qv_{m+1}$$ for $r\in\Z$ with $v_{-1}=v_{k+1} = 0$ and
$$ \phi_1^{\pm}(z)v_m = q^{k-2m}\frac{(1-azq^{-2k})(1-azq^2)}{(1-azq^{2(1-m)})(1-azq^{-2m})}v_m.$$
\end{example}
\begin{example}\label{ex:Invertibles} Fix $\gamma=(\gamma_i)_{i\in I}\in \mathfrak{t}^{\times}$ and view it as the $\ell$-weight $(\gamma_i(z))_{i\in I}\in\mathfrak{r}_0$ consisting of the constant rational functions $\gamma_i(z)=\gamma_i$ for $i\in I$. Then, the associated simple module $L(\gamma)$ of $\mathcal{O}^0$ is $1$-dimensional and is said to be an invertible\footnote{The terminology comes from the fact that, as modules over the Hopf subalgebra $\uqmu{}{\hat{\mathfrak{b}}}$ of $\uqmu{}{\hat{\g}}$ (which is also a subalgebra of $\uqmu{0}{\g}$ by Remark \ref{rem:borQaff}), $L(\gamma)\otimes L(\gamma^{-1})\simeq L(\mathbbm{1})$ where $\mathbbm{1}\in\mathfrak{t}^{\times}$ is defined by $\mathbbm{1} = (1)_{i\in I}$.} representation of $\uqmu{0}{\g}$.
\end{example}
\begin{example}[{\cite[Example 4.13]{hshift}}]\label{ex:Prefund} Fix $i\in I$ and $a\in \C^{\times}$. Then, $\OO^{\omega_i^{\vee}}$ contains the positive prefundamental representation $L(\Psi_{i,a})$ which has dimension 1 and verifies, if $v\in L(\Psi_{i,a})$,
$$x_{j}^{\pm}(z)v = 0 \text{ and } \phi_j^{\pm}(z)v = (\Psi_{i,a})_j(z)v = (1-az\delta_{i,j})v$$
for all $j\in I$. Conversely, $\OO^{-\omega_i^{\vee}}$ contains the negative prefundamental representation $L(\Psi_{i,a}^{-1})$ which is infinite-dimensional and of particular importance in the study of Baxter's $Q$-operators and quantum integrable systems (see, e.g., \cite{fh1,hshift,hj}). A realization of $L(\Psi_{1,a}^{-1})$ for $\g=\SL$ is given on the space with basis $\{v_m\}_{m\geq 0}$ via
$$ x_{1,r}^+v_m = a^rq^{2r(1-m)}v_{m-1}\text{ and }(q-q^{-1})x_{1,r}^-v_m = a^rq^{-(2r+1)m}[m+1]_qv_{m+1} $$
for $r\in \Z$ with $v_{-1}=0$ and
$$ \phi_1^{\pm}(z)v_m = q^{-2m}\frac{1-azq^2}{(1-azq^{2(1-m)})(1-azq^{-2m})}v_m.$$
In general, the structure of $\uqmu{-\omega_i^{\vee}}{\g}$-module of $L(\Psi_{i,a}^{-1})$ extends the structure of $\uqmu{}{\hat{\mathfrak{b}}}$-module introduced in \cite[Section 4]{hj} using asymptotic limits of Kirillov--Reshitikhin modules.
\end{example}
\begin{rem}\label{rem:prefundMonGen} The importance of the prefundamental representations $L(\Psi_{i,a}^{\pm 1})$ with respect to the representation theory of shifted quantum affine algebras lies in the fact that every simple $\uqmu{\mu}{\g}$-module (for any $\mu\in \Lambda^{\vee}$) can be realized as a subquotient of some \textit{fusion product} of an invertible representation $L(\gamma)$ of $\uqmu{0}{\g}$ with (typically many) prefundamental representations \cite[Section 5]{hshift}. More details about this fact will be given in the following subsection. 
\end{rem}
Define now the category $\OO^{sh}$ as the direct sum (of abelian categories)
$$ \textstyle \OO^{sh}=\bigoplus_{\mu\in \Lambda^{\vee}} \OO^{\mu}.$$
By Theorem \ref{thm:ClassSimp}, irreducible objects in this category are parametrized by $\mathfrak{r}$ up to isomorphism. We also have the following characterization of finite-dimensional simple modules.
\begin{thm}[{\cite[Theorem 6.4]{hshift}}]\label{thm:ClassSimpDimFin} The finite-dimensional irreducible objects of the category $\OO^{sh}$ are the modules $L(\Psi)$ with $\Psi$ a monomial in the following $\ell$-weights : 
\begin{itemize}
\setlength{\itemsep}{1.5pt}
\item[(i)] the $Y_{i,a}$'s for $i\in I$ and $a\in\C^{\times}$,
\item[(ii)] the $\Psi_{i,a}$'s for $i\in I$ and $a\in\C^{\times}$,
\item[(iii)] the $\gamma$'s for $\gamma\in\mathfrak{t}^{\times}$.
\end{itemize}
\end{thm}\vspace*{-3mm}
\subsection{Fusion product and $q$-characters}\label{sec:Fusion} Let $\mathcal{E}_{\ell}$ be the ring of functions $f:\mathfrak{r}\arr \Z$ for which
\begin{itemize}
\item[(i)] $\{\varpi(\Psi)\,|\,\Psi\in\mathfrak{r}\text{ and }f(\Psi)\neq 0\}\subseteq \bigcup_{i=1}^s D(\lambda_i)$ for some $\lambda_1,\dots,\lambda_s\in \mathfrak{t}^{\times}$ and
\item[(ii)] $\{\Psi\in\mathfrak{r}\,|\,f(\Psi)\neq 0\text{ and }\varpi(\Psi)=\gamma\}$ is a finite set for any $\gamma\in\mathfrak{t}^{\times}$
\end{itemize}
with the sets $D(\lambda)$ as in Definition \ref{def:O}. The ring multiplication on $\mathcal{E}_{\ell}$ is convolution. Fix $V$ in $\mathcal{O}^{sh}$. Then the \textit{$q$-character} $\chi_q(V)$ is the element of $\mathcal{E}_{\ell}$ defined as (see \cite{fr,hshift})
$$ \textstyle 
\chi_q(V) = \sum_{\Psi\in\mathfrak{r}}\dim V_{\Psi}[\Psi]$$
where, for $\Psi\in\mathfrak{r}$, we denoted by $[\Psi]$ the map in $\mathcal{E}_{\ell}$ given by $[\Psi](\Psi')=\delta_{\Psi,\Psi'}$ on all $\Psi'\in \mathfrak{r}$.\par
We also have the (usual) \textit{character} $\chi(V)$ given by 
$$\textstyle 
\chi(V) =
\sum_{\gamma\in\mathfrak{t}^{\times}}\dim V_{\gamma}[\gamma]$$
where $\mathfrak{t}^{\times}$ is seen inside $\mathfrak{r}$ as in Example \ref{ex:Invertibles}. Lastly, for $V$ in $\OO^{sh}$ having a unique $\ell$-weight $\Psi$ whose weight $\varpi(\Psi)$ is maximal in $P(V)$ (for the order induced from that of $\mathfrak{t}^{\times}$), we define~the \textit{normalized $q$-character} $\overline{\chi_q}(V)$ and the \textit{normalized character} $\overline{\chi}(V)$ as 
$$\overline{\chi_q}(V) = [\Psi^{-1}]\chi_q(V) \text{ and }\overline{\chi}(V)=[(\varpi(\Psi))^{-1}]\chi(V).$$
Now, for $i\in I$ and $a\in \C^{\times}$, let $A_{i,a}\in\mathfrak{r}_0$ be given by the formula
$$A_{i,a}=Y_{i,aq_i^{-1}}Y_{i,aq_i}\Bigg(\prod_{\{j\in I| C_{j,i}=-1\}} Y_{j,a}\prod_{\{j\in I|C_{j,i}=-2\}} Y_{j,aq^{-1}}Y_{j,aq}\prod_{\{j\in I|C_{j,i}=-3\}} Y_{j,aq^{-2}}Y_{j,a}Y_{j,aq^2}\Bigg)^{-1}$$
with the $\ell$-weights $\{Y_{j,b}\}_{j\in I,b\in \C^{\times}}$ as in Section \ref{sec:O}. Clearly, $\varpi(A_{i,a})=[\alpha_i]$ with $\varpi(Y_{i,a}) = [\omega_i]$ so that $A_{i,a}$ (resp.~$Y_{i,a}$) can be seen as an analog, inside $\mathfrak{r}_0$, of the simple root $\alpha_i$ (resp.~of the fundamental weight $\omega_i$). We will follow for this article the convention of \cite{fh1,hj} and forget the brackets $[\cdot]$ when writing a monomial in the $Y_{j,b}$'s (with $j\in I$, $b\in\C^{\times}$) in the $q$-character of an object of $\OO^{sh}$. This applies in particular to the $\ell$-weights $A_{i,a}$ defined above.
\par The next remark is well known (see, e.g., the proof of \cite[Proposition 2.14]{her}).
\begin{rem}\label{rem:AAlgFree} The $A_{i,a}$'s with $i\in I$ and $a\in\C^{\times}$ are free in the multiplicative group $\mathfrak{r}$.
\end{rem} 
This remark and the theorem below partially explain the importance of the $\ell$-weights $A_{i,a}$.
\begin{thm}[{\cite[Theorem 5.11]{hshift}}, see also \cite{fm}]\label{thm:qCharA} Fix $\Psi\in \mathfrak{r}$. Then, 
$$ \overline{\chi_q}(L(\Psi)) \in 1+\mathbb{N}[[A_{i,a}^{-1}]]_{i\in I,a\in \C^{\times}}.$$
\end{thm}
\begin{example}\label{ex:qCharEX} Fix $i\in I$, $a\in \C^{\times}$ and $\gamma\in \mathfrak{t}^{\times}$. Then $\chi_q(L(\Psi_{i,a})) = [\Psi_{i,a}]$ and $\chi_q(L(\gamma))=[\gamma]$. In addition, for $\g = \SL$ and $k\in\mathbb{N}_{>0}$, by Example \ref{ex:KR} and Example \ref{ex:Prefund},
$$\textstyle \overline{\chi_q}(W_{k,aq^{1-2k}}^{(1)})=1+A_{1,a}^{-1}(1+\sum_{s=1}^{k-1}A_{1,aq^{-2}}^{-1}\dots A_{1,aq^{-2s}}^{-1})$$\vspace*{-3.25mm}\\
with $\overline{\chi_q}(L(\Psi_{1,a}^{-1}))=\lim_{k\arr \infty}\overline{\chi_q}(W_{k,aq^{1-2k}}^{(1)})$.
\end{example}
The last sentence of the above example foreshadows the following proposition. 
\begin{prop}[{\cite{hj}}]\label{prop:PrefondKR} Fix $i\in I$ with $a\in \C^{\times}$. Let $V_k=W_{k,aq_i^{1-2k}}^{(i)}$ for all $k\in\mathbb{N}_{>0}$. Then, \vspace*{-2.1mm}
$$\textstyle \overline{\chi_q}(L(\Psi_{i,a}^{-1}))=\lim_{k\arr \infty}\overline{\chi_q}(V_k) $$\vspace*{-2.55mm}\\ as formal power series in $1+A_{i,a}^{-1}\mathbb{N}[[A_{j,b}^{-1}]]_{j\in I,b\in \C^{\times}}$. Also, for every $k\in \mathbb{N}_{>0}$ and $\Psi\in\mathfrak{r}$,
$$ \dim (V_{k})_{m_k\Psi}\leq \dim(V_{k+1})_{m_{k+1}\Psi} \leq \dim (L(\Psi_{i,a}^{-1}))_{\Psi_{i,a}^{-1}\Psi} $$\vspace*{-3.15mm}\\
where $m_k$ (resp.~$m_{k+1}$) is the highest $\ell$-weight of $V_k$ (resp.~$V_{k+1}$).
\end{prop}
Consider now, for $\mu,\nu\in\Lambda^{\vee}$ and some variable $u$, the \textit{Drinfeld coproduct} \cite[Section 5.2]{hshift} $$\Delta_{\mu,\nu}^{(u)}:\uqmu{\mu+\nu}{\g}\arr (\uqmu{\mu}{\g}\otimes\uqmu{\nu}{\g})((u))$$ 
which is defined implicitly, for $i\in I$, by $\phi_i^{\pm}(z)\mapsto\phi_i^{\pm}(z)\otimes \phi_i^{\pm}(zu)$ with 
\begin{equation}\label{eq:Dcop}
x_i^{+}(z)\mapsto x_i^+(z)\otimes 1+\phi^-_i(z)\otimes x_i^+(zu)\text{ and } x_i^{-}(z)\mapsto 1\otimes x_i^-(zu)+x_i^-(z)\otimes\phi^+_i(zu).
\end{equation}
The specialization of $\Delta_{\mu,\nu}^{(u)}$ at $u=1$ gives a well-defined algebra morphism from $\uqmu{\mu+\nu}{\g}$ to a completion of $\uqmu{\mu}{\g}\otimes\uqmu{\nu}{\g}$ (cf.~\cite[Remark 5.3]{hshift}).~The resulting map however~does~not~give directly a well-defined action of $\uqmu{\mu+\nu}{\g}$ on general tensor products of objects of $\OO^{\mu}$~and~$\OO^{\nu}$. The \textit{fusion product} of \cite{hAff,hfus,hshift} is an operation designed to resolve this technical issue.\par
To recall the definition of this operation, fix $V_1$ and $V_2$ two highest $\ell$-weight modules in $\OO^{\mu}$ and $\OO^{\nu}$ respectively. Then, the Drinfeld coproduct above gives a well-defined $\uqmu{\mu+\nu}{\g}$-action on the space $(V_1\otimes V_2)((u))$ of formal Laurent power series with coefficients inside $V_1\otimes V_2$. In addition, the subspace of rational Laurent power series $(V_1\otimes V_2)(u)\subseteq (V_1\otimes V_2)((u))$ (that is formal Laurent power series having a rational expression) is stable under this $\uqmu{\mu+\nu}{\g}$-action (cf.~\cite[Section 5.3]{hshift}). Define 
$$\mathscr{A} = \{f(u)\in \C(u)\,|\,f(u)\text{ is regular at }u=1\}$$
and fix highest $\ell$-weight vectors $v_1\in V_1$ with $v_2\in V_2$. Define also $X_{\mathscr{A}}$ as the $(\mathscr{A}\otimes \uqmu{\mu+\nu}{\g})$-submodule of $(V_1\otimes V_2)(u)$ generated by $v_1\otimes v_2$. 
\begin{defn}[{\cite[Section 5.3]{hshift}}] The \textit{fusion product} of $V_1$ and $V_2$ is the $\uqmu{\mu+\nu}{\g}$-module $$ V_1\star V_2 = X_{\mathscr{A}}/((u-1)X_{\mathscr{A}}).$$
\end{defn}
\begin{thm}[{\cite[Theorem 5.4]{hshift}}]\label{thm:HighestWeightFus} Note $\Psi_1$ (resp.~$\Psi_2$) the highest $\ell$-weight of $V_1$ (resp.~$V_2$). Then, $V_1\star V_2$ is a well-defined highest $\ell$-weight 
object of $\OO^{\mu+\nu}$ with highest $\ell$-weight $\Psi_1\Psi_2$.~Also,
$$ \chi_q(V_1\star V_2) = \chi_q(V_1)\star \chi_q(V_2) $$
and taking $q$-characters gives a well-defined injective ring morphism $\chi_q : K_0(\OO^{sh})\arr \mathcal{E}_{\ell} $~(where the multiplication of two classes of simple objects, say $[V_1]$ and $[V_2]$, in the Grothendieck group $K_0(\mathcal{O}^{sh})$ is given by the class $[V_1\star V_2]$ of the corresponding fusion product).
\end{thm}
\begin{rem} Using the fact that the fusion product of two highest $\ell$-weight modules is again of highest $\ell$-weight, we can define iteratively the fusion product $ V_1\star \dots \star V_m$ (where $V_1,\dots,V_m$ are highest $\ell$-weight modules in $\OO^{sh}$) by the convention
$$ V_1\star \dots \star V_m = (V_1\star \dots \star V_{m-1})\star V_m.$$
Note that the fusion product of two general objects in $\OO^{sh}$ (not necessarily of highest $\ell$-weight) is not defined. Hence, the category $\OO^{sh}$ is not (known to be) monoidal even if its~Grothendieck group has a ring structure. It is also not clear whether the fusion product is coassociative.
\end{rem}
\begin{rem}\label{rem:Dcoproduct} Fix $V_1$ a highest $\ell$-weight $\uqmu{\mu}{\g}$-module and suppose that $V_2$ is a $1$-dimensional $\uqmu{\nu}{\g}$-module. Then $x_i^{\pm}(z)V_2 = 0$ by \eqref{eq:xWt} and specializing $\Delta_{\mu,\nu}^{(u)}$ at $u=1$ gives a well-defined $\uqmu{\mu+\nu}{\g}$-action on the space 
$V_1\otimes V_2$. Denoting by $V$ the resulting $\uqmu{\mu+\nu}{\g}$-module, we get
$$\chi_q(V) = \chi_q(V_1)\chi_q(V_2) = \chi_q(V_1\star V_2),$$
but $V$ is not always a highest $\ell$-weight module and can thus be non-isomorphic to $V_1\star V_2$. 
\end{rem}
As fusion product produces only highest $\ell$-weight modules, $L(\Psi\Psi')$ is always isomorphic to the head of $L(\Psi)\star L(\Psi')$. This suffices to show the next result (alluded to in Remark \ref{rem:prefundMonGen}).
\begin{thm}[{\cite[Corollary 5.6]{hshift}}]\label{thm:PrefundMonGen} Fix $\Psi\in \mathfrak{r}$. Then, the simple module $L(\Psi)$ is isomorphic to the head of a fusion product of an invertible representation of $\uqmu{0}{\g}$ with (typically many) prefundamental representations (positive and negative).
\end{thm}
Another interesting result about prefundamental representations is the following.
\begin{theorem}[{\cite[Theorem 5.5]{hshift}}, see also \cite{fh1}]\label{thm:PrefundSimpFus} A fusion product involving only invertible representations and positive (resp.~negative) prefundamental representations is simple.
\end{theorem}
\begin{example}\label{ex:qWronsk} Take $\g=\mathfrak{sl}_2$. Then the quantum Wronskian relation of \cite[Example 5.9(i)]{hshift} is the relation of $K_0(\OO^{sh})$ given, for $a\in \C^{\times}$, by
$$ [L(\Psi_{1,a})][L(\Psi_{1,a}^{-1})] = 1+[-2\omega_1][L(\Psi_{1,aq^2})]L(\Psi_{1,aq^{-2}}^{-1})] $$
where $[-2\omega_1]$ stands for the isoclass in $K_0(\OO^{sh})$ of the representation 
$L([-2\omega_1])$ and where~the summand $1$ corresponds to the trivial representation $L(\mathbbm{1})$ (with $\mathbbm{1}=(1)$ here). There are\,hence non-simple fusion products of (positive and negative) prefundamental representations in $\OO^{sh}$.
\end{example}
We now end this section by extending the above results to the subalgebras $\uqmu{\nu}{\g_J}\subseteq \uqmu{\mu}{\g}$ of Section \ref{sec:Def} for $J\subseteq I$ fixed with $\nu=\res_J(\mu)$ the projection of $\mu\in\Lambda^{\vee}$ onto $\Lambda_J^{\vee}$. For this,~note that replacing $\mathfrak{t}^{\times}$ by $\mathfrak{t}^{\times}_J = (\C^{\times})^J$ (which is ordered by $\gamma\leq \gamma'$ if and only if $\gamma'\gamma^{-1}\in [\res_J(Q_J^+)]$) enables the definition a category $\OO^{\nu}_J\subseteq \uqmu{\nu}{\g_J}$--Mod as in Definition \ref{def:O}. Also, letting
$$\mathfrak{r}^J=\{\Psi=(\Psi_{j}(z))_{j\in J}\in(\C(z))^J\,|\,\Psi_j(z)\text{ is regular at }z=0\text{ and }\Psi_j(0)\neq 0\text{ for all }j\in J\}$$
and $\mathfrak{r}_{\nu}^J = \{\Psi=(\Psi_j(z))_{j\in J}\in \mathfrak{r}^J\,|\,\deg\Psi_j(z)=\alpha_j(\nu)=\alpha_j(\mu) \text{ for all }j\in J\}$ allows us to define \textit{$\ell$-weight spaces} and \textit{highest $\ell$-weight modules} for this category $\mathcal{O}^{\nu}_J$. 
\begin{rem} The functor $\res_J^{\mu}$ of Section \ref{sec:Def} typically sends objects of $\OO^{\mu}$ to $\uqmu{\nu}{\g_J}$-modules that do not belong to $\OO^{\nu}_J$. Indeed, fix $\g = \mathfrak{sl}_4$ and consider the sum of invertible representations $$\textstyle V = \bigoplus_{m\geq 0}L([-m\alpha_1]).$$
Then $V$ lies in $\OO_{0}$, but $\res_{\{3\}}^0(V) 
$ contains an infinite-dimensional weight space. 
\end{rem}\vspace*{-0.25mm}
Denote by $\res_J:\mathfrak{r}\arr\mathfrak{r}^J$ the group epimorphism sending $\Psi=(\Psi_i(z))_{i\in I}\in\mathfrak{r}$ to its \textit{$J$-part} 
$$\res_J(\Psi)=(\Psi_j(z))_{j\in J}\in\mathfrak{r}^J.$$ \newpage
The repetitive use of the notation $\res_J$ is justified by the following lemma\footnote{Note that special cases of this lemma (with similar proofs) for untwisted quantum loop algebras can easily be found in the litterature (see, e.g., \cite[Proposition 6.4]{hKRgen} and the references therein).}.
\begin{lem}\label{lem:FacRestSimple} Fix $\Psi\in\mathfrak{r}_{\mu}$ and let $V=L(\Psi)$. Fix $v\in V$ a highest $\ell$-weight vector. Then the $\uqmu{\nu}{\g_J}$-submodule $\langle v\rangle_J$ of $\res_J^{\mu}(V)$ generated by $v$ is simple and of highest $\ell$-weight $\res_J(\Psi)$. 
\end{lem}\vspace*{-2mm}
\begin{proof} Since the $\uqmu{\nu}{\g_J}$-module $\langle v\rangle_J$ is clearly of highest $\ell$-weight $\res_J(\Psi)$, it suffices to show that it contains no non-trivial proper $\uqmu{\nu}{\g_J}$-submodule. Assume the contrary and take such a proper submodule $W\neq 0$. Then, the set of weights $P(W)\subseteq \mathfrak{t}^{\times}_J$ contains a maximal element $\gamma$ as $W$ lies in $\OO^{\nu}_J$. Fixing $0\neq w\in W_{\gamma}$ and $r\in \Z$, we obtain that, by \eqref{eq:xWt} and maximality~of the weight $\gamma$, $$x_{j,r}^+w\in W_{\gamma[\alpha_j]}=0$$
for all $j\in J$. Also, by \eqref{eq:TriangularDecJ}, $w\in W\subseteq \langle v\rangle_J$ has the form $w=xv$ for a $x\in \uqmu{\nu,-}{\g_J}\subseteq \uqmu{\mu,-}{\g}$. In particular, since $0\neq w\in W_{\gamma}$, \eqref{eq:xWt} easily implies that $w$ actually remains a weight vector when seen inside $V$ (i.e.~it is also an eigenvector of the $\phi_{i,0}^+$'s with $i\not\in J$). Moreover, its weight in $V$ must be of the form $\varpi(\Psi)[-\alpha]$ for some $\alpha\in Q_J^+\subseteq Q^+$ non-zero (recall that $v\neq w$ since $W\neq\langle v\rangle_J$). Thus, for $r\in \Z$ and $i\not\in J$, as $P(V)\subseteq D(\varpi(\Psi))$ with $\alpha-\alpha_i\not\in Q^+$, 
$$x_{i,r}^+w\in V_{\varpi(\Psi)[\alpha_i-\alpha]} =0$$ 
and it follows that $\uqmu{\mu,+}{\g}\cdot w = 0$. In particular, the weights of the $\uqmu{\mu}{\g}$-submodule $\langle w\rangle \subseteq V$ generated by $w$ are contained in the cone $D(\varpi(\Psi)[-\alpha])$ by \eqref{eq:TriangularDec}--\eqref{eq:xWt}, but this cone cannot contain the weight $\varpi(\Psi)$ of the highest $\ell$-weight vector $v$ since $\alpha\neq 0$. This shows that $\langle w\rangle$ is a non-zero proper $\uqmu{\mu}{\g}$-submodule of $V$, which is impossible as $V=L(\Psi)$ by hypothesis.
\end{proof}\vspace*{-1.5mm}
A first relevant result about the category $\OO^{\nu}_J$ is the following adapted version of Theorem~\ref{thm:ClassSimp} (which can be proven exactly like \cite[Theorem 4.12]{hshift}).
\begin{thm*}\label{thm:ClassSimpJ} 
The simple objects of $\OO^{\nu}_J$ are all of highest $\ell$-weight $\Psi$ for some $\Psi\in\mathfrak{r}_{\nu}^J$~and are parametrized (up to isomorphism) by their highest $\ell$-weight. 
\end{thm*}
For $\Psi\in\mathfrak{r}^J_{\nu}$, we note $L^J(\Psi)$ the simple object of $\OO^{\nu}_J$ with highest $\ell$-weight $\Psi$. By Lemma~\ref{lem:FacRestSimple}, this module can be constructed as a $\uqmu{\nu}{\g_J}$-submodule of the restriction of the simple~object of $\OO^{\nu}$ with highest $\ell$-weight $\Psi$ (with $\Psi$ seen in $\mathfrak{r}_{\nu}$ through the obvious inclusion $\mathfrak{r}^J_{\nu}\subseteq \mathfrak{r}_{\nu}$). Let $$\textstyle \OO^{sh}_J = \bigoplus_{\nu\in\Lambda_J^{\vee}}\OO^{\nu}_J.$$
Then, isoclasses of simple modules in $\OO_{J}^{sh}$ are parametrized by $\mathfrak{r}^J$. Also, $\OO^{sh}_J$ contains prefundamental modules $L^J(\Psi_{j,a}^{\pm 1})$ (with $j\in J$, $a\in \C^{\times}$), Kirillov--Reshitikhin modules and~invertible representations $L^J(\gamma)$ (where $\gamma\in \mathfrak{t}^{\times}_J$). There is furthermore a $q$-character map $$ \chi_q : K_0(\OO^{sh}_J)\arr \mathcal{E}_{\ell,J}$$
with $\mathcal{E}_{\ell,J}$ an obvious modification of $\mathcal{E}_{\ell}$. Clearly, if $V$ is in $\OO^{\mu}$, $\Psi\in \mathfrak{r}$ and $0\neq v\in V_{\Psi}$, then $v$ becomes a $\ell$-weight vector of $\ell$-weight $\res_J(\Psi)\in\mathfrak{r}^J$ when seen inside $\res_J^{\mu}(V)$. Hence, defining $\res_J[\Psi]=[\res_J(\Psi)]$ for all $\Psi\in\mathfrak{r}$ and extending, we get 
\begin{equation}\label{eq:chiRes}
\res_J(\chi_q(V)) = \chi_q(\res_J^{\mu}(V))
\end{equation}
for each object $V$ in $\OO^{\mu}$ such that $\res_J^{\mu}(V)$ belongs to $\OO^{\nu}_J$. The result below is shown exactly~like \cite[Theorem 5.11]{hshift}. We also have obvious analogues for $\OO^{sh}_J$ of Theorems \ref{thm:HighestWeightFus}, \ref{thm:PrefundMonGen} and~\ref{thm:PrefundSimpFus}. 
\begin{thm*}\label{thm:qCharAJ} Fix $\Psi\in \mathfrak{r}_{\nu}^J$. Then, $\overline{\chi}_q(L^J(\Psi)) \in 1+ \mathbb{N}[[\res_J(A_{j,a}^{-1})]]_{j\in J,a\in\C^{\times}}$.
\end{thm*}
\section{Inflations for representations of shifted quantum affine algebras}\label{sec:Infl}
This section introduces the notion of \textit{inflation} for representations of shifted quantum affine algebras as remarkable preimages for the functors $\res_J^{\mu}$ of Section \ref{sec:Def}. The section starts with various examples (Subsection \ref{sec:Prop}) and a precise investigation of the possible $q$-characters that a given inflation can have (Subsection \ref{sec:Inflqchar}). The section then continues with results regarding the compatibility between inflations and fusion products (Subsection \ref{sec:Existence}) and with the proof of Corollary \ref{cor:ExistenceSL} which asserts the existence of inflations for simple objects of $\OO^{sh}_J$ when~$\g_J$~is isomorphic, as a Lie algebra, to a direct sum of copies of $\SL$ and $\SLt$. Finally, the section ends with a general existence theorem for inflations of finite-dimensional simple objects in $\mathcal{O}_{J}^{sh}$ and with a conjecture stating that the finite-dimensionality hypothesis above is in fact not needed (i.e.~that every simple object of $\OO^{sh}_J$, finite-dimensional of not, admits an inflation to $\g$). This conjecture is proven under the technical assumption that $\g$ is of type A--B.
\vspace*{-2mm}
\subsection{Definition and first examples}\label{sec:Prop} Let us start with the definition of inflation.
\begin{defn}\label{def:Infl} Fix $J\subseteq I$ with $\nu$ a coweight of $\g_J$ and let $W$ be an object of $\mathcal{O}_{J}^{\nu}$. Take also $\mu$ a coweight of $\g$ with $\res_J(\mu)=\nu$ and let $V$ be an object of $\OO^{\mu}$. Then, $V$ is a \textit{$J$-inflation of $W$ to $\g$ with coweight $\mu$} if the following two conditions are satisfied : 
\begin{itemize}
\item[(i)] $\res_J^{\mu}(V)\simeq W$ as $\uqmu{\nu}{\g_J}$-modules and
\item[(ii)] $x_i^{\pm}(z)V = 0$ for all $i\not\in J$.
\end{itemize}
We will sometimes call $V$ a \textit{$J$-inflation of $W$ to $\g$} or, simply, an \textit{inflation of $W$}.
\end{defn}
A $J$-inflation of $W$ to $\g$ with coweight $\mu$ is thus a special preimage for the restriction functor $\res_J^{\mu}(V)$ on which the subalgebra of $\uqmu{\mu}{\g}$ generated by $\{x_{i,r}^{\pm}\,|\,i\not\in J,\,r\in\Z\}$ acts trivially. \par\vspace*{-0.2mm}
The following two results will be of use. (Observe that the first one follows easily from the exactness of the restriction functor $\res_J^{\mu}$.)
\begin{lemma*}\label{lem:InflSimpleSimple} Fix $J\subseteq I$ and let $W$ be a simple object of $\OO^{sh}_J$. Let also $V$ be an object of $\OO^{sh}$. Suppose that $V$ is a $J$-inflation of $W$ to $\g$. Then, $V$ is irreducible in $\OO^{sh}$.
\end{lemma*}
\begin{lem}\label{lem:HighestWeightRest} Fix $J\subseteq I$ and let $\mu$ be a coweight of $\g$. Let also $V$ be a $\uqmu{\mu}{\g}$-module for which $x_i^{\pm}(z)V = 0$ whenever $i\not\in J$. Fix finally 
$v\in V$. Then, $v$ is a highest $\ell$-weight vector for $V$ if and only if it is a highest $\ell$-weight vector when seen inside $\res_J^{\mu}(V)$. Hence, $J$-inflations to $\g$ of highest $\ell$-weight modules of $\OO^{sh}_J$ are highest $\ell$-weight modules of $\OO^{sh}$.
\end{lem}\vspace*{-2mm}
\begin{proof}
Assume that $v$ is a highest $\ell$-weight vector for $V$. Then, using the decompositions~\eqref{eq:TriangularDec} and \eqref{eq:TriangularDecJ} with the fact that $x_i^{\pm}(z)V=0$ for $i\not\in J$, we get (for $\nu=\res_J(\mu)$) \vspace*{-0.5mm}
$$ V = \uqmu{\mu}{\g}\cdot v = \uqmu{\mu,-}{\g}\cdot v = \uqmu{\nu,-}{\g_J}\cdot v \subseteq \uqmu{\nu}{\g_J}\cdot v \subseteq V$$
so that $v$ is easily seen to be a highest $\ell$-weight vector for $\res_J^{\mu}(V)$. Reciprocally, suppose that $v$ is a highest $\ell$-weight vector when seen inside the $\uqmu{\nu}{\g_J}$-module $\res_J^{\mu}(V)$. Then, $v$ trivially generates $V$ as a $\uqmu{\mu}{\g}$-module and is such that $x_j^{\pm}(z)v = 0$ whenever $j\in J$. Our hypothesis hence gives $x_i^{\pm}(z)v = 0$ for all $i\in I$ and it suffices to prove that $v$ lives inside a $1$-dimensional $\ell$-weight space of $V$. This is however clear as the $\ell$-weight space of $\res_J^{\mu}(V)$ containing $v$ (for the action of the $\phi_{j,r}^+$'s with $j\in J$ and $r\in \Z$) is $1$-dimensional. This ends the proof.
\end{proof}\vspace*{-1.5mm}
\begin{example}\label{ex:InvertiblePositivePrefond} Fix $J\subseteq I$ with $\gamma\in \mathfrak{t}^{\times}_J$, $a\in \C^{\times}$ and $j\in J$. Then the invertible representation $L(\gamma)$ and the positive prefundamental representation $L(\Psi_{j,a})$ are respectively $J$-inflations of the invertible representation $L^J(\gamma)$ and the positive prefundamental representation $L^J(\Psi_{j,a})$.
\end{example}
\begin{example}\label{ex:InflPrefondSL} Fix $j\in I$ and $a\in \C^{\times}$. Consider $ \Psi_{j,a}^* = Y_{j,aq_j^{-1}}\prod_{i\in I,\,C_{i,j}<0} \Psi_{i,aq_i^{-C_{i,j}}}$ and 
$$ \textstyle \widetilde{\Psi}_{j,a}=\Psi_{j,a}^{-1}\left(\prod_{i\in I\, C_{j,i}=-1}\Psi_{i,aq_j}\right)\left(\prod_{i\in I\, C_{j,i}=-2}\Psi_{i,a}\Psi_{i,aq_j^2}\right)\left(\prod_{i\in I\, C_{j,i}=-3}\Psi_{i,aq_j^{-1}}\Psi_{i,aq_j}\Psi_{i,aq_{j}^3}\right).$$
Then, one easily deduces from \cite[Example 5.2(iv) and Example 6.6]{hshift} with Examples \ref{ex:KR} and \ref{ex:Prefund} that the objects $L(\Psi^*_{i,a})$ and $L(\widetilde{\Psi}_{i,a})$ of $\OO^{sh}$ are respectively inflations of the $2$-dimensional fundamental representation $L^J(Y_{j,a})$ of $\uqmu{0}{\g_J}\simeq U_{q_j}^0(\SL)$ and of the negative prefundamental representation $L^J(\Psi_{j,a}^{-1})$ of $\uqmu{\nu}{\g_J}\simeq U_{q_j}^{\nu}(\SL)$ (where $\nu = \res_J(-\omega_j^{\vee}) = -\omega_j^{\vee}$).
\end{example}
\begin{rem} The objects $L(\Psi_{j,a}^*)$ and $L(\widetilde{\Psi}_{j,a})$ of the above example also arise in two important systems of relations for the Grothendieck ring $K_0(\OO^{sh})$: the $QQ^*$-system and the $Q\tilde{Q}$-system (see, e.g., \cite[Example 5.9(ii) and Example 6.6]{hshift}). We will return to these systems and recall briefly their connection with quantum integrable systems in Section \ref{sec:Inflqchar}.
\end{rem}
Recall from Section \ref{sec:Def} the definition of the spectral shift automorphism $\tau_a$ of $\uqmu{\mu}{\g}$ (and $\uqmu{\nu}{\g_J}$). The lemma below follows directly from Definition \ref{def:Infl}.
\begin{lemma*}\label{lem:InflShift} Fix $J\subseteq I$ and $W$ an object of $\OO^{sh}_{J}$. Suppose that $V$ is a $J$-inflation of $W$ to $\g$. Then $V(a)$ is a $J$-inflation of $W(a)$ to $\g$ for all $a\in \C^{\times}$.
\end{lemma*}
\begin{example}\label{ex:InflPrefondSLt} Fix $\g$ and $j_1,j_2\in I$ with $\g_J \simeq \SLt$ for $J=\{j_1,j_2\}\subseteq I$. Let $\mu = \omega_{j_2}^{\vee}-\alpha_{j_1}^{\vee}-\alpha_{j_2}^{\vee}$ and denote by $\Psi_p$ the following $\ell$-weight\footnote{Remark that $C_{i,j_1} \leq -2$ for all $i\not\in J$ since $\g$ cannot be of type $G_2$ (as $\SLt\simeq \g_J$).} of $\mathfrak{r}$\vspace*{0.5mm}
{\small $$\Bigg(\prod_{i\not\in J,\, C_{j_1,i}=-1} \Psi_{i,q_{j_1}}\Bigg)\Bigg(\prod_{i\not\in J,\, C_{j_2,i}=-1} \Psi_{i,q_{j_1}^2}\Bigg)\Bigg(\prod_{i\not\in J,\, C_{j_1,i}=-2} \Psi_{i,1}\Psi_{i,q_{j_1}^2}\Bigg)\Bigg(\prod_{i\not\in J,\, C_{j_2,i}=-2} \Psi_{i,q_{j_1}}\Psi_{i,q_{j_1}^3}\Bigg).$$}\vspace*{-1.25mm}\\
For $n\geq m\geq 0$, let also $\Psi_{n,m}=\Psi_{j_1,1}^{-1} A_{j_1,1}^{-1}\dots A_{j_1,q_{j_1}^{2(1-n)}}^{-1}A_{j_2,q_{j_1}}^{-1}\dots A_{j_2,q_{j_1}^{3-2m}}^{-1}$. Then, if $i\not\in J$,
$$ (\Psi_p\Psi_{n,m})_i(z) = \left\{\begin{array}{ll}
q_{j_1}^n(1-zq_{j_1}^{1-2n}) & \text{if } C_{j_1,i} =-1 \text { and } C_{j_2,i}=0,\smallskip\\
q_{j_1}^m(1-zq_{j_1}^{2(1-m)}) & \text{if } C_{j_1,i} =-1\text { and } C_{j_2,i}=0,\smallskip\\
q_{j_1}^{2n}(1-zq_{j_1}^{2(1-n)})(1-zq_{j_1}^{-2n}) & \text{if } C_{j_1,i} =-2\text { and } C_{j_2,i}=0,\smallskip\\
q_{j_1}^{2m}(1-zq_{j_1}^{3-2m})(1-zq_{j_1}^{1-2m}) & \text{if } C_{j_2,i} =-2\text { and } C_{j_1,i}=0,\smallskip\\
1& \text{else}
\end{array}
\right.$$
with $(\Psi_p\Psi_{n,m})_j(z)=\Psi^{(n,m)}_j(z)$ for $j\in J$. Define a $\uqmu{\mu}{\g}$-action on the space $V$ with $\C$-basis $\{v_{n,m}\}_{n\geq m\geq 0}$ by $x_{i}^{\pm}(z)V = 0$ for $i\not\in J$ as well as
\begin{center}
$x_{j_1,r}^+v_{n,m} = q_{j_1}^{2r(1-n)}[n-m]_{q_{j_1}}v_{n-1,m},\quad x_{j_1,r}^-v_{n,m} = (q_{j_1}-q_{j_1}^{-1})^{-1}q_{j_1}^{-n(2r+1)}v_{n+1,m}$,\\
$x_{j_2,r}^+v_{n,m} = q_{j_1}^{r(3-2m)}v_{n,m-1}, \quad x_{j_2,r}^-v_{n,m} = q_{j_1}^{r(1-2m)}[m+1]_{q_{j_1}}[n-m]_{q_{j_1}}v_{n,m+1}$
\end{center}
and $\phi_s^{\pm}(z)v_{n,m} = (\Psi_p\Psi_{n,m})_s(z)v_{n,m}$ for $r\in \Z$ and $s\in I$. By \cite[Section 4.1]{hj},
$$\res_J^{\mu}(V)\simeq L^J(\Psi_{j_1,1}^{-1})$$ as modules over $\uqmu{\nu}{\g_J}\simeq U_{q_{j_1}}^{\nu}(\SLt)$ with $\nu = \res_J(\mu)=-\omega_{j_1}^{\vee}$. To ensure that the above indeed gives a well-defined $\uqmu{\mu}{\g}$-action, it hence suffices to check the 
relations \eqref{eq:CommPhi}--\eqref{eq:RelhPhi} involving generators of the form $x_{i,r}^{\pm}$ or $\phi_{i,r}^{\pm}$ with $i\not\in J$. This is an easy (but tedious) computation that we choose to omit for brevity. Now, by the previous remarks, $V$ is a $J$-inflation of $W = L^J(\Psi_{j_1,1}^{-1})$ to $\g$ with coweight $\mu$. In addition, by interchanging $j_1$ and $j_2$ everywhere above, one can easily construct a $J$-inflation to $\g$ for the prefundamental representation $L^J(\Psi_{j_2,1}^{-1})$ of $\OO^{sh}_J$. Thus, by Lemma \ref{lem:InflShift}, all negative prefundamental representations in $\OO_{J}^{sh}$ admit inflations in this case.
\end{example}
\begin{rem}\label{rem:InflJnotconnected} Take $J\subseteq I$ and $J'\subseteq J$ with $C_{j,j'}=0$ for all $j\in J\backslash J'$ and $j'\in J'$. Take~moreover an object $W$ in $\OO^{\nu}_{J'}$ for some $\nu\in\Lambda^{\vee}_{J'}$. Then, the $\uqmu{\nu}{\g_{J'}}$-action on $W$ induces a $\uqmu{\nu}{\g_J}$-action by setting $x_j^{\pm}(z)W=(\phi_j^{\pm}(z)-1)W=0$ for each $j\in J\backslash J'$ (see Definition \ref{def:Uqmu}). Let us call $W_J$ the resulting $\uqmu{\nu}{\g_J}$-module and $\res_{J,J'}^{\nu}$ the restriction functor induced from the inclusion of $\uqmu{\nu}{\g_{J'}}$ in $\uqmu{\nu}{\g_J}$. Then $\res_{J,J'}^{\nu}(W_J) = W$ by construction and \vspace*{-0.5mm}
$$W_J\simeq L^J(\Psi)$$ if $W\simeq L^{J'}(\Psi)$ for some $\Psi\in\mathfrak{r}_{J'}$ (by exactness of $\res_{J,J'}^{\nu}$ and definition of $W_J$).
\end{rem}
\begin{example}\label{rem:InflPrefondSS} We apply the results of the above remark to a particular case. Let $J'\subseteq J\subseteq I$ be as in this remark and suppose that $\g_{J'}\simeq \SL$ or $\g_{J'}\simeq \SLt$. Fix also $j\in J'$ with $$W=L^{J'}(\Psi_{j,1}^{-1})$$ and let $V$ be the $J'$-inflation of $W$ to $\g$ studied in Example \ref{ex:InflPrefondSL} if $\g_{J'}\simeq \SL$ or in Example \ref{ex:InflPrefondSLt} if $\g_{J'}\simeq \SLt$. We note $\mu\in \Lambda^{\vee}$ the coweight associated to this inflation $V$. Then, looking~at~the $\uqmu{\mu}{\g}$-action on $V$ (given in \cite[Example 5.2(iv)]{hshift} if $\g_{J'}\simeq \SL$ or in Example \ref{ex:InflPrefondSLt} if $\g_J'\simeq \SLt$) allows us to easily deduce with Remark \ref{rem:InflJnotconnected} that $$\res_{J}^{\mu}(V) \simeq W_J \simeq L^J(\Psi_{j,1}^{-1})$$
in $\OO^{sh}_J$. The module $V$ is thus also a $J$-inflation of $L^J(\Psi_{j,1}^{-1})$ to $\g$.
\end{example}
We now turn to the main part of this paper which is the proof of various existence theorems for inflations of irreducible modules in $\OO^{sh}_J$. We start by the proof of the following proposition (mentioned in Section \ref{sec:Intro} and Remark \ref{rem:qAffRestnotdense}) which shows that the restriction functors appearing in the study of untwisted quantum loop algebras are not essentially surjective on simple finite-dimensional modules. The proof is similar (but slightly simpler) for the case of the restriction functors appearing in usual Lie theory and in the study of affine Lie algebras/quantum~groups of finite type.~As before, we assume that $\g$ is a simple finite-dimensional Lie algebra.
\begin{prop}\label{prop:AffNoInfl} Suppose that the untwisted quantum loop algebra $\uqmu{}{\hat{\g}}$ has $2$-dimensional irreducible representations. Then $\g\simeq \SL$.
\end{prop}\vspace*{-2.75mm}
\begin{proof}
Fix a $2$-dimensional irreducible representation $V$ of $\uqmu{}{\hat{\g}}$ and use \cite[Remark 5.4(b)]{ft1} to view it as a $2$-dimensional irreducible $\uqmu{0}{\g}$-module on which $\phi_{i,0}^- -(\phi_{i,0}^+)^{-1}$ acts trivially~for all $i\in I$. Then $V$ belongs to $\OO^{0}$ (as $\OO^{0}$ contains all finite-dimensional $\uqmu{0}{\g}$-modules) and~thus has a highest $\ell$-weight vector $v\in V$ by Theorem \ref{thm:ClassSimp}. Furthermore, by \eqref{eq:TriangularDec},
$$ \dim(\uqmu{0,-}{\g}\cdot v) = \dim V = 2 $$ 
and it follows that we can fix $i\in I$ and $r\in \Z$ such that $w=x_{i,r}^-v\neq 0$. Note $\gamma=(\gamma_i)_{i\in I}\in\mathfrak{t}^{\times}$ the weight of $v$ and suppose $\g\not\simeq \SL$. Then, as $\g$ is simple, there exists $j\in I\backslash\{i\}$ with $C_{j,i}\neq 0$. Moreover, $w=x_{i,r}^-v$ has weight $\gamma[-\alpha_i]\neq \gamma$ by \eqref{eq:xWt} and $\{v,w\}$ is in particular a basis of $V$. Hence, again by \eqref{eq:xWt},
$$x_{j,0}^{\pm}V = x_{j,0}^{\pm}(V_{\gamma}\oplus V_{\gamma[-\alpha_i]}) \subseteq V_{\gamma[\pm \alpha_j]}\oplus V_{\gamma[-\alpha_i\pm \alpha_j]} = 0$$ 
so that $\gamma_j^2 = 1$ and $q_j^{2C_{j,i}}=1$ since, by \eqref{eq:Relxpxmphi},
$$ 0 = (q_j-q_j^{-1})[x_{j,0}^+,x_{j,0}^-]v = (\phi_{j,0}^+-\phi_{j,0}^-)v = (\gamma_j-\gamma_j^{-1})v$$
and
$$ 0 = (q_j-q_j^{-1})[x_{j,0}^+,x_{j,0}^-]w = (\phi_{j,0}^+-\phi_{j,0}^-)w = (\gamma_jq_j^{-C_{j,i}}-\gamma_j^{-1}q_j^{C_{j,i}})w.\vspace*{1mm}$$
This however contradicts $C_{j,i}\neq 0$ (since $q$ is not a root of unity) and thus ends the proof.
\end{proof}\vspace*{-1.5mm}
Note that the above proof actually shows that the restriction functor $\res_J^{0}$ is in general not essentially surjective on finite-dimensional simple objects of $\OO^0_J$. Nevertheless, as we will~show in the next sections, this surjectivity issue can be resolved by considering instead the functor
$$\textstyle \mathscr{R}_J: \OO^{sh}\arr \bigoplus_{\nu\in\Lambda^{\vee}_J} (\uqmu{\nu}{\g_J}\text{--Mod})$$
which is given, for $V$ in $\OO^{\mu}$, by 
$$\mathscr{R}_J(V) = \res_J^{\mu}(V)$$ 
and is essentially surjective on finite-dimensional simple objects.\vspace*{-2mm}
\subsection{$q$-characters and connections with other research}\label{sec:Inflqchar} Take $J\subseteq I$  and recall the map $\res_J:\mathfrak{r}\arr \mathfrak{r}_J$ of Section \ref{sec:Fusion}. Consider in addition the injective ring morphism
$$ \iota_J:\Z[[\res_J(A_{j,a}^{-1})]]_{j\in J,a\in \C^{\times}}\arr\Z[[A_{i,a}^{-1}]]_{i\in I,a\in \C^{\times}} $$ 
given by $\res_J(A_{j,a}^{-1})\mapsto A_{j,a}^{-1}$ for $j\in J$ and $a\in\C^{\times}$. This map induces another ring morphism 
$$\varpi(\iota_J): \Z[[\,\res_J([-\alpha_j])\,]]_{j\in J}\arr \Z[[\,[-\alpha_i]\,]]_{i\in I}$$
such that $\res_J([-\alpha_j]) \mapsto [-\alpha_j]$ for all $j\in J$. This second morphism is also injective. \par
Take now $\nu\in \Lambda_J^{\vee}$ and $\Psi^J\in\mathfrak{r}^J$ with $W$ a $\uqmu{\nu}{\g_J}$-module of highest $\ell$-weight $\Psi^J$. We have the following important result about the structure of $q$-characters for $J$-inflations of $W$ to $\g$.
\begin{prop}\label{prop:IfVInflqCar} Fix $\mu\in\Lambda^{\vee}$ such that $\res_J(\mu)=\nu$ and suppose that an object $V$ of $\OO^{\mu}$ is a $J$-inflation of $W$ to $\g$. Then, 
 \begin{itemize}
 \item[(i)] $V$ has highest $\ell$-weight $\Psi = \Psi^{J}\Psi_{p}$ with $\Psi_p\in\mathfrak{r}_{\mu-\nu}$ a product of various $\Psi_{i,a}$ and $[a\omega_i]$ for $i\not\in J$ and $a\in \C^{\times}$ (and with $\Psi^{J}$ seen inside $\mathfrak{r}_{\mu}$ by the obvious inclusion $\mathfrak{r}_{\nu}^{J}\subseteq \mathfrak{r}_{\mu}$).
\item[(ii)] The normalized characters $\overline{\chi}(V)$ and $\overline{\chi}(W)$ satisfy $$\overline{\chi}(V)=\varpi(\iota_J)(\overline{\chi}(W)).$$
\end{itemize}
Moreover, if $W$ is simple, then (ii) can be lifted to an equality $\overline{\chi_q}(V) = \iota_J(\overline{\chi_q}(W))$.
\end{prop}\vspace*{-1.5mm}
\begin{proof}
Fix some highest $\ell$-weight vector $v\in W$. Then $v$ is also a highest $\ell$-weight vector~when seen inside the $J$-inflation $V$ by Lemma \ref{lem:HighestWeightRest}. Note $\Psi = (\Psi_{i}(z))_{i\in I}\in \mathfrak{r}_{\mu}$ the $\ell$-weight of $v$ in $V$. By the discussion above Theorem \ref{thm:qCharAJ}, $\res_{J}(\Psi) = \Psi^J$. Take now $i\not\in J$. Then, as $x_i^{\pm}(z)V = 0$, \eqref{eq:Relxpxmphi} gives, for all $r\in \Z$,
$$ (\phi_{i,r}^+-\phi_{i,r}^-)v= (q_i-q_i^{-1})[x_{i,r}^+,x_{i,0}^-]v = 0$$
and it follows in particular that $\phi_{i,r}^+v = \phi_{i,r}^-v = 0$ for $r>\alpha_i(\mu)$ with $\phi_{i,\alpha_i(\mu)}^+v=\phi_{i,\alpha_i(\mu)}^-v\neq 0$. The rational function $\Psi_i(z)$ is hence a polynomial of degree exactly $\alpha_i(\mu)$ (with $\alpha_i(\mu)\geq 0$
) and (i) easily follows. For (ii), fix $\gamma\in P(V)$ and let $w\in V_{\gamma}$. Then $w=xv$ with $x\in \uqmu{\nu,-}{\g_J}$~(since $V=W$ as vector spaces with $v$ the highest $\ell$-weight vector of $W$) and it follows that
$$\gamma = \varpi(\Psi)[-\alpha]
$$
for some $\alpha\in Q_J^+$. In this situation, using the properties of the map $\res_J$ and condition (i),~we conclude that $w$ has weight
$$\res_J(\gamma)=\res_J(\varpi(\Psi))\res_J([-\alpha])=\varpi(\res_J(\Psi))\res_J([-\alpha]) = \varpi(\Psi^J)\res_J([-\alpha])$$ when seen inside $W$. This easily implies (ii) (as $(\varpi(\iota_J)\circ\res_J)([-\alpha])=[-\alpha]$ because $\alpha\in Q_J^+$).\par
To finish the proof, suppose that $W$ is simple and choose $\Psi'$ a $\ell$-weight of $V$ with $v'\in V_{\Psi'}$. Then, by Theorem \ref{thm:qCharA}, $\Psi' = \Psi M$ for some monomial $M$ in the variables $\{A_{i,b}^{-1}\}_{i\in I,b\in\C^{\times}}$. By (ii) and since the simple roots are free in the root lattice $Q$, we can deduce that $M$ is actually a monomial in the variables $\{A_{j,b}^{-1}\}_{j\in J,b\in\C^{\times}}$. Hence, $\iota_J\circ \res_J(M) = M$ and the desired result follows since $v'$ has $\ell$-weight \vspace*{-0.5mm} 
$$ \res_J(\Psi') = \Psi^J\res_J(M) = \Psi^J M$$\vspace*{-5mm}\\
when seen inside $W$.
\end{proof}\vspace*{-1.5mm}
The proposition below gives a result in the opposite direction. In this proposition, we use the terminology $J$\textit{-trivial} for a $\ell$-weight $\Psi=(\Psi_i(z))_{i\in I}\in \mathfrak{r}$ such that $\Psi_j(z)=1$ for all $j\in J$. We also use the same initial data $(\nu,\Psi^J,W)$ as for Proposition \ref{prop:IfVInflqCar}.
\begin{prop} Fix again $\mu\in\Lambda^{\vee}$ such that $\res_J(\mu)=\nu$ with an object $V$ of $\OO^{\mu}$. Suppose that $\overline{\chi}(V)=\varpi(\iota_J)(\overline{\chi}(W))$. Then, $V$ is a $J$-inflation of $W$ if either 
\begin{equation}
\res_J^{\mu}(V)\simeq W\text{ in }\OO^{\nu}_J \tag{$\diamond$}
\end{equation}
or if
\begin{equation}\label{eq:ConditionStar}
W\text{ is simple with }{\chi}_q(V)=\iota_J(\overline{\chi_q}(W))[\Psi^J\Psi_p]\text{ for some }J\text{-trivial }\Psi_p\in \mathfrak{r}_{\mu-\nu}. \tag{$\star$}
\end{equation}
\end{prop}\vspace*{-2.9mm}
\begin{proof} Fix $i\not\in J$ with $\gamma\in P(V)$ and suppose that $x_{i,r}^{+}v'\neq 0$ or $x_{i,r}^{-}v'\neq 0$ for some $r\in \Z$ and some $v'\in V_{\gamma}$. Then, the set $\{\gamma[\alpha_i],\gamma[-\alpha_i]\}\cap P(V)$ is not empty and it follows easily~from~the equality $\overline{\chi}(V)=\varpi(\iota_J)(\overline{\chi}(W))$ that \vspace*{-0.5mm}
$$[\alpha_i] \in P(V)(P(V))^{-1} \subseteq [Q_J],$$\vspace*{-5mm}\\
contradicting the injectivity of the map $[\overline{\ }]:\Lambda_{\mathbb{Q}}\arr \mathfrak{t}^{\times}$. In particular, $x_i^{\pm}(z)V = 0$ for all $i\not\in J$ and it thus only remains to prove that $\res_J^{\mu}(V)\simeq W$ in $\OO^{\nu}_{J}$ assuming \eqref{eq:ConditionStar}. For this goal, use \eqref{eq:chiRes} and the $J$-triviality of $\Psi_p$ to deduce that
 $$\chi_q(\res_J^{\mu}(V)) = \res_J(\chi_q(V)) = (\res_J\circ\iota_J)(\overline{\chi_q}(W))[\res_J(\Psi^J)\res_J(\Psi_p)] = \overline{\chi_q}(W)[\Psi^J] =\chi_q(W)$$
so that $\res_J^{\mu}(V)\simeq W$ in $\OO^{\nu}_{J}$ by injectivity of the $q$-character map and irreducibility of $W$.
\end{proof}\vspace*{-1mm}
The above results give rise to an equivalent characterization of inflations of highest~$\ell$-weight modules. Note that we use again the terminology $J$\textit{-trivial $\ell$-weights} and the data $(\nu,\Psi^J,W)$.
\begin{thm}\label{thm:EquivalentDefInf} Fix $\mu\in\Lambda^{\vee}$ such that $\res_J(\mu)=\nu$ with $V$ in $\OO^{\mu}$. 
\begin{itemize} 
\item[\normalfont{\textbf{(1)}}] Suppose that $\res_J^{\mu}(V)\simeq W$ in $\OO^{\nu}_{J}$. Then, $V$ is a $J$-inflation of $W$ to $\g$ if and only if
\begin{center}
$\overline{\chi}(V)=\varpi(\iota_J)(\overline{\chi}(W))$.
\end{center}
\item[\normalfont{\textbf{(2)}}] Suppose that $W$ is simple. Then, $V$ is a $J$-inflation of $W$ to $\g$ if and only if
\begin{center}
$\chi_q(V)
=\iota_J(\overline{\chi_q}(W))[\Psi^J\Psi_p]$ for some $J$-trivial $\Psi_p\in \mathfrak{r}_{\mu-\nu}$.
\end{center}
\end{itemize}
\end{thm}\vspace*{-2mm}
\begin{proof}
This follows easily from the above propositions.
\end{proof}\vspace*{-1.5mm}
A first, unchallenging, consequence of the previous results is the following explicit characterization of inflations for the trivial representation $L^J(\mathbbm{1})$ of $\OO^{0}_J$ (with $\mathbbm{1} = (1)_{j\in J}\in\mathfrak{r}_0^J$).
\begin{cor}\label{cor:InfTrivial} An object $V$ in $\OO^{sh}$ is a $J$-inflation to $\g$ of the representation $L^J(\mathbbm{1})$ of $\mathcal{O}^0_J$ if and only if $V \simeq L(\Psi)$ with $\Psi\in\mathfrak{r}$ a product of various $[a\omega_i]$ and $\Psi_{i,a}$ for $i\not\in J$ and $a\in\C^{\times}$.
\end{cor}\vspace*{-2mm}
\begin{proof} Sufficiency clearly follows from Proposition \ref{prop:IfVInflqCar}. For necessity, suppose that $V\simeq L(\Psi)$ for some product $\Psi$ as in the statement of the corollary. Then, $\dim V=1$ (see Theorem \ref{thm:PrefundSimpFus} or use the discussion before Theorem \ref{thm:PrefundMonGen}) and $ \chi_q(V) = [\Psi] $ where $\Psi$ is a $J$-trivial $\ell$-weight.~This implies the desired result by Theorem \ref{thm:EquivalentDefInf}.
\end{proof}\vspace*{-1mm}
The above results can also be used to establish relations for the ring $K_0(\OO^{sh})$ by ``\textit{inflating}" relations of $K_0(\OO^{sh}_J)$. This is illustrated in the examples below. Observe however that~a~proper result characterizing which relations in $K_0(\OO_J^{sh})$ can be ``\textit{inflated}" has not yet been developed\footnote{A partial result in this direction is given by the upcoming Proposition \ref{prop:InfHighest} (see also Theorem \ref{thm:InfHighestDual}). This result includes the cases studied in Examples \ref{ex:QQ}, \ref{ex:QQ2} and \ref{ex:newT} (cf.~Example \ref{ex:QQSEC}).}.
\begin{example}\label{ex:QQ} Fix $j\in I$ and let $J=\{j\}$. Then, for $\mu\in \Lambda^{\vee}$, $\nu=\res_J(\mu) = \alpha_j(\mu) \omega_j^{\vee}$ and $$\uqmu{\nu}{\g_J}\simeq U_{q_j}^{\nu}(\mathfrak{sl}_2).$$ 
In particular, for $a\in \C^{\times}$, the quantum Wronskian relation
\begin{equation}\label{eq:QQWronsk} [L^J(\Psi_{j,a})][L^J(\Psi_{j,a}^{-1})]=1+[-2\omega_j][L^J(\Psi_{j,aq_j^2})][L^J(\Psi^{-1}_{j,aq_j^{-2}})]
\end{equation} \vspace*{-3.5mm}\\
holds in $K_0(\OO^{sh}_J)$ (see Example \ref{ex:qWronsk} for the notation used). Observe that this relation can be deduced from the explicit $q$-character formula of Example \ref{ex:qCharEX}. \par 
Consider now the $Q\widetilde{Q}$-system of \cite[Example 5.9(ii)]{hshift} (see also \cite{fh2}) for $K_0(\OO^{sh})$, i.e.
\begin{equation}\label{eq:QQ}
[L(\Psi_{j,a})][L(\widetilde{\Psi}_{j,a})]=[L(\Psi_p)]+[-\alpha_j][L(\Psi_{j,aq_j^2})][L(\widetilde{\Psi}_{j,aq_j^{-2}})]
\end{equation}\vspace*{-3.25mm}\\
with $\widetilde{\Psi}_{j,a}$ the $\ell$-weight of Example \ref{ex:InflPrefondSL} and with $\Psi_p =\Psi_{j,a}\widetilde{\Psi}_{j,a}$. This equation was introduced~by \cite{mrv1,mrv2,mv1,mv2} in the context of Bethe Ansatz equations and affine ${}^L\hat{\g}$-opers (with ${}^L\hat{\g}$ the Langlands dual of $\hat{\g}$). Remark that, by Examples \ref{ex:InvertiblePositivePrefond}--\ref{ex:InflPrefondSL},
$$L(\Psi_{j,a}),\ L(\widetilde{\Psi}_{j,a}),\ L(\Psi_{j,aq_j^2})\text{ and }L(\widetilde{\Psi}_{j,aq_j^{-2}})$$ \vspace*{-3.25mm}\\
are respectively $J$-inflations to $\g$ of 
$ L^J(\Psi_{j,a}),\ L^J(\Psi_{j,a}^{-1}),\ L^J(\Psi_{j,aq_j^2})\text{ and }L^J(\Psi_{j,aq_j^{-2}}^{-1})$. \par In addition, $L([-\alpha_j])$ is easily seen to be a $J$-inflation to $\g$ of $L^J([-2\omega_j])$ and, since
$$\textstyle \Psi_p = \left(\prod_{i\in I\, C_{j,i}=-1}\Psi_{i,aq_j}\right)\left(\prod_{i\in I\, C_{j,i}=-2}\Psi_{i,a}\Psi_{i,aq_j^2}\right)\left(\prod_{i\in I\, C_{j,i}=-3}\Psi_{i,aq_j^{-1}}\Psi_{i,aq_j}\Psi_{i,aq_{j}^3}\right),$$
the object $L(\Psi_p)$ is a $J$-inflation of the trivial representation $L^J(\mathbbm{1})$ of $\OO^0_{J}$ (cf.~Corollary \ref{cor:InfTrivial}). Relation \eqref{eq:QQ} can therefore be seen as an ``\textit{inflation}" of \eqref{eq:QQWronsk}. Moreover, this relation can be shown in an elementary way using Example \ref{ex:qCharEX}, Theorem \ref{thm:HighestWeightFus} and Theorem \ref{thm:EquivalentDefInf}. Indeed,
\begin{align*}
\chi_q(L(\Psi_{j,a})\star L(\widetilde{\Psi}_{j,a})) &=\iota_J\big(\overline{\chi_q}(L^J(\Psi_{j,a}^{-1}))\big)[\widetilde{\Psi}_{j,a}\Psi_{j,a}] \\&= \textstyle [\Psi_{p}]+[\Psi_p]A_{j,a}^{-1}(1+A_{j,aq_j^{-2}}^{-1}(1+\sum_{s\geq 1}A_{j,aq_j^{-4}}^{-1}\dots A_{j,aq_j^{-2s}} ^{-1}))
\\&= \chi_q(L(\Psi_p))+\iota_J\big(\overline{\chi_q}(L^J(\Psi_{j,aq_j^{-2}}^{-1}))\big)[-\alpha_j][\widetilde{\Psi}_{j,aq_j^{-2}}\Psi_{j,aq_j^{2}}]
\\&= \chi_q(L(\Psi_p)) + \chi_q([-\alpha_j]\star L(\Psi_{j,aq_j^2})\star L(\widetilde{\Psi}_{j,aq_j^{-2}}))
\end{align*}
which implies \eqref{eq:QQ} by injectivity of the $q$-character map $\chi_q$ (see Theorem \ref{thm:HighestWeightFus}).
\end{example}
\begin{example}\label{ex:QQ2} Let again $J=\{j\}$ for some $j\in I$. Then, for $a\in \C^{\times}$, Baxter's $QT$-relation
\begin{equation}\label{eq:BaxterQT}
[L^J(\Psi_{j,a})][L^J(Y_{j,aq_j^{-1}})] = [\omega_j][L^J(\Psi_{j,aq_j^{-2}})]+[-\omega_j][L^J(\Psi_{j,aq_j^2})]
\end{equation}\vspace*{-3.5mm}\\
holds in $K_0(\OO^{sh}_J)$ (and can be proven using the explicit $q$-character formulas of Example \ref{ex:qCharEX}). Let us compare this relation with the $QQ^*$-system of \cite{hl1}, that is
\begin{equation}\label{eq:QQ2}
[L(\Psi_{j,a})][L(\Psi^*_{j,a})]=[\omega_j][L(\Psi_{p_1})]+[\omega_j-\alpha_j][L(\Psi_{p_2})]
\end{equation}\vspace*{-3.5mm}\\
where $\Psi_{j,a}^*$ is the $\ell$-weight described in Example \ref{ex:InflPrefondSL} and where
$$\textstyle \Psi_{p_1} = [-\omega_j]\Psi_{j,a}\Psi_{j,a}^*=\prod_{i\in I,\,C_{i,j}\neq 0} \Psi_{i,aq_i^{-C_{i,j}}} \text{ with }\Psi_{p_2}=[\alpha_j]\Psi_{p_1}A_{j,a}^{-1} = \prod_{i\in I,\,C_{i,j}\neq 0} \Psi_{i,aq_i^{C_{i,j}}}.$$ 
It is straightforward to show, as in the proof of Corollary \ref{cor:InfTrivial}, that the objects $L(\Psi_{p_1})$, $L(\Psi_{p_2})$ and $L(\omega_j-\alpha_j)$ are respectively $J$-inflations of $L^J(\Psi_{j,aq_j^{-2}})$, $L^J(\Psi_{j,aq_j^2})$ and $L^J([-\omega_j])$.~Relation \eqref{eq:QQ2} can thus be seen as an ``inflation" of \eqref{eq:BaxterQT} by Examples \ref{ex:InvertiblePositivePrefond}--\ref{ex:InflPrefondSL}. A proof of this ``inflated relation" using Example \ref{ex:qCharEX}, Theorem
\ref{thm:HighestWeightFus} and Theorem 
\ref{thm:EquivalentDefInf} is given by 
\begin{align*}
\chi_q(L(\Psi_{j,a})\star L(\Psi^*_{j,a})) &=\iota_J\big(\overline{\chi_q}(L^J(Y_{j,aq_j^{-1}}))\big)[\Psi^*_{j,a}\Psi_{j,a}] \\&= [\omega_j][\Psi_{p_1}](1+A_{j,a}^{-1})=[\omega_j]\chi_q(L(\Psi_{p_1}))+[\omega_j-\alpha_j]\chi_q(L(\Psi_{p_2})).
\end{align*}
Note that this proof, which is similar to the one given in \cite[Section 6.3]{fjmm}, could also have been carried out using the the explicit $q$-character formula of \cite[Example 6.6]{hshift} for $L(\Psi_{j,a}^*)$.
\end{example}
\begin{rem} The $QQ^*$-system \eqref{eq:QQ2} describes the mutation of a remarkable cluster algebra related to the Grothendieck ring of a subcategory $\OO^{sh,+}_{2\Z}{\,\subseteq\,}\OO^{sh}$, see \cite[Section 8.4]{hshift} and \cite{hl1}.
\end{rem}
For $a\in \C^{\times}$ and $k\in\mathbb{N}$, denote by
$$m_{k,a}^{(j)} = Y_{j,a}Y_{j,aq_j^2}\dots Y_{j,aq_j^{2(k-1)}}$$\vspace*{-2.5mm}\\ the highest $\ell$-weight of the Kirillov-Reshitikhin module $W_{k,a}^{(j)}$ of $\OO^{0}$ (see Example \ref{ex:KR}).
\begin{example}[Inflated $T$-system of type A${}_1$]\label{ex:newT} Let again $J=\{j\}$ for some $j\in I$ and define
$$ T_{k,a} = L^J(m_{k,aq_j^{1-2k}}^{(j)}) \text{ with } T_{0,a} = L^J(\mathbbm{1})$$\vspace*{-2.5mm}\\
for $k\in\mathbb{N}_{>0}$ and $a\in\C^{\times}$. Then, in $K_0(\OO^{sh}_J)$, one has the $T$-system (cf.~\cite{cp1,hTsyst})
\begin{equation}\label{eq:Tsyst}
[T_{k,a}][T_{k,aq_j^{-2}}]=[T_{k+1,a}][T_{k-1,aq_j^{-2}}]+1.
\end{equation}\vspace*{-3mm}\\
Recall the $\ell$-weight $\widetilde{\Psi}_{j,a}$ of Example \ref{ex:InflPrefondSL} and let $\Psi_p(a)=\Psi_{j,a}\widetilde{\Psi}_{j,a}$. Set $V_{0,a} = L(\Psi_p(a))$~with \vspace*{-1mm} $$V_{k,a}=L([k\omega_j]\widetilde{\Psi}_{j,a}\Psi_{j,aq_j^{-2k}})=L(m_{k,aq_j^{1-2k}}^{(j)}\Psi_p(a))$$\vspace*{-3mm}\\ for $k\in\mathbb{N}_{>0}$. Then it follows easily from \cite[Example 8.3]{hshift} that $\chi_q(V_{0,a}) = [\Psi_p(a)]$ and that
$$ \chi_q(V_{k,a}) 
= \iota_J\big(\overline{\chi_q}(T_{k,a})\big)[m_{k,aq_j^{1-2k}}^{(j)}\Psi_{p}(a)]$$\vspace*{-3mm}\\
if $k\in \mathbb{N}_{>0}$. Hence, $V_{k,a}$ is a $J$-inflation of $T_{k,a}$ by Theorem \ref{thm:EquivalentDefInf} (and by $J$-triviality of $\Psi_p(a)$, see Example \ref{ex:QQ}). Also, by Example \ref{ex:qCharEX} and Theorem \ref{thm:HighestWeightFus}, fixing $k\in \mathbb{N}_{>0}$, we get that
\begin{align*}
\chi_q(V_{k,a}\star V_{k,aq_j^{-2}}) &= \iota_J\big(\overline{\chi_q}(T_{k,a})\big)\iota_J\big(\overline{\chi_q}(T_{k,aq_j^{-2}})\big)[2k\omega_j][\widetilde{\Psi}_{j,a}\Psi_{j,aq_j^{-2k}}\widetilde{\Psi}_{j,aq_j^{-2}}\Psi_{j,aq_j^{-2(k+1)}}]
\end{align*}
is equal to 
\begin{equation}\label{eq:qCarnewTsystPartial}
\textstyle\big(1+\sum_{s=0}^{k-1}A_{j,a}^{-1}\dots A_{j,aq_j^{-2s}}^{-1}\big)\big(1+\sum_{s=0}^{k-1}A_{j,aq_j^{-2}}^{-1}\dots A_{j,aq_j^{-2(s+1)}}^{-1}\big)[2k\omega_j][\Psi]
\end{equation}\vspace*{-3mm}\\
where $\Psi = \widetilde{\Psi}_{j,a}\Psi_{j,aq_j^{-2k}}\widetilde{\Psi}_{j,aq_j^{-2}}\Psi_{j,aq_j^{-2(k+1)}}$. \par However, a straightforward computation allows us to write \eqref{eq:qCarnewTsystPartial} as the sum of
{\small
\begin{align*}
\chi_q(V_{k+1,a}\star V_{k-1,aq_j^{-2}})&=\iota_J(\overline{\chi_q}(T_{k+1,a}))(\delta_{k,1}+\iota_J(\overline{\chi_q}(T_{k-1,aq_j^{-2}}))(1-\delta_{k,1}))[2k\omega_j][\Psi]
\\
&=\textstyle \big(1+\sum_{s=0}^{k}A_{j,a}^{-1}\dots A_{j,aq_j^{-2s}}^{-1}\big)\big(1+\sum_{s=0}^{k-2}A_{j,aq_j^{-2}}^{-1}\dots A_{j,aq_j^{-2(s+1)}}^{-1}\big)[2k\omega_j][\Psi]
\end{align*}}\noindent
with \vspace*{-0.25mm}
$$ [k(2\omega_j-\alpha_j)][\Psi_p(a)\Psi_p\big(aq_j^{-2(k+1)}\big)]
=A_{j,aq_j^{-2}}^{-1}\dots A_{j,aq_j^{-2k}}^{-1}[2k\omega_j][\Psi] $$\vspace*{-3.25mm}\\
which is the highest $\ell$-weight of a $J$-inflation of the trivial representation of $\OO^{sh}_J$ by Corollary \ref{cor:InfTrivial}. Therefore, in $K_0(\OO^{sh})$, \eqref{eq:Tsyst} gives rise by ``inflation" to the relation\vspace*{-0.5mm}
\begin{equation}\label{eq:newTsyst} 
[V_{k,a}][V_{k,aq_j^{-2}}]=[V_{k+1,a}][V_{k-1,aq_j^{-2}}]+[k(2\omega_j-\alpha_j)][L(\Psi_p(a))][L(\Psi_p(aq_j^{-2(k+1)}))] 
\end{equation}\vspace*{-3.75mm}\\
that we call ``inflated T-system of type A${}_1$".
\end{example}\vspace*{-0.5mm}
We write the results of the above example in a proper theorem.
\begin{theorem}\label{thm:newT}
Fix $j\in I$. Then, in $K_0(\OO^{sh})$, for $a\in \C^{\times}$ and $k\in \mathbb{N}_{>0}$,\vspace*{-0.5mm}
$$ [L(\widetilde{\Psi}_{j,a}\Psi_{j,aq_j^{-2k}})][L(\widetilde{\Psi}_{j,aq_j^{-2}}\Psi_{j,aq_j^{-2(k+1)}})] - [L(\widetilde{\Psi}_{j,a}\Psi_{j,aq_j^{-2(k+1)}})][L(\widetilde{\Psi}_{j,aq_j^{-2}}\Psi_{j,aq_j^{-2k}})]$$\vspace*{-3mm}\\
is equal to the isoclass $[-k\alpha_j][L(\Psi_p(a))][L(\Psi_{p}(aq_j^{-2(k+1)}))]$ where $\Psi_p(a) = \widetilde{\Psi}_{j,a}\Psi_{j,a}$.
\end{theorem}
\begin{rem} We will talk about a potential application of the above theorem to the study~of ``compatible cluster structures" over Grothendieck rings in Section \ref{sec:applications}. 
\end{rem}\vspace*{-0.75mm}
Let us at long last end this subsection by noting that Theorem \ref{thm:EquivalentDefInf} above paves the way for a comparison between our upcoming results about inflations of finite-dimensional modules of $\OO^{sh}_J$ (see Section \ref{sec:InflPrefund}) and \cite[Proposition 6.9]{fjmm}. Indeed, the latter proposition of \cite{fjmm} shows that, given $J\subseteq I$ and a \textit{finite-type module} $W$ over a subalgebra $U_q(\mathfrak{\hat{b}}_J)$ of $U_q(\mathfrak{\hat{b}})$ (which corresponds in some sense to the inclusion $J\subseteq I$), there is a \textit{finite-type} $U_q(\mathfrak{\hat{b}})$-module $V$ such that the \textit{essential $q$-characters} of $V$ and $W$ agree (up to an inclusion resembling our map $\iota_J$). Moreover, the proof of this proposition implies that the $U_q(\mathfrak{\hat{b}})$-module $V$ can be chosen to be simple when $W$ is itself simple. Now, using the relation between the representation theory of shifted quantum affine algebras and that of Borel quantum loop algebras (see Theorem \ref{thm:ClassSimpDimFin}, \cite[Theorem 8.1]{hshift} and \cite[Theorem 6.1]{fjmm}), one can identify \textit{finite-type} representations of Borel algebras with finite-dimensional objects inside $\OO^{sh}$ as well as \textit{essential $q$-characters} for $U_q(\mathfrak{\hat{b}})$ with normalized $q$-characters for shifted quantum affine algebras. Thus, Theorem \ref{thm:EquivalentDefInf} indicates that a method for constructing inflations for finite-dimensional objects of $\OO^{sh}_J$ could possibly be deduced from the proof of \cite[Proposition 6.9]{fjmm}. Such an deduction is in fact exactly what we will aim for in Section \ref{sec:InflPrefund}. Note nevertheless that the proofs of \cite{fjmm} do not apply directly to our case because of the following two reasons: 
\begin{itemize}
\item[(1)] The simple $U_q(\mathfrak{\hat{b}})$-module $V$ (obtained from a simple $U_q(\mathfrak{\hat{b}}_J)$-module $W$) following the proof of \cite[Proposition 6.9]{fjmm} need not satisfy the conditions imposed by Theorem \ref{thm:EquivalentDefInf} on the highest $\ell$-weight of $J$-inflations.
\item[(2)] The inclusion $U_q(\mathfrak{\hat{b}}_J)\subseteq U_q(\mathfrak{\hat{b}})$ considered by \cite{fjmm} may not be compatible with our inclusions $\uqmu{\nu}{\g_J}\subseteq \uqmu{\mu}{\g}$ (for $\mu\in \Lambda^{\vee}$ with $\res_J(\mu)=\nu$). In fact, this does not seem to be the case as one can\,deduce\,with the remarks\footnote{We note that there is probably a mistake in the definition given in \cite{fjmm} for the algebra $U_q(\mathfrak{\hat{b}}_J)$ as the subset $J\subseteq I=\{1,\dots,n\}$ considered in that paper does not contain the \textit{affine node} $0$ (and this easily implies that all simple modules over their algebra $U_q(\mathfrak{\hat{b}}_J)$ would have dimension 1).} following \cite[Proposition 4.1]{fjmm}.
\end{itemize}
The adaptation of the proof of \cite[Proposition 6.9]{fjmm} to the study of inflations must hence be done cautiously. Note that our adaptation (that we give in Section \ref{sec:InflPrefund}) uses a more~explicit approach than the one used in \cite{fjmm} even if the general underlying idea is the same. Also, our proof gives a more general result in type A--B where it shows the existence of $J$-inflations for any\footnote{In fact, we conjecture in Section \ref{sec:InflPrefund} that our proof leads to such a result in any type (i.e.~not only A--B).} (finite-dimensional or not) simple object of $\OO^{sh}_J$.
\subsection{Compatibility with fusion product}\label{sec:Existence} This subsection shows that it is possible to take $J$-inflations of highest $\ell$-weight objects in $\OO^{sh}_J$ in a manner which is compatible with the fusion product of Section \ref{sec:Fusion}. We also start our quest for general existence theorems for $J$-inflations and give a first result when $\g_J$ is isomorphic to a direct sum of (possibly many) copies of $\SL$ and $\SLt$. We begin with the result below and fix a subset $J\subseteq I$ for the rest of the subsection. Denote by $\iota_J^{\mu}$ the canonical inclusion of $\uqmu{\nu}{\g_J}$ in $\uqmu{\mu}{\g}$ for $\nu=\res_J(\mu)$ (as in Section \ref{sec:Def}).
\begin{prop}\label{prop:FusRest} Fix $\mu_1,\dots,\mu_m\in \Lambda^{\vee}$. For all $k\in\{1,\dots,m\}$, let also $\nu_k=\res_J(\mu_k)$ and fix a highest $\ell$-weight $\uqmu{\mu_k}{\g}$-module $V_k$ for which $x_i^{\pm}(z)V_k=0$ if $i\not\in J$. Then, in $\OO^{\nu_1+...+\nu_m}_{J}$,
$$ \res_J^{\mu_1+...+\mu_m}(V_1\star \dots\star V_m)\simeq \res_J^{\mu_1}(V_1)\star \dots \star\res_J^{\mu_m}(V_m). $$
\end{prop}
\begin{rem} Let $\Delta_{\mu_1,\mu_2}^{(u)}$ and $\Delta_{\nu_1,\nu_2}^{(u),J}$ be the Drinfeld coproducts of $\uqmu{\mu_1+\mu_2}{\g}$ and $\uqmu{\nu_1+\nu_2}{\g_J}$ (respectively, see Section \ref{sec:Fusion}). These are applications of the form 
$$\uqmu{\mu_1+\mu_2}{\g}\arr(\uqmu{\mu_1}{\g}\otimes \uqmu{\mu_2}{\g})((u))\text{ and }\uqmu{\nu_1+\nu_2}{\g_J}\arr(\uqmu{\nu_1}{\g_J}\otimes \uqmu{\nu_2}{\g_J})((u))$$
that are easily proven to be compatible with the inclusions $\iota_J^{\mu}$ (and with their obvious analogs for the formal Laurent power series algebras) in the sense that
\begin{equation}\label{eq:CompDcopInc}
\Delta_{\mu_1,\mu_2}^{(u)}\circ \iota_J^{\mu_1+\mu_2}=(\iota_J^{\mu_1}\otimes\iota_J^{\mu_2})\circ\Delta_{\nu_1,\nu_2}^{(u),J}.
\end{equation}
We therefore have an equality of $\uqmu{\nu_1+\nu_2}{\g_J}$-modules $$\res_J^{\mu_1+\mu_2}((V_1\otimes V_2)((u)))=(\res_J^{\mu_1}(V_1)\otimes \res_J^{\mu_2}(V_2))((u))$$
that clearly gives an equality between the submodules of rational Laurent power series, i.e. \begin{equation}\label{eq:EqualityRLPS}
\res_J^{\mu_1+\mu_2}((V_1\otimes V_2)(u))=(\res_J^{\mu_1}(V_1)\otimes \res_J^{\mu_2}(V_2))(u)
\end{equation}
as $\uqmu{\nu_1+\nu_2}{\g_J}$-modules. This will be used implicitly in the proof below.
\end{rem}
\begin{rem} Suppose that $V_2$ is a $1$-dimensional $\uqmu{\mu_2}{\g}$-module and denote by $V$ the module over $\uqmu{\mu_1+\mu_2}{\g}$ defined on the space $V_1\otimes V_2$ by specializing $\Delta_{\mu_1,\mu_2}^{(u)}$ at $u=1$ (cf.~Remark \ref{rem:Dcoproduct}). Denote also by $V_J$ the $\uqmu{\nu_1+\nu_2}{\g_J}$-module defined in the same way using the coproduct $\Delta_{\nu_1,\nu_2}^{(u),J}$ and the modules $\res_J^{\mu_1}(V_1)$ and $\res_J^{\mu_2}(V_2)$. Then, \eqref{eq:CompDcopInc} directly implies that, in $\OO^{sh}_J$,
$$\res_J^{\mu_1+\mu_2}(V)\simeq V_J.$$
\end{rem}
We now proceed with the proof of Proposition \ref{prop:FusRest}.\vspace*{-2mm}
\begin{proof}
Assume that $m=2$ and fix highest $\ell$-weight vectors $v_1\in V_1$ and $v_2\in V_2$. Then,~$v_1$~and $v_2$ are also respectively highest $\ell$-weight vectors in $\res_J^{\mu_1}(V_1)$ and $\res_J^{\mu_2}(V_2)$ by Lemma \ref{lem:HighestWeightRest}. \par Write
$$ X_{\mathscr{A}} = (\mathscr{A}\otimes \uqmu{\mu_1+\mu_2}{\g})\cdot(v_1\otimes v_2)\text{ and } X_{\mathscr{A}}^J = (\mathscr{A}\otimes \uqmu{\nu_1+\nu_2}{\g_J})\cdot(v_1\otimes v_2)$$
with $\mathscr{A} = \{f(u)\in\C(u)\,|\,f(u)\text{ is regular at }u=1\}$ as in Section \ref{sec:Fusion}. Then
$$ V_1\star V_2 = X_{\mathscr{A}}/((u-1)X_{\mathscr{A}})\text{ and } \res_J^{\mu_1}(V_1)\star \res_J^{\mu_2}(V_2) = X_{\mathscr{A}}^J/((u-1)X_{\mathscr{A}}^J).$$
We claim that, as representations of $\mathscr{A}\otimes \uqmu{\nu_1+\nu_2}{\g_J}$, $$\res_J^{\mu_1+\mu_2}(X_{\mathscr{A}})=X^J_{\mathscr{A}}.$$ 
To prove this, note that \eqref{eq:TriangularDec}, \eqref{eq:TriangularDecJ} and the fact that $x_i^{\pm}(z)V_1= x_i^{\pm}(z)V_2=0$ for $i\not\in J$ imply that, as $\mathscr{A}$-modules, by definition of the Drinfeld coproduct,
$$ X_{\mathscr{A}} = (\mathscr{A}\otimes \uqmu{\mu_1+\mu_2,-}{\g})\cdot (v_1\otimes v_2) = (\mathscr{A}\otimes \uqmu{\nu_1+\nu_2,-}{\g_J})\cdot (v_1\otimes v_2) = X^J_{\mathscr{A}}.$$
Our claim thus follows easily as $X_{\mathscr{A}}^{J}$ is clearly a $(\mathscr{A}\otimes \uqmu{\nu_1+\nu_2}{\g_J})$-submodule of $\res_J^{\mu_1+\mu_2}(X_{\mathscr{A}})$. Now, by exactness of the restriction functor $\res_J^{\mu_1+\mu_2}$, we get that, as $\uqmu{\nu_1+\nu_2}{\g_J}$-modules,
\begin{align*}
\res_J^{\mu_1+\mu_2}(V_1\star V_2) &= \res_J^{\mu_1+\mu_2}({\sfrac{X_{\mathscr{A}}}{(u-1)X_{\mathscr{A}}}}) \simeq {\Large\sfrac{\res_J^{\mu_1+\mu_2}(X_{\mathscr{A}})}{(u-1)\res_J^{\mu_1+\mu_2}(X_{\mathscr{A}})}} \\&= \sfrac{X^J_{\mathscr{A}}}{(u-1)X^J_{\mathscr{A}}} = \res_J^{\mu_1}(V_1)\star \res_J^{\mu_2}(V_2).
\end{align*}
The proposition is therefore true if $m\leq 2$ and follows by simple induction for all $m\in \mathbb{N}_{>0}$.
\end{proof}\vspace*{-2mm}
\begin{cor}\label{cor:FusInfl} Fix $\nu_1,\dots,\nu_m\in \Lambda^{\vee}_J$. For all $k\in \{1,\dots,m\}$, fix also $W_k$ a highest $\ell$-weight $\uqmu{\nu_k}{\g_J}$-module and $V_k$ a $J$-inflation of $W_k$. Then $V_1\star \dots \star V_m$ is a $J$-inflation of $W_1\star \dots \star W_m$. 
\end{cor}\vspace*{-2mm}
\begin{proof}
This follows from Proposition \ref{prop:FusRest} and the definition of the Drinfeld coproduct. 
\end{proof}\vspace*{-2mm}
Thus, taking $J$-inflations of highest $\ell$-weight objects of $\OO^{sh}_J$ is compatible with fusion products. It is also compatible with the notion of Jacobson radical as proven below. We will need the $Q$-graduation of $\uqmu{\mu}{\g}$ given by 
$ \deg_Q(x_{i,r}^{\pm}) = \pm \alpha_i$ and $\deg_Q(\phi_{i,r}^{\pm}) = 0$
for $i\in I$ and $r\in \Z$. By \eqref{eq:xWt}, this graduation is such that, for $V$ in $\OO^{\mu}$, $\gamma\in \mathfrak{t}^{\times}$ and $\alpha\in Q$, \vspace*{-0.15mm}
\begin{equation}\label{eq:QDeg}
xV_{\gamma}^{\pm} \subseteq V_{\gamma[\alpha]}^{\pm}
\end{equation}\vspace*{-4.1mm}\\
for all $Q$-homogeneous elements $x\in \uqmu{\mu}{\g}$ of $Q$-degree $\deg_Q(x) =\alpha$. 
\begin{prop}\label{prop:InfHighest}
Take $\nu\in \Lambda_J^{\vee}$ with $W$ a highest $\ell$-weight $\uqmu{\nu}{\g_J}$-module. Take also $\mu\in \Lambda^{\vee}$ such that $\res_J(\mu)=\nu$ and suppose that $V$ is a $J$-inflation of $W$ to $\g$ with coweight $\mu$. Then, $\rad V$ and $\topp V$ are respectively $J$-inflations to $\g$ (with coweight $\mu$) of $\rad W$ and $\topp W$.
\end{prop}\vspace*{-2mm}
\begin{proof}
Fix a highest $\ell$-weight vector $w\in W$. Then, by Lemma \ref{lem:HighestWeightRest}, $w$ is also a highest $\ell$-weight vector in the $J$-inflation $V$ and it follows from classical highest weight theory that 
$$ \rad V = \{v\in V\,|\, \uqmu{\mu}{\g}\cdot v \neq V\} \text{ and } \rad W = \{v\in V\,|\,\uqmu{\nu}{\g_J}\cdot v\neq V\}$$
where we have identified the vector spaces $V$ and $W$ via the isomorphism $\res_J^{\mu}(V)\simeq W$.\par  We want to show that $\res_J^{\mu}(\rad V)=\rad W$ in $\OO^{\nu}_{J}$. For this, note that $\res_J^{\mu}(\rad V)$ clearly is a $\uqmu{\nu}{\g_J}$-submodule of $\rad W$. Suppose now that there exists $v \in \rad W$ with $v \not\in \res_J^{\mu}(\rad V)$ and take $x\in \uqmu{\mu}{\g}$ such that $w = xv$. Using \eqref{eq:TriangularDec} and \eqref{eq:QDeg}, we can suppose, without loss~of generality, that $v$ is a weight vector in $V$ and that\vspace*{-0.5mm} 
$$x = u_-u_0u_+$$\vspace*{-4.45mm}\\
for $Q$-homogeneous elements $u_-\in \uqmu{\mu,-}{\g}$, $u_0 \in \uqmu{\mu,0}{\g}$ and $u_+ \in \uqmu{\mu,+}{\g}$.\par 
Note that $\deg_Q u_- = 0$. Indeed, if this was not true, we would have $\deg_Q u_- < 0$ and hence $u_0u_+v\in V$ would be a weight vector with weight strictly 
greater than the highest $\ell$-weight~of $V$. This is impossible and we can thus suppose $u_-=1$ in full generality. Now, by \eqref{eq:QDeg}, $u_+v$ and $u_0u_+v = w$ live in the same weight space (which is of dimension 1). Therefore, $u_+v=\gamma w$ for some $\gamma\in\C^{\times}$ and, as $x_i^{+}(z)V = 0$ for $i\not\in J$,
 $$ w \in \uqmu{\mu,+}{\g}\cdot v = \uqmu{\nu,+}{\g_J}\cdot v \subseteq \uqmu{\nu}{\g_J}\cdot v, $$
but this implies $W=\uqmu{\nu}{\g_J}\cdot w \subseteq \uqmu{\nu}{\g_J}\cdot v$ which contradicts $v\in \rad W$. Hence, in $\OO^{\nu}_J$,
$$\res_J^{\mu}(\rad V)=\rad W$$ 
and the relation $x_i^{\pm}(z)\rad V \subseteq x_i^{\pm}(z)V =0$ for $i\not\in J$ implies that $\rad V$ is indeed a $J$-inflation of $\rad W$ to $\g$ with coweight $\mu$. It is then easy to end our proof as, by exactness of $\res_J^{\mu}$, 
$$ \res_J^{\mu}(\topp V) \simeq \res_J^{\mu} (V)/\res_J^{\mu}(\rad V) \simeq W/\rad W = \topp W $$ 
with $x_j^{\pm}(z) \topp V = 0$ trivially.
\end{proof}\newpage
\begin{example}\label{ex:QQSEC} Take $J=\{j\}$ for some $j\in I$ and fix $a\in \C^{\times}$. Then the quantum Wronskian relation \eqref{eq:QQWronsk} gives (with Theorem \ref{thm:HighestWeightFus}) that the object $W = L^J(\Psi_{j,a})\star L^J(\Psi_{j,a}^{-1})$ of $\OO^{sh}_J$ verifies
$$ \rad W = L^J(\Psi_{j,aq_j^2})\star L^J(\Psi_{j,aq_j^{-2}}^{-1})\text{ and } \topp W=L^J(\mathbbm{1}).$$\vspace*{-3mm}\\
Now, using the $Q\widetilde{Q}$-system \eqref{eq:QQ} and the object $V = L(\Psi_{j,a})\star L(\widetilde{\Psi}_{j,a})$ of $\OO^{sh}$, one easily gets
$$ \rad V = L(\Psi_{j,aq_j^2})\star L(\widetilde{\Psi}_{j,aq_j^{-2}})\text{ and }\topp V = L(\Psi_p)$$\vspace*{-3.5mm}\\
with, as in Example \ref{ex:QQ}, $ \Psi_p = \textstyle \Psi_{j,a}\widetilde{\Psi}_{j,a}$. This is compatible with Proposition \ref{prop:InfHighest} by Example \ref{ex:QQ} and Corollary \ref{cor:FusInfl}. Similarly, Proposition \ref{prop:InfHighest} agrees with Examples \ref{ex:QQ2} and \ref{ex:newT}.
\end{example}
\begin{theorem}\label{thm:InflIfPrefond} Suppose that the negative prefundamental representation $L_{j,1}^-$ of $\OO_J^{sh}$ admits a $J$-inflation to $\g$ for each $j\in J$. Then every simple module in $\OO_{J}^{sh}$ admits a $J$-inflation to~$\g$. 
\end{theorem}\vspace*{-2.1mm}
\begin{proof}
Fix a simple object $W$ in $\OO^{sh}_{J}$. Then, by Theorem \ref{thm:PrefundMonGen}, $W$ can be realized as the head of a fusion product $W_f$ of some invertible representation with (potentially many) prefundamental representations (positive and negative). By hypothesis, Example \ref{ex:InvertiblePositivePrefond} and Lemma \ref{lem:InflShift}, we can fix a $J$-inflation to $\g$ for all the factors appearing in this product $W_f$. Also, by Corollary \ref{cor:FusInfl} and Proposition \ref{prop:InfHighest}, the fusion product $V_f$ of the chosen $J$-inflations is a $J$-inflation of $W_f$ and the head of $V_f$ is a $J$-inflation of $W$. This ends the proof.
\end{proof}\vspace*{-1.5mm}
This leads us to our first existence result (see Theorem \ref{thm:main2}(ii)).
\begin{cor}\label{cor:ExistenceSL} Suppose that $\g_J$ is isomorphic (as a Lie algebra) to a direct sum of copies~of $\SL$ and $\SLt$. Then, every fusion product of simple objects in $\OO_{J}^{sh}$ has an inflation. In particular,  
$${\textstyle  \mathscr{R}_J:\OO^{sh} \arr \bigoplus_{\nu\in \Lambda^{\vee}_J}\uqmu{\nu}{\g_J}\Mod}$$ is essentially surjective on (fusion products of) simple objects of 
$\OO_{J}^{sh}$.
\end{cor}\vspace*{-2.1mm}
\begin{proof}
This stems from Examples \ref{ex:InflPrefondSL}, \ref{ex:InflPrefondSLt} and \ref{rem:InflPrefondSS} with Corollary \ref{cor:FusInfl} and Theorem \ref{thm:InflIfPrefond}. 
\end{proof}\vspace*{-1.5mm}
We finish this subsection with the following theorem which seems to be \textit{dual}, in some sense, to the above Proposition \ref{prop:InfHighest} (and is again compatible with Examples \ref{ex:QQ}, \ref{ex:QQ2} and \ref{ex:newT}).
\begin{theorem}\label{thm:InfHighestDual} Let $\nu\in\Lambda_J^{\vee}$ and fix objects $W$ and $W_S$ of $\OO^{\nu}_J$ related by a short exact sequence 
$$ 0 \arr W_S \arr W \arr W/W_S \arr 0$$
Fix now $\mu\in \Lambda^{\vee}$ such that $\res_J(\mu)=\nu$ and let $V$ be a $J$-inflation of $W$ to $\g$ with coweight $\mu$. Assume that $W_S$ is of highest $\ell$-weight.~Then there is a short exact sequence in $\OO^{\mu}$ of the~form
$$ 0 \arr V_S \arr V \arr V/V_S \arr 0$$
where $V_S$ and $V/V_S$ are respectively $J$-inflations of $W_S$ and $W/W_S$ to $\g$ (with coweight $\mu$).
\end{theorem}\vspace*{-2.12mm}
\begin{proof}
Take a highest $\ell$-weight vector $w\in W_S\subseteq W$ and let $V_S$ be the $\uqmu{\mu}{\g}$-submodule of $V$ generated by $w$ (with the spaces $V$ and $W$ identified again via the isomorphism $\res_{J}^{\mu}(V)\simeq W$). Remark that $x_i^{\pm}(z)V_S \subseteq x_i^{\pm}(z)V = 0$ for $i\not\in J$. Furthermore, the usual argument shows that $\res_{J}^{\mu}(V_S)\simeq W_S$ as $\uqmu{\nu}{\g_J}$-modules. Indeed, as $\res_{J}^{\mu}(V)\simeq W$ and $V_S\subseteq V$, there is an injective $\uqmu{\nu}{\g_J}$-morphism of $\res_{J}^{\mu}(V_S)$ into $W$. Our claim then follows from\vspace*{-0.9mm}
$$ V_S = \uqmu{\mu}{\g}\cdot w \supseteq \uqmu{\mu,-}{\g} \cdot w = \uqmu{\nu,-}{\g_J}\cdot w = \uqmu{\nu}{\g_J}\cdot w = W_S$$\vspace*{-4.75mm}\\
where we have used once again \eqref{eq:TriangularDecJ} with Definition \ref{def:Infl}. Now, by exactness of $\res_J^{\mu}$, 
$$ \res_J^{\mu}(V/V_S) \simeq \res_J^{\mu}(V)/\res_J^{\mu}(V_S) \simeq W/W_S$$
as $\uqmu{\nu}{\g_J}$-modules. This ends the proof as, trivially, $x_i^{\pm}(V/V_S) = 0$ for all $i\not\in J$.  
\end{proof}\newpage
\subsection{Inflations for negative prefundamental representations}\label{sec:InflPrefund} Fix $J\subseteq I$ and denote by $\mathcal{A}_J$ (resp.~$\mathcal{Y}_J$ and $\mathcal{Y}_J^+$) the set of monomials in the variables $\{A_{j,c}^{-1}\}_{j\in J,c\in \C^{\times}}$ (resp.~ $\{Y_{j,c}^{\pm 1}\}_{j\in J,c\in \C^{\times}}$ and $\{Y_{j,c}\}_{j\in J,c\in \C^{\times}}$). We will write $\mathcal{A}_I=\mathcal{A}$, $\mathcal{Y}_I=\mathcal{Y}$ with \vspace*{-0.35mm}
$$\mathcal{Y}_{\{i\}}^+ = \mathcal{Y}_i\text{ and }\res_{\{i\}} = \res_i:\mathfrak{r} \arr \mathfrak{r}^{\{i\}}$$ for $i\in I$. We will also use the following technical definition and theorem. Note that we~always suppose that monomials of $\mathcal{Y}$ are expressed in a \textit{reduced manner} (that is without possible~simplification of variables $Y_{i,c}$ and $Y_{i,c}^{-1}$ for $i\in I$ and $c\in \C^{\times}$).
\begin{defn}[\cite{fm}] Fix a monomial $M\in \mathcal{Y}$ and let, for $a\in \C^{\times}$, $$L(a,M) = \max\{\ell\in\Z\,|\,Y_{i,aq^{\ell}}\text{ or } Y_{i,aq^{\ell}}^{-1} \text{ appear in } M \text{ for some } i\in I\}.$$
Then $M$ is \textit{right-negative} if it does not contain the variables $Y_{i,aq^{L(a,M)}}$ for $i\in I$ and $a\in \C^{\times}$ (i.e.~if the variables $Y_{i,aq^{\ell}}^{\pm 1}$ appearing in it, with $\ell$ maximal, appear only with negative~powers).
\end{defn}
\begin{lem}[{\cite[Lemma 4.4]{hTsyst}}]\label{lem:HRightNeg} 
Fix $k\in \mathbb{N}_{>0}$, $i\in I$ and let $$m_k 
= Y_{i,q_i^{1-2k}}\dots Y_{i,q_i^{-1}}.$$ Take also a $\ell$-weight $\Psi\in \mathfrak{r}$ of the Kirillov--Reshitikhin module $W=L(m_k)$ such that~$\Psi\neq m_k$. Then, $\Psi = m_kA_{i,1}^{-1}M$ for a monomial $M\in \mathcal{A}$. Moreover, $\Psi$ is a right-negative element of~$\mathcal{Y}$.
\end{lem}
Recall that $\uqmu{0}{\g}$ is a central extension of the (untwisted) quantum loop algebra $\uqmu{}{\hat{\g}}$.~We have the following fundamental result due to Chari--Pressley (compare with Theorem \ref{thm:ClassSimpDimFin}).
\begin{thm}[{\cite[Theorem 3.3]{CP2}}\label{thm:SimpDimFinqLoop}
] Fix a finite-dimensional irreducible module of~the~quantum loop algebra $\uqmu{}{\hat{\g}}$ with highest $\ell$-weight $\Psi\in\mathfrak{r}_0$. Then $\Psi\in\mathcal{Y}^+$. 
\end{thm}
The following technical lemma is well known (see, e.g., \cite{fjmm} or \cite[Proposition 3.1]{young}).
\begin{lem}[{\cite[Lemma 5.5]{fjmm}}]\label{thm:YMY} Fix again a finite-dimensional simple $\uqmu{}{\hat{\g}}$-module~$V$. Fix also $i\in I$, $r\in \Z$, $\Psi\in \mathfrak{r}$ and $v\in V_{\Psi}$. Then there exist $a_1,\dots,a_R\in \C^{\times}$ such that 
$$\textstyle x_{i,r}^+v\in \bigoplus_{s=1}^R V_{\Psi A_{i,a_s}}.$$
\end{lem}\vspace*{-0.25mm}
Define now, for $i\not\in J$ and $V$ a simple object of $\OO^{sh}$ with highest $\ell$-weight $\Psi\in \mathfrak{r}$,
$$ \mathcal{P}_{i,J,V} = \{b\in \C^{\times}\,|\, V_{\Psi M A_{i,b}^{-1}}\neq 0 \text{ for some monomial } M \in \mathcal{A}_J\}.$$
\begin{lem}\label{lem:FinitePPrefond} Fix $j\in J$ and let $V = L(\Psi_{j,1}^{-1})$. Fix also $i\not\in J$. Then, $\mathcal{P}_{i,J,V}$ is a finite set. 
\end{lem}\vspace*{-2mm}
\begin{proof}
Fix $b\in \mathcal{P}_{i,J,V}$. Then, there is a monomial $M\in \mathcal{A}_J$ satisfying $V_{\Psi_{j,1}^{-1} M A_{i,b}^{-1}}\neq 0$ and\vspace*{-1.75mm}
\begin{center} 
$V_{\Psi_{j,1}^{-1} M'A_{i,b}^{-1}} = 0$ for every monomial $M'\in \mathcal{A}_J$ such that $\varpi(M')>\varpi(M)$.
\end{center}\vspace*{-1.25mm} 
Therefore, by Proposition \ref{prop:PrefondKR}, there exists $k\in\mathbb{N}_{>0}$ for which the Kirillov-Reshitikhin module $W = L(m_k)$ with $m_k 
= Y_{j,q_j^{1-2k}}\dots Y_{j,q_j^{-1}}$ verifies $W_{m_{k} M A_{i,b}^{-1}}\neq 0$ and\vspace*{-2.1mm}
\begin{center} $W_{m_{k} M'A_{i,b}^{-1}} = 0$ for every monomial $M'\in \mathcal{A}_J$ such that $\varpi(M')>\varpi(M).$
\end{center}\vspace*{-2.75mm} 
We claim that $\Psi = m_k MA_{i,b}^{-1}$ is \textit{$J$-dominant}, meaning that $\res_J(\Psi)\in \mathcal{Y}_J^+$. Indeed, take~$s\in J$, $r\in \Z$ and let $v\in  W_{\Psi}$ be a joint eigenvector for the action of the Cartan--Drinfeld subalgebra $\uqmu{0,0}{\g}$ of $\uqmu{0}{\g}$. 
Then, by Theorem \ref{thm:YMY} (and as $W$ is a finite-dimensional~irreducible $\uqmu{}{\hat{\g}}$-module by Example \ref{ex:KR}), 
\begin{equation}\label{eq:xDecPset}
\textstyle x_{s,r}^+v \in \bigoplus_{t=1}^R W_{\Psi A_{s,a_t}}
\end{equation}
for some $a_1,\dots,a_R\in \C^{\times}$. Fixing $1\leq t\leq R$ and using Remark \ref{rem:AAlgFree} as well as Theorem \ref{thm:qCharA}, we deduce that, without loss of generality, 
$$\Psi A_{s,a_t} = m_k M'A_{i,b}^{-1}$$ 
for a monomial $M'\in \mathcal{A}_J$ with $\varpi(M') = \varpi(M)[\alpha_s]>\varpi(M)$. 
Thus $W_{\Psi A_{s,a_t}}=0$ by the above and \eqref{eq:xDecPset} shows that $v$ is a highest $\ell$-weight vector for an irreducible subquotient~of~$\res_J(W)$. 
However, this subquotient is a finite-dimensional simple module for the quantum loop algebra $\uqmu{}{\widehat{\g_J}}$ (i.e.~for the quantum loop algebra associated to the semisimple Lie algebra $\g_J\subseteq \g$)~and its highest $\ell$-weight (i.e~$\res_J(\Psi)$ by construction) must therefore belong to $\mathcal{Y}_J^+$ by Theorem~\ref{thm:SimpDimFinqLoop} (see also \cite[Theorem 2.7]{hAff} for a version of this theorem in the case of quantum loop algebras with general semisimple underlying Lie algebras). This ends the proof of our claim.
\par
Now, note that the monomial $\Psi\in\mathcal{Y}$ trivially satisfies $\res_{I'}(\Psi)\in \mathcal{Y}_{I'}^+$ for $I'= I\backslash (J\cup \{i\})$~(as the expression $\Psi=m_kMA_{i,b}^{-1}$ does not involve variables of the form $A_{s,c}^{-1}$ with $s\in I'$, $c\in \C^{\times}$). Hence, by the above results, $\res_{I'\cup J}(\Psi)\in \mathcal{Y}_{I'\cup J}^+$ and $\Psi$ is \textit{$(I'\cup J)$-dominant}. However,~Lemma \ref{lem:HRightNeg}~shows that $\Psi$ is also a right-negative element of $\mathcal{Y}$ such that $\Psi\in m_kA_{i,1}^{-1}\mathcal{A}$. Let
$$ L = \max\{\ell\in \mathbb{Z}\,|\,Y_{s,q^{\ell}} \text{ or } Y_{s,q^{\ell}}^{-1} \text{ appears in } \Psi \text{ for some } s\in I \},$$
$$ L'= \max\{\ell\in \Z\,|\, A_{s,q^{\ell}}^{-1} \text{ appears in } MA_{i,b}^{-1} \text{ for some } s\in I\}. $$
Note that $L'\geq 0$ as $\Psi\in m_kA_{i,1}^{-1}\mathcal{A}$. \par 
We claim that $L \geq L'+ 1$. To show this, remark that $A_{s,q^{L'}}^{-1}$ appears in $MA_{i,b}^{-1}$ for some~$s\in I$ (by definition of $L'$). Also, by definition of the $\ell$-weight $A_s=A_{s,q^{L'}}$, the variable $Y_s^{-1}$ with 
$$Y_s= Y_{s,q^{L'+d_s}}$$
appears in $A_{s,q^{L'}}^{-1}$. We consider two cases.
\begin{itemize} 
\setlength{\itemsep}{1.5pt}
\item[(Case 1)] Suppose that $Y_s^{-1}$ also appears in (the \textit{reduced form} of) $\Psi$. Then, 
$L \geq L'+d_s > L'$. 
\item[(Case 2)] Assume now that $Y_s^{-1}$ does not appear in (the \textit{reduced form} of) $\Psi$ and remark that,~as $L'\geq 0$, $Y_s^{-1}$ appears in (the \textit{reduced form} of) $m_kA_{s}^{-1}$. Hence, to be cancelled out~in~$\Psi$, $Y_{s}$ must appear in~a variable $A_{r,c}^{-1}$ of the monomial $MA_{i,b}^{-1}\in \mathcal{A}$. In this case, 
$$C_{s,r}<0\text{ and }c = q^{\ell} \text{ for some }\ell\in\Z \text{ satisfying } \ell \geq L'+d_s+C_{s,r}+1 $$
and it follows in particular that $d_s+1 \leq {-}C_{s,r}$ (as $L'$ must be at least $\ell$ by definition). Thus, using the well known classification of connected (finite type) Dynkin diagrams, we deduce that $-C_{s,r}\in\{2,3\}$, $d_r > 1$ and $C_{r,p}\in\{0,-1\}$ for all $p\in I$. Now, $Y_r^{-1}$~with
$$ Y_{r} = Y_{r,q^{\ell+d_r}}$$
appears in $A_{r,c}^{-1}$. Let us suppose that $Y_r^{-1}$ does not appear in (the \textit{reduced form} of)~$\Psi$. Then, by the above reasoning (and as $\ell+d_r> C_{s,r}+3 \geq 0$), we can deduce that $Y_r^{-1}$ appears in some variable $A_{p,c'}^{-1}$ of the monomial $MA_{i,b}^{-1}$ with 
$$ C_{r,p} <0 \text{ and } c' = q^{\ell'} \text{ for some } \ell'\in \Z \text{ satisfying } \ell'\geq \ell+d_r +C_{r,p}+1.$$
This however contradicts the definition of $L'$ as we would then obtain $$\ell'\geq \ell+d_r \geq \ell+2 \geq L'+d_s+C_{s,r}+3 > L'.$$
Hence, $Y_r^{-1}$ appears in (the \textit{reduced form} of) $\Psi$ and the reasoning of Case~1 (with the above results) gives as claimed 
$L\geq \ell+d_r > L'$. 
\end{itemize}
In all cases, $L> L'\geq 0$. In addition, by right-negativity and $(I'\cup J)$-dominance of~$\Psi$,
\begin{equation}\label{eq:resiPreuveP}
\res_i(\Psi) = Y_{i,q^L}^{-1}M_Y
\end{equation}
for a monomial $M_Y$ in the variables $\{Y_{i,q^L}^{-1}\}\cup\{Y_{i,c}^{\pm 1}\}_{c\in \C^{\times}\backslash\{q^{\ell}\}_{\ell\geq L}}$. Now, as $\Psi=m_kMA_{i,b}^{-1}$, the definition of $A_{i,b}$ also gives\vspace*{-0.5mm}
$$ \res_i(\Psi) = \res_i(M)Y_{i,bq_i}^{-1}Y_{i,bq_i^{-1}}^{-1}$$
with $\res_i(M)\in \mathcal{Y}_i^+$. Thus, $q^L\in\{bq_i,bq_i^{-1}\}$ and $b \in \{q^{L+d_i},q^{L-d_i}\}\subseteq \{q^{\ell}\,|\,{-}d_i< \ell \leq L+d_i\}$~as $L> 0$. On the other hand, by\footnote{In the notation of \cite[Corollary 6.14]{fm}, the monomial $m_k$ has positive lattive support with base $q_j^{1-2k}$.} \cite[Corollary 6.14]{fm}, \eqref{eq:resiPreuveP} gives
$$L\leq r^{\vee}h^{\vee}-d_j$$
with $h^{\vee}$ the \textit{dual Coxeter number} of $\g$ and where $r^{\vee}$ is the maximal number of edges connecting two vertices in the associated Dynkin diagram (i.e.~the \textit{lacing number} of $\g$). Thus, the element $b\in \mathcal{P}_{i,J,V}$ actually belongs to the finite set $\mathcal{P}=\{q^{\ell}\,|\,-d_i<\ell\leq r^{\vee}h^{\vee}+d_i-d_j\}$ (which does~not depend on $b$ and $k$) and it follows that the whole of $\mathcal{P}_{i,J,V} $ is contained in $\mathcal{P}$. 
\end{proof}\vspace*{-1mm}  
The next result follows from Proposition \ref{prop:PrefondKR} and \cite[Theorem 3.10]{hAff2} (see also \cite[Remark 3.6]{hAff2}). We also include below it a direct corollary of\footnote{The results of \cite{young,fjmm} are stated for the category $\mathcal{O}_{\hat{\g}}$ associated to the unshifted quantum loop algebra $\uqmu{}{\hat{\g}}$, but their proofs hold equally well in the shifted setting (if one keeps track of the shift for the $\phi_{i,r}^-$'s).} \cite[Proposition 3.1]{young} (and \cite[Lemma 5.5]{fjmm}). Note that the latter result can be seen as a refinement of \cite[Proposition 4.10]{hshift}. 
\begin{lem}[\cite{hAff2}]\label{lem:multPsiboundedPrefondAB} 
Fix $\g$ of type A--B and $j\in I$. Then, $\dim (L(\Psi_{j,1}^{-1}))_{\Psi}\leq 1$ for all $\Psi\in \mathfrak{r}$. 
\end{lem}
\begin{lem}[\cite{young}]\label{lem:FJMM} Fix an object $V$ of $\OO^{sh}$ and let $\Psi\in \mathfrak{r}$ satisfy $V_{\Psi}\neq 0$. Fix $i\in I$. Consider the finite set $\mathcal{S}_{i,\Psi,V} = \{b\in \C^{\times}\,|\, V_{\Psi A_{i,b}^{-1}}\neq 0\}$. Then, 
$$ \textstyle \left(\prod_{b\in \mathcal{S}_{i,\Psi,V}} (1-bz)^{\dim V_{\Psi A_{i,b}^{-1}}+\dim V_{\Psi}-2}\right)x_{i}^{-}(z)V_{\Psi}=0. $$ 
\end{lem}
\begin{rem} The set $\mathcal{S}_{i,\Psi,V}$ of the above lemma is finite since, as $V$ belongs to $\OO^{sh}$,
$$
|\mathcal{S}_{i,\Psi,V}| \leq \sum_{b\in \mathcal{S}_{i,\Psi,V}} \dim V_{\Psi A_{i,b}^{-1}} \leq \dim V_{\varpi(\Psi)[-\alpha_i]} < \infty. $$
\end{rem}\vspace*{-0.5mm}
Here is now the main result of this subsection (with a proof inspired by the proof of \cite[Proposition 6.9]{fjmm}, see the end of Section \ref{sec:Inflqchar}).
\begin{thm}\label{thm:ExistenceTech} Fix $\nu \in \Lambda_J^{\vee}$ and let $W$ be a simple $\uqmu{\nu}{\g_J}$-module of highest $\ell$-weight $\Psi\in \mathfrak{r}^{J}_{\nu}$. Consider furthermore the simple $\uqmu{\nu}{\g}$-module $V'$ of highest $\ell$-weight $\Psi$ (with $\Psi$ seen in $\mathfrak{r}_{\nu}$~via the inclusion $\mathfrak{r}_{\nu}^J\subseteq \mathfrak{r}_{\nu}$). Suppose that $\mathcal{P}_{i,J,V'}$ is a finite set for all $i\not\in J$ and that 
\begin{equation}\label{eq:MaxFin}
\max_{M\in \mathcal{A}_J} (\dim V'_{\Psi M A_{i,b}^{-1}}+\dim V'_{\Psi M})<\infty
\end{equation}
for every $b\in \mathcal{P}_{i,J,V}$. 
Then $W$ admits a $J$-inflation to $\g$. 
\end{thm}\vspace*{-2mm}
\begin{proof}
Fix a highest $\ell$-weight vector $v\in V'$ and denote by $\langle v\rangle_J$ the $\uqmu{\nu}{\g_J}$-submodule of $V'$ generated by $v$. Then, $\langle v\rangle_J\simeq W$ as $\uqmu{\nu}{\g_J}$-modules by Lemma \ref{lem:FacRestSimple}. Consider the $\ell$-weight
$$ 
\Psi_p = \prod_{i\not\in J}\prod_{b\in \mathcal{P}_{i,J,V'}}\Psi_{i,b}^{n_{i,b}}$$
where $n_{i,b} = \max_{M\in \mathcal{A}_J}(\dim V'_{\Psi M A_{i,b}^{-1}}+\dim V'_{\Psi M}-2)\in \mathbb{N}$ (for all $i\not\in J$ and $b\in \mathcal{P}_{i,J,V'}$).\newpage 
By \eqref{eq:MaxFin} (and as $|\mathcal{P}_{i,J,V'}|<\infty$), $\Psi_p$ is a well-defined $\ell$-weight in $\mathfrak{r}$. Define a coweight~$\mu\in \Lambda^{\vee}$ such that $\res_J(\mu)=\nu$ by
$$\mu = \nu+\sum_{i\not\in J}\sum_{b\in \mathcal{P}_{i,J,V'}}n_{i,b}\,\omega_i^{\vee}$$
and let $ V_p = V'\otimes L(\Psi_p)$. By Remark \ref{rem:Dcoproduct}, there is a well-defined action of $\uqmu{\mu}{\g}$ on $V_p$ given by specializing the coproduct 
$\Delta_{\nu,\mu-\nu}^{(u)} : \uqmu{\mu}{\g} \arr (\uqmu{\nu}{\g} \otimes \uqmu{\mu-\nu}{\g})((u))$ at $u=1$. Let $$V = \langle v\rangle_J \otimes L(\Psi_p) \subseteq V_p.$$
We assert that $V$ is a $\uqmu{\mu}{\g}$-submodule of $V_p$. Indeed, take a non-zero vector $v_p$ inside $L(\Psi_p)$. Then, $L(\Psi_p)=\C v_p$ and $x_i^{\pm}(z)v_p = 0$ with $\phi_i^{\pm}(z)v_p = (\Psi_p)_i(z)v_p$ for all $i\in I$. Take $w\in\langle v\rangle_J$ a simultaneous eigenvector of the $\phi_{j,0}^+$'s with $j\in J$. By the reasoning of the proof of Lemma \ref{lem:FacRestSimple}, there is $\alpha\in Q_J^+$ such that $w$ lies in the weight space $V'_{\varpi(\Psi)[-\alpha]}$ of $V'$. Also, 
\vspace*{-0.3mm}
$$\textstyle w = \sum_{r=1}^R \gamma_r w_r $$
with $\gamma_1,\dots,\gamma_R\in\C^{\times}$ and where $w_1,\dots,w_R\in V'$ are $\ell$-weight vectors (for the $\uqmu{\nu}{\g}$-action) of weight $\varpi(\Psi)[-\alpha]$. Fix $1\leq r\leq R$ and let $\Psi_r$ be the $\ell$-weight of $w_r$. Then, $\varpi(\Psi_r)=\varpi(\Psi)[-\alpha]$ and Theorem \ref{thm:qCharA} implies that $\Psi_r=\Psi M$ for some $M\in \mathcal{A}_J$. Thus, by Lemma \ref{lem:FJMM}, for $i\not\in J$, 
$$\textstyle (\Psi_p)_i(z)x_i^-(z)w_r = (\prod_{b\in \mathcal{P}_{i,J,V'}}(1-bz)^{n_{i,b}})x_i^-(z)w_r = 0$$
and it follows from \eqref{eq:Dcop} that, in $V_p$,
$$ \textstyle x_i^{-}(z)\cdot (w_r\otimes v_p) = (x_i^-(z)w_r)\otimes (\phi_i^+(z)v_p) = ((\Psi_p)_i(z)x_i^-(z)w_r)\otimes v_p = 0.$$
Hence, $x_i^-(z)\cdot (w\otimes v_p) = 0$ for all $i\not\in J$. Furthermore, for any such $i$, $P(V')\subseteq D(\varpi(\Psi))$ and $\alpha-\alpha_i\not\in Q^+$ imply $x_{i,s}^+w\in V'_{\varpi(\Psi)[\alpha_i-\alpha]}=0$ for all $s\in \Z$ with $$x_i^+(z)\cdot (w\otimes v_p) = (x_i^+(z)w)\cdot v_p = 0.$$
Fix now $j\in J$ and note that, by $J$-triviality of $\Psi_p$ (see below Proposition \ref{prop:IfVInflqCar}), 
$$ x_{j}^{\pm}(z) \cdot (w\otimes v_p) = (x_j^{\pm}(z)w)\otimes v_p \text { and } \phi_j^{\pm}(z)\cdot (w\otimes v_p) = (\phi_j^{\pm}(z)w)\otimes v_p$$
so that, for $s\in \Z$,
\begin{equation}\label{eq:RestIPreuveEx}
x_{j,s}^{\pm} \cdot (w\otimes v_p) = (x_{j,s}^{\pm}w)\otimes v_p \text { and } \phi_{j,s}^{\pm}\cdot (w\otimes v_p) = (\phi_{j,s}^{\pm}w)\otimes v_p
\end{equation}
are both contained in $V = \langle v\rangle_J \otimes L(\Psi_p)$. \par In addition, for $i\not\in J$ and $s\in \Z$, there is $x\in \uqmu{\nu,-}{\g}$ with $\phi_{i,s}^{+}w = xv$. However, by~\eqref{eq:xWt},\vspace*{-0.65mm}
$$\phi_{i,s}^{+}w\in V'_{\varpi(\Psi)[-\alpha]}$$ \vspace*{-3.5mm}\\
and $\alpha\in Q_J^+$ implies that $x$ actually belongs to $\uqmu{\nu,-}{\g_J}\subseteq \uqmu{\nu,-}{\g}$. In particular, $\phi_{i,s}^+ w \in \langle v\rangle_J$ and, analogously, $\phi_{i,s}^- w \in \langle v\rangle_J$. Therefore,
$$ \phi_i^{\pm}(z)\cdot (w\otimes v_p) = (\phi^{\pm}_i(z) w)\otimes (\phi^{\pm}_i(z) v_p) = (\Psi_p)_i(z)(\phi^{\pm}_i(z) w)\otimes v_p\in V$$
and $V$ is, as claimed, a $\uqmu{\mu}{\g}$-submodule of $V_p$. Moreover, by the above results, for $i\not\in J$, $$x_i^{\pm}(z)\cdot V = 0$$ and $\res_J^{\mu}(V) \simeq  \langle v\rangle_J\simeq W$
in $\OO_{J}^{sh}$ by \eqref{eq:RestIPreuveEx}. This ends the proof of the theorem.
\end{proof}
\begin{thm}\label{cor:ABFin} Fix a simple object $W$ in $\OO^{sh}_{J}$. Then $W$ admits a $J$-inflation to $\g$ if either 
\begin{itemize}
\setlength{\itemsep}{1.5pt}
\item[(i)] $\dim W < \infty$ or if
\item[(ii)] $\g$ is of type A--B.
\end{itemize}
In particular, the functor
$$ \textstyle\mathscr{R}_J:\OO^{sh}\arr \bigoplus_{\nu\in\Lambda^{\vee}_J}\uqmu{\nu}{\g_J}\text{\normalfont--Mod}$$
is essentially surjective on finite-dimensional simple modules (and is essentially surjective on all simple modules of $\OO^{sh}_J$ if $\g$ has type A--B).
\end{thm}\vspace*{-2mm}
\begin{proof}
The result follows from Theorem \ref{thm:ExistenceTech} if $\dim W<\infty$. Moreover, for $\g$ of type A--B, the prefundamental representation $L(\Psi_{j,1}^{-1})$ of $\OO^{sh}_J$ admits a $J$-inflation for all $j\in J$ by Lemma \ref{lem:FinitePPrefond}, Lemma \ref{lem:multPsiboundedPrefondAB} and Theorem \ref{thm:ExistenceTech}. Thus, Theorem \ref{thm:InflIfPrefond} ends the proof in this case.
\end{proof}
\begin{rem} Theorem \ref{cor:ABFin} is the principal result of the present paper.\,Observe that it~implies that --- under condition (i) or condition (ii) --- every fusion product of simple objects of $\OO^{sh}_J$ admits a $J$-inflation to $\g$ (see Corollary \ref{cor:FusInfl}). Similarly, Proposition \ref{prop:InfHighest} (and Theorem~\ref{thm:InfHighestDual}) implies that Theorem \ref{cor:ABFin} can be extended to the Jacobson radical (and to any submodule~of highest $\ell$-weight) of fusion products of simple objects of $\OO^{sh}_J$. 
\end{rem}
We note that the only obstacle preventing us from generalizing part (ii) of Theorem \ref{cor:ABFin} to $\g$ of arbitrary type is the absence of a proven bound for the dimensions of particular $\ell$-weight spaces for 
prefundamental representations in type\footnote{Note that the existence of inflations for simple objects of $\OO^{sh}_J$ follows from Corollary \ref{cor:ExistenceSL} if $\g$ has type~$G_2$.} C--D--E--F; that is, of a proof that 
$$ \max_{M\in\mathcal{A}_J}(\dim V_{\Psi_{j,1}^{-1} MA_{i,b}^{-1}}+\dim V_{\Psi_{j,1}^{-1}M})<\infty$$
whenever $i\not\in J$, $b\in \C^{\times}$ and $V = L(\Psi_{j,1}^{-1})$ (for some $j\in J$). \par 
Computations in type $D_4$ give us confidence that a finite bound as above exists in all cases (even if it will not always be equal to $1$ or $2$ as when Lemma \ref{lem:multPsiboundedPrefondAB} applies). We formulate this in a proper conjecture. We wish to return to the proof of this conjecture in a near future.
\begin{conj}\label{conj:multPsiboundedPrefond} Fix $j\in J$ and let $V = L(\Psi_{j,1}^{-1})$. Then, for all $i\not\in J$ and all $b\in \C^{\times}$,
$$ \max_{M\in\mathcal{A}_J}(\dim V_{\Psi_{j,1}^{-1} MA_{i,b}^{-1}}+\dim V_{\Psi_{j,1}^{-1}M})<\infty.$$
\end{conj}
We end this section with the following result which is proven like Theorem \ref{cor:ABFin}.
\begin{cor}\label{thm:ExistenceIfConj} Suppose that Conjecture \ref{conj:multPsiboundedPrefond} holds for the Lie algebra $\g$. Then every~simple object of $\OO^{sh}_J$ admits a $J$-inflation to $\g$.
\end{cor}\vfill
\section{Concluding remarks}\label{sec:Concl} In this section, we describe further possible research avenues and potential applications~for the theory of inflations. We start by studying in what sense an inflation of a given object~$W$~of $\OO^{sh}_J$ can be unique and state a general conjecture about the relation between all inflations~of~$W$ when $W$ is irreducible. We also ask some questions about a possible alternative construction~of inflations (suggested by Joel Kamnitzer) --- which comes from the (partly~conjectural)~relation between Coulomb branches and \textit{truncated shifted quantum affine algebras} --- and give results linking inflations to ``\textit{compatible cluster structures}" over Grothendieck rings and \textit{$R$-matrices}.\par For all this section, we fix $J\subseteq I$ and assume that Conjecture \ref{conj:multPsiboundedPrefond} holds for our choice~of~Lie algebra $\g$. (One can assume $\g$ of type A--B in order to not rely on unproven results.)
\vspace*{-1.5mm}
\subsection{Unicity and isomorphisms between truncations}\label{sec:Unicity} By Corollary \ref{cor:InfTrivial}, the trivial representation $L^J(\mathbbm{1})$ admits, for any dominant coweight $\mu\in\Lambda^{\vee}_+$, infinitely many $J$-inflations~to $\g$ with coweight $\mu$ (up to isomorphism of $\uqmu{\mu}{\g}$-modules).\newpage  Moreover, if $W$ is a highest $\ell$-weight module in $\OO^{sh}_J$ with
\begin{itemize}
\setlength{\itemsep}{1.5pt}
\item[(1)] $X$ a $J$-inflation of $L^J(\mathbbm{1})$ to $\g$ with coweight $\mu_1\in \Lambda^{\vee}_+$ and 
\item[(2)] $V$ a $J$-inflation of $W$ to $\g$ with coweight $\mu_2\in \Lambda$,
\end{itemize}
then $V\star X$ is also a $J$-inflation of $W	\simeq W\star L^J(\mathbbm{1})$ (with coweight $\mu_1+\mu_2$) by Corollary~\ref{cor:FusInfl}. Inflations of highest $\ell$-weight objects are hence never uniquely determined. They are however ordered with respect to two distinct strict partial orders, as explained below.
\begin{defn} Fix a highest $\ell$-weight module $W$ in $\OO^{sh}_J$. Fix also $J$-inflations $V_1$ and $V_2$ of $W$ to $\g$ with coweight $\mu_1$ and $\mu_2$ (resp.). Then, we say that
\begin{itemize}
\setlength{\itemsep}{1.5pt}
\item[(i)] $V_1$ \textit{weakly precede} $V_2$ if $0\neq \mu_2-\mu_1 = \sum_{i\not\in J}a_i\omega_i^{\vee}$ for some $\{a_i\}_{i\not\in J}\subseteq \mathbb{N}$; and that 
\item[(ii)] $V_1$ \textit{strongly precede} $V_2$ if the highest $\ell$-weights $\Psi_1,\Psi_2$ of $V_1,V_2$ are such that $L(\Psi_1^{-1}\Psi_2)$ is a $J$-inflation of the trivial representation $L^J(\mathbbm{1})$ of $\OO^{sh}_J$, but $L(\Psi_2^{-1}\Psi_1)$ is not.
\end{itemize}
We write $V_1\prec_{\text{weak}}V_2$ (resp.~$V_1\prec_{\text{str}} V_2$) if $V_1$ weakly (resp.~strongly) precede $V_2$ and call $V_1$ \textit{weakly minimal} (resp.~\textit{strongly minimal}) if it is minimal for $\prec_{\text{weak}}$ (resp.~$\prec_{\text{str}}$). 
\end{defn}
Note that the strict partial orders $\prec_{weak}$ and $\prec_{\text{str}}$ can only give rise to finite chains between given inflations. Hence, Corollary \ref{thm:ExistenceIfConj} (and Conjecture \ref{conj:multPsiboundedPrefond} that we suppose here true)~shows that all simple modules of $\OO^{sh}_J$ admit both weakly and strongly minimal inflations. In addition, since $V_1\prec_{\text{str}}V_2$ implies $V_1\prec_{\text{weak}}V_2$, all weakly minimal inflations are also strongly minimal. 
\begin{lem}\label{lem:StrongPrecedeImpliesJinfTriv} Take a simple module $W$ of $\OO^{sh}_J$ with $V_1,V_2$ some $J$-inflations of $W$ to $\g$.~Then $V_1 \prec_{\normalfont\text{str}} V_2$ if and only if $V_1\prec_{\text{weak}} V_2$ and $V_2\simeq V_1\star X$ in $\OO^{sh}$ for some $J$-inflation $X$ of $L^J(\mathbbm{1})$.
\end{lem}\vspace*{-2mm}
\begin{proof}
Suppose that $V_1\prec_{\normalfont\text{str}} V_2$ and take $X= L(\Psi_1^{-1}\Psi_2)$ for $\Psi_1$ and $\Psi_2$ the highest $\ell$-weights of $V_1$ and $V_2$. Then, $V=V_1\star X$ is of highest $\ell$-weight $\Psi_2$ by Theorem \ref{thm:HighestWeightFus} and is a $J$-inflation of $W$ by hypothesis and Corollary \ref{cor:FusInfl}. In particular, $V$ is simple by Lemma \ref{lem:InflSimpleSimple} and satisfies $V\simeq V_2$ by Theorem \ref{thm:ClassSimp}. This shows sufficiency (since $V_1\prec_{\text{weak}} V_2$ trivially).~Necessity~follows in turn from Theorem \ref{thm:HighestWeightFus} and Corollary \ref{cor:InfTrivial}.
\end{proof}
\begin{prop}\label{prop:CharInflwithMin} Take irreducible modules $W$ in $\OO^{sh}_J$ and $V$ in $\OO^{sh}$. Then, $V$ is a $J$-inflation of $W$ if and only if $V\simeq V_{\normalfont\text{min}}\star X$ where 
\begin{itemize}
\setlength{\itemsep}{1.5pt}
\item[(1)] $V_{\normalfont{\text{min}}}$ is a strongly minimal $J$-inflation of $W$ and
\item[(2)] $X$ is a $J$-inflation of $L^J(\mathbbm{1})$.
\end{itemize}
\end{prop}\vspace*{-2.75mm}
\begin{proof}
Suppose that $V$ is a $J$-inflation of $W$. Then either $V$ is strongly minimal (and the~result follows from $V\simeq V_{\text{min}}\star X$ with $X=L(\mathbbm{1})$ and $V_{\text{min}}=V$) or there exists some strongly minimal $J$-inflation $V_{\text{min}}$ of $W$ verifying $V_{\text{min}}\prec_{\text{str}}V$. In this last case, $V\simeq V_{\text{min}}\star X$ for some~$J$-inflation $X$ of $L^J(\mathbbm{1})$ by Lemma \ref{lem:StrongPrecedeImpliesJinfTriv}. Necessity~follows directly from Corollary \ref{cor:FusInfl}.
\end{proof}\vspace*{-1.15mm}
We have the following conjecture and two easy corollaries.
\begin{conj}\label{conj:Unicity} Let $W$ be an irreducible object of $\OO_J^{sh}$ and let $V_1$ and $V_2$ be strongly minimal inflations of $W$. Then, there is a $J$-trivial $\gamma\in \mathfrak{t}^{\times}\subseteq\mathfrak{r}$ such that $V_1\simeq V_2\star L(\gamma)$. Hence, strongly minimal inflations are unique up to fusion product with $J$-trivial invertible representations.
\end{conj}
\begin{cor} Suppose that Conjecture \ref{conj:Unicity} holds and fix $W$ a simple module of $\OO^{sh}_J$ with $V$ some $J$-inflation of $W$. Then, $V$ is strongly minimal if and only if it is weakly minimal.
\end{cor}\vspace*{-2mm}
\begin{proof}
Suppose $V$ strongly minimal, but not weakly minimal. There is then a weakly minimal inflation $V'$ of $W$ for which $V'\prec_{\text{weak}} V$. However, $V'$ is in this case also strongly minimal and Conjecture \ref{conj:Unicity} implies $V'\simeq V\star L(\gamma)$ for a $J$-trivial $\gamma\in \mathfrak{t}^{\times}$. This contradicts $V'\prec_{\text{weak}} V$.
\end{proof}\newpage
\begin{cor} Suppose that Conjecture \ref{conj:Unicity} holds and fix a simple object $W$ of $\OO^{sh}_J$. Fix also $V_{\normalfont\text{min}}$ a strongly minimal $J$-inflation of $W$ and $V$ an object of $\OO^{sh}$. Then, 
\begin{center}
$V$ is a $J$-inflation of $W$ if and only if $V\simeq V_{\normalfont\text{min}}\star X$ for some $J$-inflation $X$ of $L^J(\mathbbm{1})$.
\end{center}
\end{cor}\vspace*{-2mm}
\begin{proof}
Assume that $V$ is a $J$-inflation of $W$. Then, Proposition \ref{prop:CharInflwithMin} gives us a strongly minimal $J$-inflation $V'$ of $W$ and a $J$-inflation $X$ of $L^J(\mathbbm{1})$ with $V\simeq V'\star X$. In addition, by Conjecture \ref{conj:Unicity}, $V'\simeq V_{\text{min}}\star L(\gamma)$ for some $J$-trivial $\gamma\in \mathfrak{t}^{\times}$. Take $X'=L(\gamma)\star X$. Then, $X'$ is a $J$-inflation of $L^J(\mathbbm{1})$ by Theorem \ref{thm:PrefundSimpFus}, Corollary \ref{cor:InfTrivial} and $J$-triviality of $\gamma$. Moreover,
$$\chi_q(V_{\text{min}})\chi_q(X')=\chi_q(V')\chi_q(X)=\chi_q(V)$$
so that $V_{\text{min}}\star X'\simeq V$ by injectivity of the $q$-character map and simplicity of $V$. The converse direction follows directly from Corollary \ref{cor:FusInfl}.
\end{proof}
\begin{rem} The observations made after \cite[Proposition 6.9]{fjmm} mention something similar to the partial order $\prec_{\text{str}}$, but it is unclear to us whether or not \cite{fjmm} thought about unicity of strongly minimal $J$-inflations and about Conjecture \ref{conj:Unicity}.
\end{rem}
We were recently told by Joel Kamnitzer of a potential alternative method for constructing inflations using the relation between $K$-theoretic Coulomb branches of $3d$ $\mathcal{N}=4$ SUSY quiver gauge theories and shifted quantum affine algebras. To describe this alternative method more precisely, fix $\mu\in \Lambda^{\vee}$ and consider the \textit{truncated shifted quantum affine algebras} $U_{q,\lambda}^{\mu,\mathcal{Z}}(\g)$ of~\cite[Section 8]{ft1} (see also \cite[Sections 9-10]{hshift}). These algebras, that are parametrized~by~a~dominant coweight $\lambda\in \Lambda_+^{\vee}$ and a subset $\mathcal{Z}\subseteq \C^{\times}$ (satisfying both technical conditions), can be defined as quotients of a slight extension 
of $\uqmu{\mu}{\g}$ (called \textit{adjoint version}). Moreover, fixing~$(\lambda,\mathcal{Z}$),~there is conjecturally an isomorphism (cf.~\cite[Conjecture 8.15]{ft1})
$$ U_{q,\lambda}^{\mu,\mathcal{Z}}(\g) \simeq \mathbf{A}_{\mu,\lambda}/\mathbf{I}_{\mathcal{Z}}$$
with $\mathbf{A}_{\mu,\lambda}$ some Coulomb branch and with $\mathbf{I}_{\mathcal{Z}}$ an ideal specializing the \textit{equivariant parameters} of the algebra $\mathbf{A}_{\mu,\lambda}$ to elements of the subset $\mathcal{Z}$ (see \cite[Section 10]{hshift} for details). Let us~denote by $\mathbf{G}_{\mu,\lambda}$ the \textit{gauge group} associated to the Coulomb branch $\mathbf{A}_{\mu,\lambda}$ and by $\mathbf{N}_{\mu,\lambda}$~the~corresponding $\mathbf{G}_{\mu,\lambda}$-module. Then, the key observation of Kamnitzer is that truncated shifted quantum affine algebras corresponding to the Lie algebra $\g$ and to the Lie subalgebra $\g_J$ can give rise~to Coulomb branches having exactly the same gauge group $\mathbf{G}_{\mu,\lambda}$ and representation $\mathbf{N}_{\mu,\lambda}$.~Hence, using the (partly conjectural) ``\textit{Coulomb branch interpretation}" of truncated shifted quantum affine algebras, one could potentially construct isomorphisms of the form 
\begin{equation}\label{eq:IsoTSQAA}
U_{q,\lambda_J}^{\nu,\mathcal{Z}_J}(\g_J)\simeq U_{q,\lambda}^{\mu,\mathcal{Z}}(\g)
\end{equation}
which would then induce morphisms from $\uqmu{\mu}{\g}$ to $U_{q,\lambda_J}^{\nu,\mathcal{Z}_J}(\g_J)$. Also, as any finite-dimensional simple object in $\OO^{sh}_J$ naturally is a module over some truncated shifted quantum affine algebra \cite[Theorem 12.8]{hshift},~the above morphisms allow to view these simple objects as objects~of~$\OO^{sh}$. \par It thus seems appropriate to ask the following questions:
\begin{question} Fix a finite-dimensional simple module $W$ of $\OO^{sh}_J$ and suppose that we can define a $\uqmu{\mu}{\g}$-action on $W$ like above (for some $\mu\in \Lambda^{\vee}$). Is the resulting module a $J$-inflation~of~$W$? Is this still true for $W$ an arbitrary object of $\OO^{sh}_J$ endowed with an action of a truncated shifted quantum affine algebra? What about infinite-dimensional simple objects of $\OO^{sh}_J$?
\end{question}
\begin{question} Can all inflations (at least of finite-dimensional simple objects) be constructed~as above? Does this lead to a proof of Conjecture~\ref{conj:Unicity} (at least in some cases)?
\end{question}
A positive answer to these questions would be remarkable as it could lead to a \textit{functorial~way of defining inflations for all modules in $\OO^{sh}_J$ over which there is an action of a truncated~shifted quantum affine algebra} (whereas our actual method does not seem~to~be functorial because~of non-unicity).\,Remark however that the morphisms \eqref{eq:IsoTSQAA} defined using Coulomb branches~are\footnote{These morphisms also rely on \cite[Conjecture 8.15]{ft1} that remains, to our best knowledge, unproven.} somewhat hard to describe 
 and that it is not yet known whether or not all infinite-dimensional simple objects of $\OO^{sh}$ are endowed with an action of a truncated algebra (as~expected from~\cite[Conjecture 12.2]{hshift} and shown in \cite{hz} for the related situation of shifted Yangians).\,Following~the approach advised by Kamnitzer thus requires further research 
 that we leave for a later project. \par
Another question for the future, that is connected to the previous discussion, is 
\begin{question} Can we generalize the concept of $J$-inflations to more general Coulomb branches? To shifted Yangians (see \cite{bfn,hz} for more details about these algebras)? 
\end{question}\vspace*{-2mm}
\subsection{Possible application to the study of cluster structures on Grothendieck rings}\label{sec:applications} 
The main utility of inflations is their ability to reduce problems in $\OO^{sh}$ to corresponding, often easier to handle, problems in $\OO^{sh}_J$. In other words, $J$-inflations to $\g$ create a framework~within $\OO^{sh}$ in which objects/quantities (i.e.~$q$-characters, fusion products, etc.) can be computed as~in $\OO^{sh}_J$. An extreme example of this is obtained when one takes $J=\{j\}$ for some $j\in I$. Indeed, in this case, as in Example \ref{ex:QQ}, for $\mu\in \Lambda^{\vee}$ with $\nu = \res_J(\mu)=\alpha_j(\mu)\omega_j^{\vee}$,
 $$\uqmu{\nu}{\g_J}\simeq U_{q_j}^{\nu}(\mathfrak{sl}_2)$$
and explicit $q$-character formulas for inflations of simple modules of $\OO^{sh}_J$ can often be deduced from Example \ref{ex:qCharEX} and results of \cite{cp1,hshift}. (In constrast, getting such explicit 
formulas for general simple objects of $\OO^{sh}$ --- even finite-dimensional ones --- is a notoriously difficult open problem in the representation theory of shifted/unshifted quantum affine algebras.)\par
Another possible use of inflations lies in the construction of \textit{$R$-matrices} for the category~$\OO^{sh}$. Indeed, let $J,J'\subseteq I$ satisfy $J\cap J'=\emptyset$ and fix $V$ a $J$-inflation of an object $W$ of $\OO^{sh}_J$ with $V'$ a $J'$-inflation of an object $W'$ of $\OO^{sh}_{J'}$. Assume that $V$ and $V'$ respectively belong to $\OO^{\mu}$~and~$\OO^{\mu'}$ for some $\mu,\mu'\in \Lambda^{\vee}$. Then, as $x_i^{\pm}(z)V = 0$ for all $i\not\in J$ with $x_{i'}^{\pm}(z)V' = 0$ for all $i'\not\in J'$, there is a well-defined $\uqmu{\mu+\mu'}{\g}$-action on $V\otimes V'$ given by specializing at $u=1$ the map
$$\Delta_{\mu,\mu'}^{(u)}:\uqmu{\mu+\mu'}{\g}\arr (\uqmu{\mu}{\g}\otimes\uqmu{\mu'}{\g})((u))$$ 
of Section \ref{sec:Fusion}. Let us denote $V\otimes_D V'$ the associated $\uqmu{\mu+\mu'}{\g}$-module. There is an analogous module $V'\otimes_D V$ defined on the product $V'\otimes V$ by specialization at $u=1$ of the coproduct
$$\Delta_{\mu',\mu}^{(u)}:\uqmu{\mu+\mu'}{\g}\arr (\uqmu{\mu'}{\g}\otimes\uqmu{\mu}{\g})((u)).$$
Recall that $V\otimes_D V'$ with $V'\otimes_D V$ need not to be highest $\ell$-weight modules if $V$ and $V'$ are irreducible (whereas in this case $V\star V'$ and $V'\star V$ are of highest $\ell$-weight by Theorem \ref{thm:HighestWeightFus}). There can hence be non-zero module morphisms 
$$ V\otimes_D V'\arr V'\otimes_D V$$
with non-trivial kernel. Morphisms of this form play an important role in the representation theory of Borel quantum loop algebras (see, e.g., \cite{hbourb,her,p}). In this setting, these morphisms usually come from \textit{meromorphic $R$-matrices} which are isomorphism of the form 
$$ V(a)\otimes V'\arr V'\otimes V(a)$$
with $a\in \C^{\times}$ a generic shift parameter (as in Lemma \ref{lem:InflShift}). Here is a question for the future: \newpage
\begin{question} Can we construct a meromorphic $R$-matrix in our context given an arbitrary pair $(V,V')$ with $V$ a $J$-inflation and $V'$ a $J'$-inflation as above?
\end{question}
\begin{example} Fix $\g = \SLt$ with $V=L(\widetilde{\Psi}_{1,1})$ and $V'= L(\widetilde{\Psi}_{2,1})$. Denote by $\{v_{\ell}\}_{\ell\geq 0}\subseteq V$ and $\{v_m\}_{m\geq 0}\subseteq V'$ the $\C$-bases given in \cite[Example 5.2(iv)]{hshift}. For $a\in \C^{\times}\backslash q^{2\Z+1}$, define a map
$$\psi_a:V(a)\otimes_D V'\arr V'\otimes_D V(a)$$ by $\psi_a(v_{\ell}\otimes v'_{m}) =\gamma_{\ell,m}v'_{m}\otimes v_{\ell}$ where\vspace*{-0.6mm} $$ \gamma_{\ell,m} = a^{\ell}q^{-\ell m}\frac{\left(\prod_{k=1}^{m}(1-aq^{2k-1})\right)}{\left(\prod_{s=1}^\ell (a-q^{2(s-m)-1})\right)}.$$
Then, the $\uqmu{\omega_2^{\vee}-\omega_1^{\vee}}{\g}$-action (resp.~$\uqmu{\omega_1^{\vee}-\omega_2^{\vee}}{\g}$-action) for $V$ (resp.~$V'$) given in \cite[Example 5.2(iv)]{hshift} makes it easy to prove that $\psi_a$ is a $\uqmu{0}{\g}$-linear isomorphism. Also, for $a = q^{-(2r+1)}$ with $r\in\mathbb{N}$, the same formula for $\psi_a$ give a non-zero $\uqmu{0}{\g}$-morphism with non-trivial kernel.
\end{example}
Let us finally discuss some results connecting inflations with the study of \textit{compatible~cluster structures} over Grothendieck rings. This is connected to the use of inflations for the deduction of new relations in the Grothendieck ring $K_0(\OO^{sh})$ (see Examples \ref{ex:QQ}, \ref{ex:QQ2} and \ref{ex:newT}), which is in itself an interesting (but still not totally understood) application of our results. We refer to \cite{fz1,fz2,kel12} for more informations about cluster algebras and to \cite{ghl,hl1,hl2,kkop} for the usual notion of \textit{monoidal categorification} of cluster algebras (with some examples).~The following definitions are inspired from \cite{ghl,hl2}. All our cluster algebras are over $\Z$.
\begin{defn} A simple object $V$ of $\OO^{sh}$ is \textit{real} if $V\star V$ is irreducible and \textit{prime} if, for any factorization $V\simeq L(\Psi_1)\star L(\Psi_2)$ with $\Psi_1,\Psi_2\in\mathfrak{r}$, either $\Psi_1\in \mathfrak{t}^{\times}$ or $\Psi_2\in \mathfrak{t}^{\times}$ (i.e.~at least one of $L(\Psi_1)$ and $L(\Psi_2)$ is an invertible representation).
\end{defn}
\begin{defn}\label{def:monCat} Fix a cluster algebra $\mathsf{A}$ with coefficients and let $\varphi:\mathsf{A}\arr K_0(\OO^{sh})$ be a ring isomorphism. Then $\varphi$ is a \textit{compatible $\mathsf{A}$-cluster structure on $\OO^{sh}$} if
\begin{itemize}
\item[(i)] $\varphi$ sends cluster monomials of $\mathsf{A}$ to isoclasses of real simple objects
and 
\item[(ii)] $\varphi$ sends cluster variables and coefficients 
of $\mathsf{A}$ to isoclasses of prime real simple objects.
\end{itemize}
\end{defn}
Here is the main result of this section.
\begin{thm}\label{thm:SminPrime} Fix a simple object $W$ of $\OO^{sh}_J$.
\begin{itemize}
\item[(i)] Suppose that $W$ is real. Then, any $J$-inflation of $W$ to $\g$ is real.
\item[(ii)] Suppose that $W$ is prime. Then, any strongly minimal $J$-inflation of $W$ to $\g$ is prime.
\end{itemize}
\end{thm}
To prove the above theorem, we will need the following result.
\begin{lem}\label{lem:SminLem} Fix a simple object $W$ in $\OO^{sh}_J$. Fix also a $J$-inflation $V$ of $W$ and suppose~that $V$ factorizes as $V\simeq V_1\star V_2$ for some simple objects $V_1$ and $V_2$ of $\OO^{sh}$. Then, for $i\not\in J$, 
$$x_i^{\pm}(z)V_1 = x_i^{\pm}(z)V_2=0.$$
\end{lem}
\vspace*{-2.1mm}
\begin{proof} Note $\Psi_1$ (resp.~$\Psi_2$) the highest $\ell$-weight of $V_1$ (resp.~$V_2$) and take $\Psi\in \mathfrak{r}$ with $(V_1)_{\Psi}\neq 0$. Then $V_{\Psi\Psi_2}\neq 0$ since $\chi_q(V)=\chi_q(V_1)\chi_q(V_2)$ and Theorem~\ref{thm:qCharAJ} with Theorem \ref{thm:EquivalentDefInf} show that there is 
$M\in \mathcal{A}_J$ for which \vspace*{-0.25mm}
$$ \Psi\Psi_2 = \Psi_1\Psi_2 M$$
\vspace*{-4.5mm}\\
(where we used Theorem \ref{thm:HighestWeightFus} and Lemma \ref{lem:InflSimpleSimple} to conclude that $V$ is simple of highest $\ell$-weight $\Psi_1\Psi_2$). Hence, $\varpi(\Psi) = \varpi(\Psi_1)[-\alpha]$ for some $\alpha\in Q_J^+$ and
\begin{equation}\label{eq:chiV1}
\chi(V_1) \subseteq \varpi(\Psi_1)[-Q_J^+].
\end{equation}
Let now $i\not\in J$ and $r\in \Z$. Fix a weight vector $v\in V_1$. Then, $v$ has weight $\varpi(\Psi_1)[-\alpha]$ for~some $\alpha\in Q_J^+$ by \eqref{eq:chiV1} and $\alpha\mp\alpha_i\not\in Q_J^+$ (with \eqref{eq:chiV1} again) implies $x_{i,r}^{\pm}v = 0$. This gives the desired result for $V_1$. Similarly, $x_i^{\pm}(z)V_2= 0$ for all $i\not\in J$ as wanted.
\end{proof}\vspace*{-2.85mm}
\begin{proof}[Proof of Theorem \ref{thm:SminPrime}.]
Let $V$ be a $J$-inflation of $W$ with coweight $\mu$. By Proposition \ref{prop:FusRest}, \vspace*{-0.25mm}
$$ \res_J^{2\mu}(V\star V)\simeq \res_J^{\mu}(V)\star \res_J^{\mu}(V) \simeq W\star W$$\vspace*{-4.5mm}\\
and $V$ is real if $W$ is real by exactness of the functor $\res_J^{2\mu}$. This ends the proof of part (i).~For part (ii), suppose that the $J$-inflation $V$ is strongly minimal, but not prime. Then,~$V\simeq V_1\star V_2$ for simple objects $V_1=L(\Psi_1)$ in $\OO^{\mu_1}$ and $V_2=L(\Psi_2)$ in $\OO^{\mu_2}$ with respective highest $\ell$-weights $\Psi_1,\Psi
_2\not\in\mathfrak{t}^{\times}$. In this case, $V_1$ (resp.~$V_2$) is a $J$-inflation of $\res_J^{\mu_1}(V_1)$ (resp.~$\res_J^{\mu_2}(V_2)$) by Lemma \ref{lem:SminLem}. In particular, by Proposition \ref{prop:FusRest},
$$ W \simeq \res_J^{\mu}(V) \simeq \res_J^{\mu_1}(V_1)\star \res_J^{\mu_2}(V_2)$$
with $\res_J^{\mu_1}(V_1) \simeq L^J(\res_J(\Psi_1))$ and $\res_J^{\mu_2}(V_2)\simeq L^J(\res_J(\Psi_2))$ (by Lemma \ref{lem:FacRestSimple} and irreducibility of $W$). Now, since $W$~is prime by hypothesis, 
$$\{\res_J(\Psi_1),\res_J(\Psi_2)\}\cap \mathfrak{t}^{\times}_J\neq\emptyset.$$
Suppose $\res_J(\Psi_1)=\gamma\in \mathfrak{t}_J^{\times}$ and set $V_1'=L(\gamma)$ (with $\gamma$ seen inside $\mathfrak{t}^{\times}$). Then, since $V_1=L(\Psi_1)$ is a $J$-inflation of $\res_J^{\mu}(V_1)\simeq
L^J(\gamma)$, Proposition \ref{prop:IfVInflqCar} gives $\Psi_1 = \gamma \Psi_p$ for $\Psi_p\in\mathfrak{r}$ a product of (possibly many) $\Psi_{i,a}$ and $[a\omega_i]$ with $i\not\in J$ and $a\in \C^{\times}$. Moreover, $\Psi_p\in \mathfrak{r}_{\mu'}$ for some non-zero $\mu'\in\Lambda^{\vee}_+$ as $\Psi_1\not\in \mathfrak{t}^{\times}$. Therefore, by Corollary \ref{cor:InfTrivial},
$$ L(\Psi_p) = L(\Psi_1\Psi_2(\gamma\Psi_2)^{-1}) \text{ is a }J\text{-inflation of } L^J(\mathbbm{1})\text{, but }L(\Psi_p^{-1})=L(\gamma\Psi_2(\Psi_1\Psi_2)^{-1}) \text{ is not} $$ 
and Corollary \ref{cor:FusInfl} easily implies that $V'= V_1'\star V_2
$ is a $J$-inflation of 
$$L^J(\gamma)\star \res_J^{\mu_2}(V_2) \simeq \res_J^{\mu_1}(V_1) \star \res_J^{\mu_2}(V_2) \simeq \res_J^{\mu}(V) \simeq W$$
such that $V'\prec_{\text{str}}V$. This however contradicts the strong minimality of $V$ and ends the proof (as the case where $\res_J(\Psi_2)\in \mathfrak{t}^{\times}$ leads to a similar contradiction by a similar reasoning). 
\end{proof}\vspace*{-1.1mm}
Assume that $\varphi:\mathsf{A}_J\arr K_0(\OO^{sh}_J)$ is a compatible cluster structure for some cluster algebra~$\mathsf{A}_J$ with coefficients. For each cluster variable or coefficient $x$ of $\mathsf{A}_J$, fix $\Psi_x \in \mathfrak{r}^J$ such that $$\varphi(x) = [L^J(\Psi_x)]$$ and take some strongly minimal $J$-inflation $V_x$ of $L^J(\Psi_x)$ to $\g$. Then, by Definition \ref{def:monCat} and Theorem \ref{thm:SminPrime}, $L^J(\Psi_x)$ and $V_x$ are respectively prime real simple modules in $\OO^{sh}_J$ and $\OO^{sh}$. In addition, given cluster variables and coefficients $x_1,\dots,x_r$ belonging to a common cluster of $\mathsf{A}_J$, Corollary \ref{cor:FusInfl} shows that
$$ V_{x_1\dots x_r}= V_{x_1}\star \dots \star V_{x_r}$$
is a $J$-inflation of $L^J(\Psi_{x_1})\star \dots\star L^J(\Psi_{x_r})$. Hence, $V_{x_1\dots x_r}$ is a real simple object of $\OO^{sh}$ (again by Definition \ref{def:monCat} and Theorem \ref{thm:SminPrime}) and it follows that $\varphi$ induces a correspondence between cluster monomials (or cluster variables/coefficients) of $\mathsf{A}_J$ and a subset of real simple modules (resp.~prime real simple modules) of $\OO^{sh}$. However, this correspondence rarely induces a ring morphism $\mathsf{A}_J\arr K_0(\OO^{sh})$ and it is typically necessary to add coefficients to the cluster algebra $\mathsf{A}_J$ to deduce a compatible cluster structure on $\OO^{sh}$. Let us illustrate this with an example.
\begin{example}\label{ex:Cluster} Fix $\g=\SLt$ and $J=\{1\}$. Then, \cite{hl1} gives a compatible cluster structure on $\OO^{sh}_J$ (in fact a \textit{monoidal categorification} over $\OO_{\hat{\g}}$) for which the initial seed is given by
\begin{center}
\begin{tikzpicture}
\node (A) at (0,0) {$[L^J(Y_{1,1})]$};
\node (Bl) at (1.45,0) {\ };
\node (B) at (2.75,0) [rectangle,draw] {$[L^J(Y_{1,1}Y_{1,q^2})]$};
\draw[->] (A) -- (Bl);
\end{tikzpicture}
\end{center}
(where we used a box to indicate the framed node). Indeed, mutating at the only mutable node and using \eqref{eq:Tsyst} (with $k=1$ and $a=q^3$) gives the seed
\begin{center}
\begin{tikzpicture}
\node (A) at (-0.07,0) {$[L^J(Y_{1,q^2})]$};
\node (Bl) at (1.45,0) {\ };
\node (B) at (2.75,0) [rectangle,draw] {$[L^J(Y_{1,1}Y_{1,q^2})]$};
\draw[->] (Bl) -- (A);
\end{tikzpicture}
\end{center}
and this suffices to specify a cluster structure of cluster type A${}_1$. Remark now that $L(Y_{1,1}\Psi_{2,q^2})$, $L(Y_{1,q^2}\Psi_{2,q^4})$ and $L(Y_{1,1}Y_{1,q^2}\Psi_{2,q^4})$ are respectively $J$-inflations to $\g$ of $L^J(Y_{1,1})$, $L^J(Y_{1,q^2})$ and $L^J(Y_{1,1}Y_{1,q^2})$ by Example \ref{ex:newT}. These inflations are moreover easily shown to be weakly (and thus also strongly) minimal. However, by \eqref{eq:newTsyst} (with $k=1$ and $a = q^3$), 
\begin{align*}
[L(Y_{1,1}\Psi_{2,q^2})][L(Y_{1,q^2}\Psi_{2,q^4})] &= [L(Y_{1,1}Y_{1,q^2}\Psi_{2,q^4})][L(\Psi_{2,q^2})]+[\omega_2][L(\Psi_{2,1})][L(\Psi_{2,q^4})]\\
&\neq [L(Y_{1,1}Y_{1,q^2}\Psi_{2,q^4})]+1
\end{align*}
and just replacing the cluster variables and coefficients of $\mathsf{A}_J$ by their chosen strongly minimal $J$-inflations does not give here a ring morphism $\mathsf{A}_J\arr K_0(\OO^{sh})$. Nevertheless, considering~the enhanced initial seed (i.e.~with more coefficients)
\begin{center}
\begin{tikzpicture}
\node (A) at (0,0) {$[L(Y_{1,1}\Psi_{2,q^2})]$};
\node (Bl) at (2.1,0) {\ };
\node (B) at (3.75,0) [rectangle,draw] {$[L(Y_{1,1}Y_{1,q^2}\Psi_{2,q^4})]$};
\node (C) at (-3,0) [rectangle,draw] {$[L(\Psi_{2,q^2})]$};
\node (Cr) at (-2,0) {\ };
\node (D) at (0,-1.11) [rectangle,draw] {$[L(\Psi_{2,q^4})]$};
\node (Du) at (0,-0.82) {\ };
\node (E) at (0,1.11) [rectangle,draw] {$[\omega_2][L(\Psi_{2,1})]$};
\node (Ed) at (0,0.82) {\ };
\draw[->] (A) -- (Bl);
\draw[->] (Du) -- (A);
\draw[->] (A) -- (Cr);
\draw[->] (Ed) -- (A);
\end{tikzpicture} 
\end{center}
gives, by \eqref{eq:newTsyst}, a well-defined compatible cluster structure on $\OO^{sh}$ with mutated seed 
\begin{center}
\begin{tikzpicture}
\node (A) at (0,0) {$[L(Y_{1,q^2}\Psi_{2,q^4})]$};
\node (Bl) at (2.1,0) {\ };
\node (B) at (3.75,0) [rectangle,draw] {$[L(Y_{1,1}Y_{1,q^2}\Psi_{2,q^4})]$};
\node (C) at (-3,0) [rectangle,draw] {$[L(\Psi_{2,q^2})]$};
\node (Cr) at (-2,0) {\ };
\node (D) at (0,-1.11) [rectangle,draw] {$[L(\Psi_{2,q^4})]$};
\node (Du) at (0,-0.82) {\ };
\node (E) at (0,1.11) [rectangle,draw] {$[\omega_2][L(\Psi_{2,1})]$};
\node (Ed) at (0,0.82) {\ };
\draw[->] (Bl) -- (A);
\draw[->] (A) -- (Du);
\draw[->] (Cr) -- (A);
\draw[->] (A) -- (Ed);
\end{tikzpicture} 
\end{center}
Observe that the classes added above --- that is $[\omega_2][L(\Psi_{2,1})]$, $[L(\Psi_{2,q^2})]$ and $[L(\Psi_{2,q^4})]$ --- are all easily seen to be isoclasses of prime real simple modules in $\OO^{sh}$. 
\end{example}
In general, an \textit{exchange (or mutation) relation} for the cluster algebra $\mathsf{A}_J$ looks like
\begin{equation}\label{eq:ExchRelex} 
x_ix_i'= x_{\text{in}} + x_{\text{out}}
\end{equation}
where $x_i$ and $x_i'$ are cluster variables and $x_{\text{in}}$ and $x_{\text{out}}$ are cluster monomials belonging to the same cluster as $x_i$. Using the compatible cluster structure $\varphi : \mathsf{A}_J\arr K_0(\OO^{sh}_J)$, Relation~\eqref{eq:ExchRelex} can be interpreted as the decomposition of a fusion product of two prime real simple objects~of $\OO^{sh}_J$ in real simple composition factors. There is thus an underlying short exact sequence
$$ 0\arr \rad W \arr W\arr \topp W\arr 0$$
with $\{[\rad W],[\topp W]\} = \{\varphi(x_{\text{in}}),\varphi(x_{\text{out}})\}$ and $W = L^J(\Psi_{x_{i}})\star L^J(\Psi_{x_{i}'})$ (where~we~use~the~notation introduced above Example \ref{ex:Cluster}).  By Corollary \ref{cor:FusInfl} and Proposition \ref{prop:InfHighest}, there is, in $\OO^{sh}$, a corresponding short exact sequence
$$ 0\arr \rad V \arr V \arr \topp V \arr 0$$
where $V = V_{x_i}\star V_{x_i'}$ and where $\rad V$ and $\topp W$ are $J$-inflations of $\rad W$ and $\topp W$ (resp.). \newpage This gives an analog of \eqref{eq:ExchRelex} in $K_0(\OO^{sh})$. However, as seen in the example above, writing $$\topp W = L^J(\Psi_{y_1})\star \dots \star L^J(\Psi_{y_r})$$
for some cluster variables and coefficients $y_1,\dots,y_r\in \mathsf{A}_J$, then, typically,
$$ \topp V \not\simeq  V_{y_1}\star \dots\star V_{y_r}$$
and it may be needed to add coefficients to $\mathsf{A}_J$ to obtain a compatible cluster structure on~$\OO^{sh}$. This seems somewhat difficult to do in general and is left for further work. 

\let\oldaddcontentsline\addcontentsline
\renewcommand{\addcontentsline}[3]{}

 \let\addcontentsline\oldaddcontentsline

\end{document}